\documentclass[final,leqno]{siamltex}
\usepackage{amsmath}
\usepackage{amssymb}
\usepackage{cases}
\usepackage{enumerate}
\usepackage{graphicx}
\usepackage{indentfirst}
\usepackage{doi}
\usepackage{cleveref}
\usepackage[usenames]{color}
\usepackage{wrapfig}
\usepackage{stmaryrd}
\usepackage{array} % <- Preamble
\usepackage[caption=false]{subfig}
\usepackage{booktabs}
\SetSymbolFont{stmry}{bold}{U}{stmry}{m}{n}

\usepackage{multirow}
\usepackage{mathrsfs}
\usepackage[normalem]{ulem}
\usepackage{enumitem}

\newtheorem{thm}{Theorem}[section]
\newtheorem{defi}{Definition}[section]

\newtheorem{rem}[thm]{Remark}

\renewcommand{\theequation}{\arabic{section}.\arabic{equation}}

\newcommand{\norm}[1]{\Vert#1\Vert}

\newcommand{\bR}{{\mathbb R}}
\newcommand{\mT}{\mathscr{T}}
\newcommand{\mQ}{\mathcal{Q}}

\newcommand{\Gt}{{\Gamma(t)}}
\newcommand{\Gm}{{\Gamma^m}}

\newcommand{\bV}{\mathbb{V}}

\newcommand{\vol}{\operatorname{vol}}
\newcommand{\dH}{{\rm d}\mathcal{H}}

\newcommand{\mH}{\mathcal{H}}
\newcommand{\rd}{\;{\rm d}}
\newcommand{\id}{{\rm id}}
\newcommand{\dd}[1]{\frac{\rm d}{{\rm d}#1}}
\newcommand{\ddt}{\dd{t}}
\newcommand{\nn}{\nonumber}

\newcommand{\ipd}[1]{\bigl(#1\bigr)}
\newcommand{\nabs}{\nabla_{\!s}}

\graphicspath{{./figures/}{figures/}}

\newcommand{\Id}{{I\!d}}

\title{A minimizing-movement framework for geometric gradient flows with admissible
tangential motion}

\author{Xiaoxiao Liu\thanks{School of Mathematical Sciences, University of Science and
Technology of China, 230026 Hefei, Anhui, China (xxl0226@mail.ustc.edu.cn)}
\and Quan Zhao\thanks{School of Mathematical Sciences, University of Science and
Technology of China, 230026 Hefei, Anhui, China (quanzhao@ustc.edu.cn)}
}

\date{}
\begin{document}

\maketitle
%%%%% Begin Abstract %%%%%%%%%%%

\begin{abstract}
We develop a minimizing-movement framework for parametric finite element approximations
of geometric gradient flows with admissible tangential motion. At each time step, the
discrete variational problem combines a metric dissipation term for the normal
displacement with a surface Dirichlet energy. The metric determines the normal
geometric evolution: the $L^2(\Gamma)$ metric gives mean curvature flow, while the
$H^{-1}(\Gamma)$ metric gives surface diffusion flow. Tangential velocity is selected
independently through weak constraints on the deformation map. The central structural
condition is admissibility, namely, that the identity map satisfies the constraint. This
condition keeps the identity map available as a comparison function and yields the
natural stability estimate. The framework recovers the classical
Barrett--Garcke--N\"urnberg (BGN) scheme from the unconstrained formulation and the
dual minimal-deformation-rate (MDR) scheme from the MDR constraint. We further
introduce two new admissible variants: an admissible BGN scheme and a relaxed MDR
scheme. For the resulting fully discrete schemes, we prove existence and uniqueness
under natural nondegeneracy assumptions and establish unconditional energy stability.
Numerical experiments compare the admissible and classical schemes and illustrate their
stability properties and mesh-quality behavior.

\end{abstract}
%%%%% end %%%%%%%%%%%

%%%%% Keywords %%%%%%%%%%%



\pagestyle{myheadings} \markboth{X. Liu and Q.~Zhao}
{A minimizing-movement framework for geometric flows}

%=============================================Introduction===============================================
%=====================================================================================================
\section{Introduction} \label{sec:intro}

Geometric gradient flows provide a variational framework for interface evolution in many
physical systems and arise naturally in applications ranging from grain growth and
biomembrane dynamics to thin-film technology. For such applications, see,
e.g.,~\cite{Mullins57,Helfrich73elastic,cahn1994, Thompson12solid}. Typical examples
include mean curvature flow and surface diffusion flow, where the interface motion is
governed by a normal velocity law involving curvature quantities. Let
$\{\Gt\}_{t\geq0}\subset\bR^d$, $d\in\{2,3\}$, be a family of smooth, oriented closed
hypersurfaces. The corresponding normal velocity laws are
\begin{equation}\label{eq:norvel}
\mathcal V =\left\{\begin{array}{cl}
\varkappa\quad &\mbox{for mean curvature flow},\\[0.3em]
-\Delta_s\varkappa\quad &\mbox{for surface diffusion flow},
\end{array}
\right.
\end{equation}
where $\mathcal V$ denotes the normal velocity of $\Gt$ in the direction of the unit
normal $\vec\nu$, $\varkappa$ denotes the scalar mean curvature, and $\Delta_s$ is the
Laplace--Beltrami operator on $\Gt$. We denote by $|\Gt|$ the surface area of $\Gt$.
Then the two velocity laws in~\eqref{eq:norvel} are, respectively, the $L^2$- and
$H^{-1}$-gradient flows of the surface area. Consequently, they satisfy the
energy-dissipation identity
\begin{equation}\label{eq:dissplaw}
\ddt|\Gt|= -\norm{\mathcal V}_X^2 =\left\{\begin{array}{ll}
-\int_{\Gt}\varkappa^2\dH^{d-1}\quad &\mbox{for mean curvature flow},\\[0.4em]
-\int_{\Gt}|\nabs\varkappa|^2\dH^{d-1}\quad &\mbox{for surface diffusion flow},
\end{array}
\right.
\end{equation}
where $\norm{\cdot}_X$ denotes the natural metric norm, with $X=L^2(\Gt)$ for mean
curvature flow and $X=H^{-1}(\Gt)$ for surface diffusion flow. Here $\nabs$ is the
surface gradient, and $\dH^{d-1}$ denotes integration with respect to the
$(d-1)$-dimensional Hausdorff measure in
$\bR^d$.

Over the past several decades, many numerical methods have been developed for the two
gradient flows in~\eqref{eq:norvel}. Among them, parametric finite element methods,
abbreviated here as parametric FEMs, have attracted particular attention;
see~\cite{Dziuk91,Deckelnick2005,Barrett20}. Much effort has therefore been devoted to
the design of structure-preserving parametric FEMs that inherit the gradient-flow
structure at the fully discrete level, especially the energy decay property. Early
examples include Dziuk's method for mean curvature flow~\cite{Dziuk91} and the mixed
parametric FEM for surface diffusion flow~\cite{BMN05}. In both approaches, however, the
surface is evolved by a purely normal motion, without additional tangential velocity,
which can lead to severe degeneration of the polyhedral mesh. Indeed, the geometric
evolution is determined solely by the normal velocity at the continuous level, whereas
tangential velocity changes the parametrization and therefore plays a crucial role in
surface mesh quality at the discrete
level~\cite{Remacle10,Mikula14,DeTurck17,Duan24new,PAN26}. Incorporating such tangential
motion without destroying energy stability is a delicate issue.

This difficulty has motivated several energy-stable parametric FEMs with carefully
designed tangential motion. The schemes of Barrett, Garcke, and N\"urnberg (BGN) rely on
the curvature identity and provide an elegant class of energy-stable methods for
curvature-driven flows~\cite{BGN07,BGN08parametric,BGN08willmore}. The BGN scheme
induces a favorable built-in tangential velocity that, under semidiscretization, yields
so-called conformal polyhedral surfaces in $\bR^3$. The energy-stable minimal
deformation rate (MDR) schemes~\cite{Gao26,Gao26dualMDR} were later proposed to control
the undesirable mesh instabilities observed in BGN schemes for small time-step sizes.
The MDR idea was first explored in the context of a convergence analysis of mean
curvature flow for surfaces~\cite{Hu22evolving}. Other related developments can be
found, for example, in~\cite{Duan25,Kemmochi25structure,GJSZ25,LX25}.

Despite these advances, existing energy-stable parametric FEMs are still largely
constructed case by case, and their stability estimates are typically tied to specific
choices of tangential velocity. A general structural condition that allows tangential
motion to be incorporated without destroying the energy estimate remains unclear. This
motivates a unified variational framework for designing energy-stable parametric FEMs
with different tangential constraints.

One motivation for the present work is the recent normal--tangential velocity-splitting
technique in~\cite{GNZ25willmore, GNZ26}, where an energy-stable parametric FEM for
Willmore flow was proposed while allowing the tangential velocity to be chosen
independently. In their formulation, separate equations are used to approximate the
gradient-flow structure and to impose the tangential velocity. A second motivation
comes from minimizing-movement frameworks for Willmore
flow~\cite{BalzaniR12,OlischlagerR09,RumpfSS25preprint}, which provide energy stability
through variational time discretization. In those works, however, the metric dissipation
is usually based on the full surface velocity. Motivated by these ideas, we develop a
minimizing-movement framework in which the metric dissipation determines only the normal
velocity, while tangential motion is introduced through a weak constraint on the
deformation map. At the level of the constrained minimizing problem, the key requirement
for energy stability is admissibility: the identity map must satisfy the constraint.
This condition keeps the identity map available as a comparison map and thereby yields
the natural energy estimate.

Within this framework, the classical BGN scheme in~\cite{BGN08parametric} is recovered
from the unconstrained formulation, while the dual-MDR scheme in~\cite{Gao26dualMDR} is
recovered from the MDR constraint. We then introduce two new admissible constructions.
The first is an admissible BGN-type scheme based on a mass-lumped vector curvature and a
corresponding modified vertex normal, which restores admissibility of the discrete BGN
constraint. The second is a relaxed MDR scheme, where a relaxed tangential projection
term is added to the MDR constraint. For both schemes, we prove existence and uniqueness
under mild nondegeneracy assumptions and establish unconditional energy stability.

The remainder of the paper is organized as follows. In~\S\ref{sec:mm-classical}, we
formulate the constrained minimizing-movement framework and identify the admissibility
condition that yields the basic energy estimate. We also derive the corresponding
Euler--Lagrange systems and show how the unconstrained BGN formulation and the MDR
constraint fit into this setting. In~\S\ref{sec:pfem}, we introduce the parametric
finite element discretizations. Section~\ref{sec:mm-admissible} contains the main fully
discrete constructions, namely the admissible BGN scheme and the relaxed MDR scheme,
together with their solvability and unconditional stability results.
In~\S\ref{sec:further-insights}, we discuss possible extensions of the framework.
Numerical results are presented in~\S\ref{sec:numerics}, and concluding remarks are
given in~\S\ref{sec:conclusions}.

\section{Minimizing-movement framework}
\label{sec:mm-classical}

Let $\Gamma\subset\bR^d$ be a smooth, orientable hypersurface without boundary. The
standard $L^2$ and $H^1$ spaces on $\Gamma$ are denoted by
$L^2(\Gamma)$ and $H^1(\Gamma)$.  For any Banach space $X$, we denote its dual space by
$X^\prime$, and the associated duality pairing by
$\langle\cdot,\cdot\rangle_{X^\prime, X}$.  If $X$ is a Hilbert space, its inner
product is denoted by $\ipd{\cdot,\cdot}_X$.

For $X=L^2(\Gamma)$ or $H^{-1}(\Gamma)$, we consider the functional
\begin{equation}
J:[H^1(\Gamma)]^d\to\bR,\qquad
J(\vec u) :=\frac{1}{2\tau}\norm{(\vec u - \vec\id|_{\Gamma})\cdot\vec\nu}^2_{X}
+ \frac{1}{2}\int_{\Gamma}|\nabs\vec u|^2\dH^{d-1},
\end{equation}
where $\norm{\cdot}_{X}$ is the norm induced by the inner product
$\ipd{\cdot, \cdot}_X$, $\tau>0$ is the time step
size, $\vec\nu$ is the outward unit normal to $\Gamma$, and $\vec\id$ is the identity
map in $\bR^d$.

We do not use $|\vec u(\Gamma)|$ directly, because this quantity is highly nonlinear as
a function of the parametrization $\vec u$. Instead, we use its Dirichlet-energy
approximation around the identity; see~\cite[(2.1)]{BalzaniR12}:
\begin{equation*}
|\vec u(\Gamma)|
= \frac{3-d}{2}|\Gamma|
+ \frac{1}{2}\int_{\Gamma}|\nabs\vec u|^2\,\dH^{d-1}
+ O\bigl(\norm{\vec u-\vec\id}_{C^1(\Gamma)}^2\bigr).
\end{equation*}
Here the constant term is omitted from $J$, since it is independent of
$\vec u$.  Moreover, in contrast to the full-velocity dissipation used
in~\cite{BalzaniR12}, our metric dissipation involves only the normal
displacement, or equivalently the normal velocity scaled by the time step. This is
consistent with~\eqref{eq:dissplaw} and leaves the tangential velocity to be selected
independently.

We minimize $J(\vec u)$ either without constraints or subject to a weak constraint
through a feasible set. For later use, set
$\mathcal U=[H^1(\Gamma)]^d$.

\begin{defi}[weak constraint]
Let $\mathcal Z$ be a Banach space for the auxiliary unknown $\lambda$, and let
$\mathcal Y$ be a Hilbert test space.  A weak constraint is a residual operator
\[
\mathcal C:\mathcal U\times\mathcal Z\to\mathcal Y',
\]
or, equivalently, a form
\begin{equation}\label{eq:weak-constraint}
\mathcal G(\vec u,\lambda;\vec y)
=\langle\mathcal C(\vec u,\lambda),\vec y\rangle_{\mathcal Y',\mathcal Y},
\qquad \vec y\in\mathcal Y ,
\end{equation}
which is linear and continuous in the test function $\vec y$. The constraint is imposed
by
\[
\mathcal C(\vec u,\lambda)=0\quad\mbox{in }\mathcal Y',
\quad\mbox{or equivalently}\quad
\mathcal G(\vec u,\lambda;\vec y)=0\quad\forall\,\vec y\in\mathcal Y.
\]
\end{defi}

Here $\lambda\in\mathcal Z$ denotes the auxiliary unknown appearing in the constraint.
If no such unknown is needed, we set $\mathcal Z=\{0\}$ and suppress $\lambda$. Since
$\mathcal Y$ is Hilbert, one may equivalently use the Riesz representative
%$\widehat{\mathcal C}(\vec u,\lambda)=R_{\mathcal Y}^{-1}\mathcal C(\vec u,\lambda)$
%in $\mathcal Y$ 
to write
$\mathcal G(\vec u,\lambda;\vec y)=
\ipd{\widehat{\mathcal C}(\vec u,\lambda),\vec y}_{\mathcal Y}$.  We keep the
notation $\mathcal C(\vec u,\lambda)\in\mathcal Y'$ to emphasize that the constraint is
imposed in weak form. We introduce the feasible set of the constraint by
\begin{equation}\label{eq:feasibleset}
\mathcal A
=\Bigl\{\vec u\in\mathcal U:\,
\text{there exists }\lambda\in\mathcal Z
\quad\text{such that}\quad \mathcal C(\vec u,\lambda)=0
\quad\text{in}\quad \mathcal Y'\Bigr\}.
\end{equation}
When no constraint is imposed, we take $\mathcal Y=\{0\}$ and
$\mathcal C\equiv0$, so that $\mathcal A=\mathcal U$.  The corresponding
minimizing problem is
\begin{align}
    \min_{\vec u\in\mathcal A} J(\vec u).
    \label{eq:constrained-min}
\end{align}

\begin{defi}[admissible constraint]
The feasible set $\mathcal A$ in~\eqref{eq:feasibleset}, or equivalently the weak
constraint~\eqref{eq:weak-constraint} defining it, is called admissible if the identity
map belongs to the feasible set:
\begin{equation}\label{eq:admissible}
\vec\id|_{\Gamma}\in\mathcal A.
\end{equation}
Equivalently, there exists an auxiliary variable $\lambda_{\id}\in\mathcal Z$ such that
\begin{equation}\label{eq:adc}
\mathcal G(\vec\id|_{\Gamma},\lambda_{\id};\vec y)=0
\quad\forall\,\vec y\in\mathcal Y.
\end{equation}
\end{defi}

The unconstrained case is automatically admissible. We have the following stability
estimate for solutions of the minimization problem~\eqref{eq:constrained-min}.
\begin{thm}[stability estimate]
If $\vec u^*$ solves~\eqref{eq:constrained-min} and the feasible set
$\mathcal A$ is admissible, then
\begin{equation}\label{eq:stab1c}
\frac{1}{2}\int_{\Gamma}|\nabs\vec u^*|^2\,\dH^{d-1}
+ \frac{1}{2\tau}\norm{(\vec u^*-\vec\id|_{\Gamma})\cdot\vec\nu}^2_X
\leq \frac{1}{2}\int_{\Gamma}|\nabs\vec\id|^2\,\dH^{d-1}.
\end{equation}
\end{thm}
\begin{proof}
By admissibility, $\vec\id|_{\Gamma}\in\mathcal A$. Hence
$J(\vec u^*)\leq J(\vec\id|_{\Gamma})$, which is exactly~\eqref{eq:stab1c}.
\end{proof}

The admissibility keeps the identity map available as a comparison map and thereby
preserves the variational inequality of the constrained minimizing movement. In the
discrete schemes below, we implement this principle by formulating the weak constraints
so that the identity map remains feasible at the discrete level.

\subsection{Formulation without constraints}

We now consider the unconstrained case. The first variation of $J(\vec u)$ in an
arbitrary direction
$\vec\eta\in[H^1(\Gamma)]^d$ yields the first-order optimality condition for
the minimizer $\vec u^*$:
\begin{align}
0={}&J^\prime(\vec u^*)[\vec\eta] = \frac{\rd}{\rd\epsilon} J(\vec u^* + \epsilon\vec\eta)|_{\epsilon=0} \nn\\
={}&
\frac{1}{\tau}
\ipd{[\vec u^*-\vec\id]\cdot\vec\nu,~\vec\eta\cdot\vec\nu}_{X}
+ \int_{\Gamma}\nabs\vec u^*: \nabs\vec\eta\,\dH^{d-1}
\quad\forall\,\vec\eta\in[H^1(\Gamma)]^d.
\label{eq:firstv}
\end{align}

For the $L^2(\Gamma)$ metric,~\eqref{eq:firstv} becomes
\begin{equation}\label{eqn:BGNmf}
0={}
\frac{1}{\tau}\int_{\Gamma}
[\vec u^*-\vec\id]\cdot\vec\nu~(\vec\eta\cdot\vec\nu)\,\dH^{d-1}
+ \int_{\Gamma}\nabs\vec u^*: \nabs\vec\eta\,\dH^{d-1}
\quad\forall\,\vec\eta\in[H^1(\Gamma)]^d.
\end{equation}
For the $H^{-1}(\Gamma)$ metric, the inverse Laplace--Beltrami operator is understood on
the zero-mean subspace. More precisely, for a zero-mean function $f$, let $\psi_f\in
H^1(\Gamma)/\bR$ solve
\[
\int_{\Gamma}\nabs\psi_f\cdot\nabs\chi\,\dH^{d-1}
=\int_{\Gamma}f\,\chi\,\dH^{d-1}
\quad\forall\,\chi\in H^1(\Gamma).
\]
This solution is the $H^{-1}$ Riesz representative of $f$; in particular,
$\ipd{f,g}_{H^{-1}(\Gamma)}
=\int_{\Gamma}g\,\psi_f\,\dH^{d-1}$ for zero-mean $f$ and $g$. Applying this
construction with
$f=[\vec u^*-\vec\id]\cdot\vec\nu$ and setting
$z=\tau^{-1}\psi_f$,~\eqref{eq:firstv} is equivalently written as
\begin{subequations}\label{eqn:BGNsd}
\begin{align}
0={}&
\frac{1}{\tau}\int_{\Gamma}
[\vec u^*-\vec\id]\cdot\vec\nu\,\varphi\,\dH^{d-1}
- \int_{\Gamma} \nabs z\cdot \nabs\varphi\,\dH^{d-1}
\quad\forall\,\varphi\in H^1(\Gamma),
\label{eqn:BGNsd-z}\\[0.4em]
0={}&
\int_{\Gamma}z\,\vec\eta\cdot\vec\nu\,\dH^{d-1}
+ \int_{\Gamma}\nabs\vec u^*: \nabs\vec\eta\,\dH^{d-1}
\quad\forall\,\vec\eta\in[H^1(\Gamma)]^d.
\label{eqn:BGNsd-u}
\end{align}
\end{subequations}
In particular, choosing $\varphi=1$ in~\eqref{eqn:BGNsd-z} gives the zero-mean
compatibility condition for the normal displacement.

Here~\eqref{eqn:BGNsd}, combined with the parametric piecewise linear approximations in
space, reduces to the well-known BGN scheme for surface diffusion flow;
see~\cite[(2.6)]{BGN08parametric}. Similarly,~\eqref{eqn:BGNmf} can be understood as the
BGN approximation for mean curvature flow without spatial discretization. In both cases
the minimization determines not only the normal displacement but also an implicit
BGN-type tangential velocity.

\subsection{Formulation with constraints}
For a constrained problem, the first-order condition contains the variation of the weak
residual. Assume that $\mathcal C$ is Fr\'echet differentiable at the feasible pair
$(\vec u^*,\lambda^*)$ and that the linearized constraint operator
\[
D\mathcal C(\vec u^*,\lambda^*):
\mathcal U\times\mathcal Z\to\mathcal Y'
\]
is surjective. This constraint qualification gives a Lagrange multiplier
$\vec p\in\mathcal Y$, after identifying $(\mathcal Y')'$ with $\mathcal Y$.
The Lagrangian can be written directly as
\[
\mathscr L(\vec u,\lambda,\vec p)
=J(\vec u)+\langle\mathcal C(\vec u,\lambda),\vec p\rangle_{\mathcal Y',\mathcal Y}
=J(\vec u)+\mathcal G(\vec u,\lambda;\vec p).
\]
Since $\vec p$ is fixed when variations are taken in $(\vec u,\lambda)$, the derivatives
of $\mathcal G$ are the duality pairings with the derivatives of $\mathcal C$:
\[
D_{\vec u}\mathcal G(\vec u,\lambda;\vec p)[\vec\eta]
=\langle D_{\vec u}\mathcal C(\vec u,\lambda)[\vec\eta],
\vec p\rangle_{\mathcal Y',\mathcal Y},
\]
and similarly for the $\lambda$-variation. The Euler--Lagrange system then reads: find
$(\vec u^*,\lambda^*,\vec p)$ such that
\begin{subequations}\label{eq:weak-constrained-el}
\begin{align}
0={}&
\frac{1}{\tau}
\ipd{[\vec u^*-\vec\id]\cdot\vec\nu,~\vec\eta\cdot\vec\nu}_{X}
+ \int_{\Gamma}\nabs\vec u^*: \nabs\vec\eta\,\dH^{d-1}
\nn\\
&\quad
+ \langle D_{\vec u}\mathcal C(\vec u^*,\lambda^*)[\vec\eta],
\vec p\rangle_{\mathcal Y',\mathcal Y}
\quad\forall\,\vec\eta\in[H^1(\Gamma)]^d,
\label{eq:weak-constrained-el-u}\\[0.4em]
0={}&
\langle D_{\lambda}\mathcal C(\vec u^*,\lambda^*)[\rho],
\vec p\rangle_{\mathcal Y',\mathcal Y}
\quad\forall\,\rho\in\mathcal Z,
\label{eq:weak-constrained-el-lambda}\\[0.4em]
0={}&
\langle\mathcal C(\vec u^*,\lambda^*),\vec y\rangle_{\mathcal Y',\mathcal Y}
\quad\forall\,\vec y\in\mathcal Y.
\label{eq:weak-constrained-el-g}
\end{align}
\end{subequations}
If the weak constraint contains no auxiliary unknown $\lambda$, we omit the stationarity
condition~\eqref{eq:weak-constrained-el-lambda}. If no constraint is imposed at all, the
constraint terms are absent, and the system reduces to~\eqref{eq:firstv}.

As an example, we consider the constraint from the MDR formulation for the tangential
velocity in~\cite{Hu22evolving}.
\begin{equation}\label{eq:mdr-strong}
\lambda\vec\nu
=\Delta_s\left(\vec u-\vec\id|_{\Gamma}\right)
\quad\mbox{on }\Gamma.
\end{equation}
This is obtained from the velocity form
$\widetilde\lambda\vec\nu
=\Delta_s((\vec u-\vec\id|_{\Gamma})/\tau)$ by multiplying by $\tau$ and setting
$\lambda=\tau\widetilde\lambda$. We take
$\mathcal Z=L^2(\Gamma)$ and $\mathcal Y=[H^1(\Gamma)]^d$.
The weak form of~\eqref{eq:mdr-strong} is obtained by testing with
$\vec y\in\mathcal Y$ and integrating by parts on $\Gamma$:
\begin{equation}\label{eq:mdr-weak}
0=\int_{\Gamma}
\nabs\left(\vec u-\vec\id|_{\Gamma}\right):
\nabs\vec y\,\dH^{d-1}
+\int_{\Gamma}\lambda\,\vec y\cdot\vec\nu\,\dH^{d-1}
\quad\forall\,\vec y\in[H^1(\Gamma)]^d.
\end{equation}
Thus the corresponding residual operator
$\mathcal C:\mathcal U\times L^2(\Gamma)\to\mathcal Y'$ is defined by
\begin{equation}\label{eq:mdr-C}
\langle\mathcal C(\vec u,\lambda),\vec y\rangle_{\mathcal Y',\mathcal Y}
=\int_{\Gamma}
\nabs\left(\vec u-\vec\id|_{\Gamma}\right):
\nabs\vec y\,\dH^{d-1}
+\int_{\Gamma}\lambda\,\vec y\cdot\vec\nu\,\dH^{d-1}.
\end{equation}
%Note that the identity map is admissible, since
%$\vec u=\vec\id|_{\Gamma}$ and $\lambda=0$ make
%\eqref{eq:mdr-weak} vanish.

For use in~\eqref{eq:weak-constrained-el}, the derivatives of this residual are
\begin{align}
\langle D_{\vec u}\mathcal C(\vec u,\lambda)[\vec\eta],
\vec p\rangle_{\mathcal Y',\mathcal Y}
&=
\int_{\Gamma}\nabs\vec\eta:\nabs\vec p\,\dH^{d-1},
\label{eq:mdr-Du}\\
\langle D_{\lambda}\mathcal C(\vec u,\lambda)[\rho],
\vec p\rangle_{\mathcal Y',\mathcal Y}
&=
\int_{\Gamma}\rho\,\vec p\cdot\vec\nu\,\dH^{d-1}.
\label{eq:mdr-Dlambda}
\end{align}

Substituting these expressions into~\eqref{eq:weak-constrained-el}, the MDR version of
the constrained Euler--Lagrange system is: find
$(\vec u^*,\lambda^*,\vec p)\in[H^1(\Gamma)]^d
\times L^2(\Gamma)\times[H^1(\Gamma)]^d$ such that
\begin{subequations}\label{eqn:mdr-weakel}
\begin{align}
0={}&
\frac{1}{\tau}
\ipd{[\vec u^*-\vec\id|_{\Gamma}]\cdot\vec\nu,
\vec\eta\cdot\vec\nu}_{X}
+ \int_{\Gamma}\nabs\vec u^*: \nabs\vec\eta\,\dH^{d-1}
\nn\\
&\quad
+\int_{\Gamma}\nabs\vec\eta:\nabs\vec p\,\dH^{d-1}
\quad\forall\,\vec\eta\in[H^1(\Gamma)]^d,
\label{eq:mdr-weakel-u}\\
0={}&
\int_{\Gamma}\rho\,\vec p\cdot\vec\nu\,\dH^{d-1}
\quad\forall\,\rho\in L^2(\Gamma),
\label{eq:mdr-weakel-lambda}\\
0={}&
\int_{\Gamma}
\nabs\left(\vec u^*-\vec\id|_{\Gamma}\right):
\nabs\vec y\,\dH^{d-1}
+\int_{\Gamma}\lambda^*\,\vec y\cdot\vec\nu\,\dH^{d-1}
\quad\forall\,\vec y\in[H^1(\Gamma)]^d.
\label{eq:mdr-weakel-g}
\end{align}
\end{subequations}
The second equation states that the Lagrange multiplier $\vec p$ is tangential in the
weak sense, i.e. $\vec p\cdot\vec\nu=0$ in $L^2(\Gamma)$.

When combined with parametric piecewise linear spatial
discretizations,~\eqref{eqn:mdr-weakel} yields the recent dual-MDR scheme: the
$L^2(\Gamma)$ metric gives mean curvature flow, while the $H^{-1}(\Gamma)$
metric gives surface diffusion flow; see~\cite[(2.6),(2.7)]{Gao26dualMDR}.

\section{Parametric finite element approximations}
\label{sec:pfem}

We now introduce the fully discrete parametric finite element approximations. Let $T>0$
and $\tau = T/M$ with a fixed
$M\in\mathbb{N}_+$.  We set $t_m=m\,\tau$ for
$m=0,1,\ldots, M$ as the discrete times.

Assume we are given a closed hypersurface $\Gm\subset\bR^d$ such that
\begin{equation}
\Gm = \cup_{j=1}^J\overline{\sigma}_j^m\quad\mbox{with}\quad
\mT^m=\{\sigma_j^m\}_{j=1}^J,\quad
\mQ^m = \{\vec q_k^m\}_{k=1}^K,
\end{equation}
where $\mT^m$ is a collection of mutually disjoint $(d-1)$-simplices in $\bR^d$, and
$\mQ^m$ is the set of globally labeled vertices. Associated with $\Gm$, we define the
finite element space
\begin{equation}\label{eq:FEMspace}
\bV^h(\Gm):=\left\{\varphi\in C(\Gm):\;\varphi|_{\sigma}\;\mbox{is affine}\quad\forall\sigma\in\mT^m\right\}.
\end{equation}

Let $\{\vec q_{j_k}^m\}_{k=0}^{d-1}$ be the vertices of $\sigma_j^m$, ordered with the
same orientation for all $\sigma_j^m\in\mT^m$. We introduce the geometric unit normal
$\vec\nu^m$ to $\Gm$ elementwise by
\begin{align}\label{eq:geonormal}
\vec\nu_j^m := \vec\nu^m|_{\sigma_j^m}
:=\frac{\vec N(\sigma_j^m)}{|\vec N(\sigma_j^m)|},
\quad
\vec N(\sigma_j^m)
= (\vec q_{j_1}^m-\vec q_{j_0}^m)\wedge\cdots\wedge
(\vec q_{j_{d-1}}^m-\vec q_{j_0}^m),
\end{align}
where $\wedge$ is the wedge product and $\vec N(\sigma)$ denotes the orientation vector
of $\sigma\in\mT^m$.

With a slight abuse of notation, we write
$\ipd{\cdot,\cdot}_{L^2(\Gm)} = \ipd{\cdot,\cdot}_{\Gm}$.  We also
introduce the mass-lumped approximation over the current polyhedral surface
$\Gamma^m$ via
\begin{equation}
\ipd{u, v}_{\Gm}^h:=
\frac{1}{d}\sum_{j=1}^{J} \mH^{d-1}(\sigma^{m}_j)
\sum_{k=0}^{d-1}
\underset{\sigma^{m}_j\ni \vec{x}\to \vec{q}^{m}_{j_k}}{\lim}\,
(u\cdot v)(
\vec{x}),\label{eq:tprule}
\end{equation}
where $u,v$ are piecewise continuous, with possible jumps across the edges of
$\sigma\in\mT^m$, and
$\mH^{d-1}(\sigma^{m}_j)= \frac{1}{(d-1)!}\,|\vec N(\sigma_j^{m})|$
is the measure of $\sigma^{m}_j$.

We next present the finite element approximations of the BGN
formulation~\eqref{eqn:BGNsd} and the dual-MDR formulation~\eqref{eqn:mdr-weakel}. For
ease of presentation, we focus on surface diffusion flow with $X=H^{-1}(\Gamma)$. We use
the following common discrete minimizing-movement functional. For
$\vec v_h\in[\bV^h(\Gm)]^d$, set $\vec r_h=\vec v_h-\vec\id|_{\Gm}$ and let
$z_{\vec v_h}\in\bV^h(\Gm)$, up to a constant, solve
\begin{equation}\label{eq:discretezv}
\ipd{\nabs z_{\vec v_h},\nabs\zeta}_{\Gm}
=\frac1{\tau}\ipd{\vec r_h\cdot\vec\nu^m,\zeta}_{\Gm}^h
\quad\forall\,\zeta\in\bV^h(\Gm).
\end{equation}
This discrete Poisson problem defines the $H^{-1}(\Gm)$ potential associated with the
normal displacement $\vec r_h\cdot\vec\nu^m$. It requires the compatibility condition
$\ipd{\vec r_h\cdot\vec\nu^m,1}_{\Gm}^h=0$, which is obtained
from~\eqref{eq:discretezv} by choosing $\zeta=1\in\bV^h(\Gm)$.  With this
convention for the discrete $H^{-1}$ norm, the discrete functional is
\begin{align}
J_h^m(\vec v_h)
&:=\frac{1}{2\tau}\norm{(\vec v_h-\vec\id|_{\Gm})\cdot\vec\nu^m}_{H^{-1}(\Gm)}^2
+\frac12\ipd{\nabs\vec v_h,\nabs\vec v_h}_{\Gm}\nn \\
&=\frac{\tau}{2}\ipd{\nabs z_{\vec v_h},\nabs z_{\vec v_h}}_{\Gm}
+\frac12\ipd{\nabs\vec v_h,\nabs\vec v_h}_{\Gm}.
\label{eq:discreteJF}
\end{align}
The second equality follows by taking $\zeta=z_{\vec v_h}$ in~\eqref{eq:discretezv}.

\subsection{The BGN scheme}
\label{subsec:classical-bgn}

We first state the fully discrete BGN system associated with~\eqref{eqn:BGNsd}. Given
the initial nondegenerate polyhedral surface $\Gamma^0$, for each $m\geq 0$, find
\[
(\vec u^{m+1},z^{m+1})\in [\bV^h(\Gm)]^d\times\bV^h(\Gm)
\]
such that
\begin{subequations}\label{eq:bgn-system}
\begin{align}
0={}&
\frac1{\tau}\ipd{(\vec u^{m+1}-\vec\id)\cdot\vec\nu^m,\zeta^h}_{\Gm}^h-\ipd{\nabs z^{m+1},\nabs\zeta^h}_{\Gm}
\quad\forall\,\zeta^h\in\bV^h(\Gm),
\label{eq:bgn-z}\\[0.4em]
0={}&
\ipd{z^{m+1},\vec\eta^h\cdot\vec\nu^m}^h_{\Gm}
+\ipd{\nabs\vec u^{m+1},\nabs\vec\eta^h}_{\Gm}
\quad\forall\,\vec\eta^h\in[\bV^h(\Gm)]^d,
\label{eq:bgn-u}
\end{align}
\end{subequations}
and then set $\Gamma^{m+1}=\vec u^{m+1}(\Gm)$.

The scheme~\eqref{eq:bgn-system} coincides with the BGN scheme
in~\cite[(2.6)]{BGN08parametric}. Its unique solvability, unconditional stability, and
mesh properties were established in~\cite{BGN08parametric}.

\subsection{The dual-MDR scheme}
\label{subsec:classical-mdr}

The dual-MDR discretization is obtained by applying the finite element approximations
to~\eqref{eqn:mdr-weakel}. With the same initial discrete data, for each $m\geq0$, find
\[
(\vec u^{m+1},z^{m+1},\vec p^{m+1},\lambda^{m+1})\in
[\bV^h(\Gm)]^d\times\bV^h(\Gm)\times[\bV^h(\Gm)]^d\times\bV^h(\Gm)
\]
such that, with $\vec w^{m+1}:=\vec u^{m+1}-\vec\id|_{\Gm}$,
\begin{subequations}\label{eq:dual-mdr-system}
\begin{align}
0={}&
\frac1{\tau}\ipd{\vec w^{m+1}\cdot\vec\nu^m,\zeta^h}_{\Gm}^h
-\ipd{\nabs z^{m+1},\nabs\zeta^h}_{\Gm}
\quad\forall\,\zeta^h\in\bV^h(\Gm),
\label{eq:dual-mdr-z}\\[0.4em]
0={}&
\ipd{z^{m+1},\vec\eta^h\cdot\vec\nu^m}_{\Gm}^h
+\ipd{\nabs\vec u^{m+1},\nabs\vec\eta^h}_{\Gm}\nn\\
{}&
+\ipd{\nabs\vec\eta^h,\nabs\vec p^{m+1}}_{\Gm}
\quad\forall\,\vec\eta^h\in[\bV^h(\Gm)]^d,
\label{eq:dual-mdr-u}\\[0.4em]
0={}&
\ipd{\rho^h\,\vec\nu^m,\vec p^{m+1}}_{\Gm}^h
\quad\forall\,\rho^h\in\bV^h(\Gm),
\label{eq:dual-mdr-lambda}\\[0.4em]
0={}&
\ipd{\nabs\vec w^{m+1},\nabs\vec y^h}_{\Gm}
+\ipd{\lambda^{m+1}\vec\nu^m,\vec y^h}_{\Gm}^h
\quad\forall\,\vec y^h\in[\bV^h(\Gm)]^d,
\label{eq:dual-mdr-constraint}
\end{align}
\end{subequations}
and then set $\Gamma^{m+1}=\vec u^{m+1}(\Gm)$.

The identity map is an admissible comparison map for the discrete MDR constraint: if
$\vec u^{m+1}=\vec\id|_{\Gm}$, then $\vec w^{m+1}=\vec0$
and~\eqref{eq:dual-mdr-constraint} holds with $\lambda^{m+1}=0$. The dual-MDR
scheme~\eqref{eq:dual-mdr-system} inherits the minimizing-movement energy estimate.
Testing~\eqref{eq:dual-mdr-z} with $\zeta^h=z^{m+1}$, testing~\eqref{eq:dual-mdr-u} with
$\vec\eta^h=\vec w^{m+1}$, using~\eqref{eq:dual-mdr-constraint} with
$\vec y^h=\vec p^{m+1}$ together with~\eqref{eq:dual-mdr-lambda} with
$\rho^h=\lambda^{m+1}$, and
combining these equations gives
\begin{align}
{}&\frac12\ipd{\nabs\vec u^{m+1},\nabs\vec u^{m+1}}_{\Gm}
+\frac12\ipd{\nabs\vec w^{m+1},\nabs\vec w^{m+1}}_{\Gm}
+\tau\ipd{\nabs z^{m+1},\nabs z^{m+1}}_{\Gm}
\nn\\
&\qquad
=\frac12\ipd{\nabs\vec\id,\nabs\vec\id}_{\Gm}.
\label{eq:dual-mdr-energy}
\end{align}
Thus the dual-MDR scheme is unconditionally energy stable. Further properties
of~\eqref{eq:dual-mdr-system} can be found in the recent work~\cite{Gao26dualMDR}.

\section{New admissible formulations}
\label{sec:mm-admissible}

In this section, we present two new schemes obtained by imposing different admissible
weak constraints for the tangential velocities.

\subsection{The admissible BGN scheme}
\label{subsec:admissible-bgn}

Analogously to the MDR approach, we first consider minimizing $J(\vec u)$ subject to the
BGN constraint
\begin{equation}\label{eq:BGNcon}
\lambda\,\vec\nu = \Delta_s\vec u\qquad\mbox{on}\quad\Gamma,
\end{equation}
which is admissible because of the curvature identity
$\varkappa\,\vec\nu = \Delta_s\vec\id$.  Recalling~\eqref{eq:weak-constrained-el}, the BGN
version of the constrained Euler--Lagrange system for surface diffusion flow is as
follows: find
$(\vec u^*, z^*, \lambda^*,\vec p)\in[H^1(\Gamma)]^d\times H^1(\Gamma)
\times L^2(\Gamma)\times[H^1(\Gamma)]^d$ such that
\begin{subequations}\label{eqn:bgn-weakel}
\begin{align}
0={}&
\frac{1}{\tau}
\int_{\Gamma}[\vec u^*-\vec\id|_{\Gamma}]\cdot\vec\nu\,\varphi\,\dH^{d-1}
-\int_{\Gamma}\nabs z^*\cdot\nabs\varphi\,\dH^{d-1}
\quad\forall\varphi\in H^1(\Gamma),
\label{eq:bgn-weakel-z}\\
0={}&
\int_{\Gamma}z^*\,\vec\nu\cdot\vec\eta\,\dH^{d-1}
+ \int_{\Gamma}\nabs\vec u^*:\nabs\vec\eta\,\dH^{d-1}\nn\\
{}&+\int_{\Gamma}\nabs\vec\eta:\nabs\vec p\,\dH^{d-1}
\quad\forall\,\vec\eta\in[H^1(\Gamma)]^d,
\label{eq:bgn-weakel-u}\\
0={}&
\int_{\Gamma}\rho\,\vec p\cdot\vec\nu\,\dH^{d-1}
\quad\forall\,\rho\in L^2(\Gamma),
\label{eq:bgn-weakel-lambda}\\
0={}&
\int_{\Gamma}
\nabs\vec u^*:
\nabs\vec y\,\dH^{d-1}
+\int_{\Gamma}\lambda^*\,\vec y\cdot\vec\nu\,\dH^{d-1}
\quad\forall\,\vec y\in[H^1(\Gamma)]^d.
\label{eq:bgn-weakel-g}
\end{align}
\end{subequations}
A natural parametric finite element approximation then leads to a BGN-type scheme whose
discrete counterpart of~\eqref{eq:bgn-weakel-g} is
\begin{equation}\label{eq:bgn-uv}
0=
\ipd{\lambda^{m+1},\vec y^h\cdot\vec\nu^m}^h_{\Gm}
+\ipd{\nabs\vec u^{m+1},\nabs\vec y^h}_{\Gm}
\quad\forall\,\vec y^h\in[\bV^h(\Gm)]^d.
\end{equation}

However, this direct discretization does not, in general, preserve the unconditional
stability estimate. The obstruction arises at the discrete level: in
general,~\eqref{eq:bgn-uv} does not hold for
$\vec u^{m+1}=\vec\id|_{\Gm}$, and hence the discrete constraint is not
admissible. This observation is consistent with the fact that, even for zero normal
displacement, i.e., $\vec u-\vec\id=\vec0$, the BGN scheme can still induce an implicit
tangential velocity.

We next modify the discrete constraint~\eqref{eq:bgn-uv} to make it admissible. For
later use, following~\cite{BGN08parametric}, we introduce the vertex normal vector
$\vec\nu_{h,p}^{m}\in [\bV^h(\Gm)]^d$ as the mass-lumped $L^2$--projection of
$\vec\nu^m$ onto
$[\bV^h(\Gm)]^d$, i.e.,
\begin{equation} \label{eq:nuhomegah}
\ipd{\vec\nu_{h,p}^m, \vec\chi^h}_{\Gm}^h
= \ipd{\vec\nu^m,~\vec\chi^h}_{\Gm}^h
= \ipd{\vec\nu^m,~\vec\chi^h}_{\Gm}
\quad\forall\vec\chi^h\in [\bV^h(\Gm)]^d,
\end{equation}
where we have used~\eqref{eq:tprule}. It follows that
\begin{equation}\label{eq:vpiden}
\ipd{\chi\,\vec\nu_{h,p}^m,~\vec\eta^h}_{\Gm}^h
= \ipd{\chi\,\vec\nu^m,~\vec\eta^h}_{\Gm}^h
\qquad\forall\chi\in \bV^h(\Gm),\quad\vec\eta^h\in [\bV^h(\Gm)]^d.
\end{equation}

At each time step, we first compute the mass-lumped vector curvature
$\vec\varkappa_h^m\in[\bV^h(\Gm)]^d$ from
\begin{align}
\ipd{\vec\varkappa_h^m,\vec\eta^h}_{\Gm}^h
+\ipd{\nabs\vec\id,\nabs\vec\eta^h}_{\Gm}=0
\quad\forall\,\vec\eta^h\in[\bV^h(\Gm)]^d.
\label{eq:adm-bgn-vector-curvature}
\end{align}
For each vertex $\vec q\in\mQ^m$, define
\begin{align}
\vec\nu_{h,{\rm rd}}^m(\vec q)=
\left\{
\begin{array}{ll}
\dfrac{\vec\varkappa_h^m(\vec q)}
{|\vec\varkappa_h^m(\vec q)|},
&\quad\mbox{if }|\vec\varkappa_h^m(\vec q)|\neq0,\\[1.0em]
\vec\nu_{h,p}^m(\vec q),
&\quad\mbox{otherwise.}
\end{array}
\right.
\label{eq:adm-bgn-rd-normal}
\end{align}
We set $\varkappa_{h,{\rm rd}}^m(\vec q) =|\vec\varkappa_h^m(\vec q)|$ at vertices of
nonzero vector-curvature and
$\varkappa_{h,{\rm rd}}^m(\vec q)=0$ otherwise.  Then
$\varkappa_{h,{\rm rd}}^m\vec\nu_{h,{\rm rd}}^m=\vec\varkappa_h^m$ in the
mass-lumped nodal sense, and hence
\begin{align}
0=\ipd{\nabs\vec\id,\nabs\vec\eta^h}_{\Gm}
+\ipd{\varkappa_{h,{\rm rd}}^m\vec\nu_{h,{\rm rd}}^m,
\vec\eta^h}_{\Gm}^h
\quad\forall\,\vec\eta^h\in[\bV^h(\Gm)]^d.
\label{eq:bgn-id-admissible}
\end{align}
This motivates the following admissible modification of the BGN
constraint~\eqref{eq:bgn-uv}:
\begin{align}
0=\ipd{\nabs\vec u^{m+1},\nabs\vec\eta^h}_{\Gm}
+\ipd{\lambda^{m+1}\vec\nu_{h,{\rm rd}}^m,
\vec\eta^h}_{\Gm}^h
\quad\forall\,\vec\eta^h\in[\bV^h(\Gm)]^d,
\label{eq:newbgn-admi}
\end{align}
so that it is admissible.

Given the initial nondegenerate polyhedral surface $\Gamma^0$, for each
$m\geq0$ the admissible BGN scheme seeks
\[
(\vec u^{m+1},z^{m+1},\vec p^{m+1},\lambda^{m+1})\in
[\bV^h(\Gm)]^d\times\bV^h(\Gm)\times[\bV^h(\Gm)]^d\times\bV^h(\Gm)
\]
such that
\begin{subequations}\label{eq:adm-bgn-system}
\begin{align}
0={}&
\frac1{\tau}\ipd{[\vec u^{m+1}-\vec\id]\cdot\vec\nu_{h,{\rm rd}}^m,\zeta^h}_{\Gm}^h
-\ipd{\nabs z^{m+1},\nabs\zeta^h}_{\Gm}
\quad\forall\,\zeta^h\in\bV^h(\Gm),
\label{eq:adm-bgn-z}\\[0.4em]
0={}&
\ipd{z^{m+1},\vec\eta^h\cdot\vec\nu_{h,{\rm rd}}^m}_{\Gm}^h
+\ipd{\nabs\vec u^{m+1},\nabs\vec\eta^h}_{\Gm}\nn\\
{}&
+\ipd{\nabs\vec p^{m+1},\nabs\vec\eta^h}_{\Gm}
\quad\forall\,\vec\eta^h\in[\bV^h(\Gm)]^d,
\label{eq:adm-bgn-u}\\[0.4em]
0={}&
\ipd{\mu^h\vec\nu_{h,{\rm rd}}^m,\vec p^{m+1}}_{\Gm}^h
\quad\forall\,\mu^h\in\bV^h(\Gm),
\label{eq:adm-bgn-lambda}\\[0.4em]
0={}&
\ipd{\nabs\vec u^{m+1},\nabs\vec y^h}_{\Gm}
+\ipd{\lambda^{m+1}\vec\nu_{h,{\rm rd}}^m,\vec y^h}_{\Gm}^h
\quad\forall\,\vec y^h\in[\bV^h(\Gm)]^d,
\label{eq:adm-bgn-p}
\end{align}
\end{subequations}
and then set $\Gamma^{m+1}=\vec u^{m+1}(\Gm)$.

The scheme~\eqref{eq:adm-bgn-system} induces a tangential velocity different from that
of the classical BGN scheme~\eqref{eq:bgn-system}. The normal
$\vec\nu_{h,{\rm rd}}^m$ is also used in~\eqref{eq:adm-bgn-lambda} to obtain the
stability estimate, and in~\eqref{eq:adm-bgn-z}--\eqref{eq:adm-bgn-u} to obtain unique
solvability. The following theorem states the result.
\begin{thm}[existence and uniqueness]
Assume that
\begin{enumerate}[label=$(\mathbf{A\arabic*})$, ref=$\mathbf{A\arabic*}$]
\item \label{assumpI} $\mH^{d-1}(\sigma) > 0$ for all $\sigma \in \mT^m$;
\item \label{assupII} $\dim\operatorname{span}\left(
\bigl\{\vec\nu_{h,{\rm rd}}^{m}(\vec q)\bigr\}_{\vec q\in \mQ^m}\right)=d$;
\item \label{assupIII} $\vec\nu_{h,{\rm rd}}^{m}(\vec q)\neq \vec 0$ for all $\vec q\in\mQ^m$.
\end{enumerate}
Then the scheme~\eqref{eq:adm-bgn-system} admits a unique solution
\[
(\vec u^{m+1},z^{m+1},\vec p^{m+1},\lambda^{m+1})\in
[\bV^h(\Gm)]^d\times\bV^h(\Gm)\times[\bV^h(\Gm)]^d\times\bV^h(\Gm).
\]
\end{thm}

\begin{proof}
Since~\eqref{eq:adm-bgn-system} is a square finite-dimensional linear system, it
suffices to show that the corresponding homogeneous system has only the trivial
solution. Consider the homogeneous problem: find
\[
(\vec u^h,z^h,\vec p^h,\lambda^h)\in
[\bV^h(\Gm)]^d\times\bV^h(\Gm)\times[\bV^h(\Gm)]^d\times\bV^h(\Gm)
\]
such that, for all
$(\zeta^h, \vec\eta^h, \mu^h, \vec y^h)\in
\bV^h(\Gm)\times[\bV^h(\Gm)]^d\times\bV^h(\Gm)
\times[\bV^h(\Gm)]^d$,
\begin{subequations}\label{eq:hobgn-system}
\begin{align}
0={}&
\frac1{\tau}\ipd{\vec u^h\cdot\vec\nu_{h,{\rm rd}}^m,\zeta^h}_{\Gm}^h
-\ipd{\nabs z^h,\nabs\zeta^h}_{\Gm},
\label{eq:hobgn-z}\\
0={}&
\ipd{z^h,\vec\eta^h\cdot\vec\nu_{h,{\rm rd}}^m}_{\Gm}^h
+\ipd{\nabs\vec u^h,\nabs\vec\eta^h}_{\Gm}+\ipd{\nabs\vec p^h,\nabs\vec\eta^h}_{\Gm},
\label{eq:hobgn-u}\\
0={}&
\ipd{\mu^h\vec\nu_{h,{\rm rd}}^m,\vec p^h}_{\Gm}^h,
\label{eq:hobgn-lambda}\\
0={}&
\ipd{\nabs\vec u^h,\nabs\vec y^h}_{\Gm}
+\ipd{\lambda^h\vec\nu_{h,{\rm rd}}^m,\vec y^h}_{\Gm}^h.
\label{eq:hobgn-p}
\end{align}
\end{subequations}
We test these equations with
$\zeta^h = z^h$ in~\eqref{eq:hobgn-z},
$\vec\eta^h=\vec u^h$ in~\eqref{eq:hobgn-u},
$\mu^h =\lambda^h$ in~\eqref{eq:hobgn-lambda}, and
$\vec y^h = \vec p^h$ in~\eqref{eq:hobgn-p}, respectively.  Multiplying the
tested~\eqref{eq:hobgn-z} by $\tau$ and subtracting the tested~\eqref{eq:hobgn-p} from
the tested~\eqref{eq:hobgn-u}, while using the tested~\eqref{eq:hobgn-lambda}, gives
\begin{equation*}
\tau\ipd{\nabs z^h,~\nabs z^h}_{\Gm}
+ \ipd{\nabs\vec u^h,~\nabs\vec u^h}_{\Gm}=0,
\end{equation*}
which, by assumption~\ref{assumpI}, implies that $z^h = z^c$ and $\vec u^h = \vec u^c$
are both constants. Next, we choose $\vec\eta^h = \vec p^h$ in~\eqref{eq:hobgn-u} and
$\mu^h = z^h$ in~\eqref{eq:hobgn-lambda} to obtain
\begin{equation*}
0=\ipd{\nabs\vec p^h,~\nabs\vec p^h}_{\Gm},
\end{equation*}
which, by assumption~\ref{assumpI}, implies that
$\vec p^h = \vec p^c$ is also a constant vector.

Inserting $z^h = z^c$ and $\vec u^h = \vec u^c$ into~\eqref{eq:hobgn-z} gives
\begin{equation*}
0=\ipd{\vec u^c\cdot\vec\nu_{h,{\rm rd}}^m,~\zeta^h}^h_{\Gm},
\end{equation*}
and hence
\[
\vec u^c\cdot\vec\nu_{h,{\rm rd}}^m(\vec q)=0
\quad\forall\,\vec q\in\mQ^m.
\]
Assumption~\ref{assupII} then implies
$\vec u^h=\vec u^c=\vec 0$.  Applying the same argument
to~\eqref{eq:hobgn-lambda} with $\vec p^h=\vec p^c$, and again using
assumption~\ref{assupII}, yields $\vec p^h=\vec 0$.

Inserting these identities into~\eqref{eq:hobgn-u} and~\eqref{eq:hobgn-p} reduces the
remaining equations to
\begin{subequations}
\begin{align}\label{eq:zceta}
\ipd{z^c,~\vec\eta^h\cdot\vec\nu_{h,{\rm rd}}^m}_{\Gm}^h&=0
\quad\forall\,\vec\eta^h\in [\bV^h(\Gm)]^d,\\
\ipd{\lambda^h,~\vec y^h\cdot\vec\nu_{h,{\rm rd}}^m}_{\Gm}^h&=0
\quad\forall\,\vec y^h\in [\bV^h(\Gm)]^d.
\end{align}
\end{subequations}
Choose $\vec\eta^h$ with nodal values
$\vec\eta^h(\vec q)=z^c\vec\nu_{h,{\rm rd}}^m(\vec q)$ for
all $\vec q\in\mQ^m$. The first equation in~\eqref{eq:zceta} and the mass-lumped
quadrature rule give
\begin{equation*}
0= \ipd{z^c,~\vec\eta^h\cdot\vec\nu_{h,{\rm rd}}^m}_{\Gm}^h
= \frac{1}{d}\sum_{j=1}^J\mH^{d-1}(\sigma_j^m)
\sum_{k=0}^{d-1}
\bigl|z^c\vec\nu_{h,{\rm rd}}^m(\vec q_{j_k}^m)\bigr|^2.
\end{equation*}
Assumptions~\ref{assumpI} and~\ref{assupIII} imply $z^c=0$. Similarly, choosing
$\vec y^h$ with nodal values
$\vec y^h(\vec q)=\lambda^h(\vec q)\vec\nu_{h,{\rm rd}}^m(\vec q)$ gives
$\lambda^h=0$.

Thus the homogeneous system~\eqref{eq:hobgn-system} has only the trivial solution.
Therefore, the scheme~\eqref{eq:adm-bgn-system} is uniquely solvable.
\end{proof}

The scheme~\eqref{eq:adm-bgn-system} has the same unconditional energy identity as the
minimizing movement, as stated in the following theorem.
\begin{thm}[unconditional stability] Let
\[
(\vec u^{m+1},z^{m+1},\vec p^{m+1},\lambda^{m+1})\in
[\bV^h(\Gm)]^d\times\bV^h(\Gm)\times[\bV^h(\Gm)]^d\times\bV^h(\Gm)
\]
be a solution to~\eqref{eq:adm-bgn-system}. Then
\begin{equation}\label{eq:adm-bgn-energy}
|\Gamma^{m+1}| + \tau\ipd{\nabs z^{m+1}, \nabs z^{m+1}}_{\Gm}\leq |\Gm|\qquad 0\leq m\leq M-1.
\end{equation}
\end{thm}

\begin{proof}
We test~\eqref{eq:adm-bgn-z} with
$\zeta^h=z^{m+1}$,~\eqref{eq:adm-bgn-u} with
$\vec\eta^h=\vec u^{m+1}-\vec\id$ and use~\eqref{eq:adm-bgn-p} with
$\vec y^h=\vec p^{m+1}$ together with~\eqref{eq:adm-bgn-lambda} with
$\mu^h=\lambda^{m+1}$.  Combining these equations gives
\begin{align}
\tau\ipd{\nabs z^{m+1},\nabs z^{m+1}}_{\Gm}
+\ipd{\nabs\vec u^{m+1},\nabs[\vec u^{m+1}-\vec\id]}_{\Gm}=\ipd{\nabs\vec p^{m+1},\nabs\vec\id}_{\Gm}.
\end{align}
We further test $\vec\eta^h=\vec p^{m+1}$ in the identity~\eqref{eq:bgn-id-admissible}
and
$\mu^h=\varkappa_{h,{\rm rd}}^m$ in~\eqref{eq:adm-bgn-lambda} to obtain that
\[\ipd{\nabs\vec p^{m+1},\nabs\vec\id}_{\Gm}=0.\]

Now recall the inequality (see~\cite[Lemma 57]{Barrett20})
\begin{equation}\label{eq:surfaceinq}
\ipd{\nabs\vec u^{m+1},\nabs[\vec u^{m+1}-\vec\id]}_{\Gm}\geq |\Gamma^{m+1}| -|\Gm|,
\end{equation}
which gives the desired stability result~\eqref{eq:adm-bgn-energy}.
\end{proof}

\begin{rem}
For the admissible BGN scheme~\eqref{eq:adm-bgn-system}, one may also consider the
following reduced two-unknown variant: find
\[
(\vec u^{m+1},z^{m+1})\in [\bV^h(\Gm)]^d\times\bV^h(\Gm)
\]
such that
\begin{subequations}\label{eq:sadm-bgn-system}
\begin{align}
0={}&
\frac1{\tau}\ipd{[\vec u^{m+1}-\vec\id]\cdot\vec\nu_{h,{\rm rd}}^m,\zeta^h}_{\Gm}^h
-\ipd{\nabs z^{m+1},\nabs\zeta^h}_{\Gm}
\quad\forall\,\zeta^h\in\bV^h(\Gm),
\label{eq:sadm-bgn-z}\\
0={}&
\ipd{z^{m+1},\vec\eta^h\cdot\vec\nu_{h,{\rm rd}}^m}_{\Gm}^h
+\ipd{\nabs\vec u^{m+1},\nabs\vec\eta^h}_{\Gm}\quad\forall\,\vec\eta^h\in[\bV^h(\Gm)]^d.
\label{eq:sadm-bgn-u}
\end{align}
\end{subequations}
Compared with the classical BGN scheme~\eqref{eq:bgn-system}, the only change
in~\eqref{eq:sadm-bgn-system} is that the modified vertex normal
$\vec\nu_{h,{\rm rd}}^m$ is used in place of the element normal
$\vec\nu^m$.  The same arguments give unique solvability under
assumptions~\ref{assumpI}--\ref{assupIII}, as well as unconditional energy stability.

Nevertheless, the full mixed formulation~\eqref{eq:adm-bgn-system}, with the auxiliary
variables $\vec p^{m+1}$ and $\lambda^{m+1}$, has the advantage of separating the metric
equations in~\eqref{eq:adm-bgn-z}-\eqref{eq:adm-bgn-u} that determine the normal
velocity and the discrete BGN constraint that determines the tangential motion. This
separation is useful, for instance, in volume-preserving approximations, where the
metric normal may be replaced by a time-weighted interface normal,
see~\S\ref{sec:further-insights}. A direct replacement in the reduced
formulation~\eqref{eq:sadm-bgn-system} would mix this normal choice with the induced
BGN-type tangential motion and may no longer preserve the admissibility. \end{rem}

\subsection{The relaxed MDR scheme}
\label{subsec:relaxed-mdr}

For the MDR constraint~\eqref{eq:mdr-strong}, we introduce the relaxed form
\begin{equation*}
\frac{1}{\alpha}P_{\Gamma}
\left(\frac{\vec u-\vec\id|_{\Gamma}}{\tau}\right)+\lambda\,\vec\nu
=\Delta_s\left(\frac{\vec u-\vec\id|_{\Gamma}}{\tau}\right)
\quad\mbox{on}\quad\Gamma,
\end{equation*}
or, after multiplying by $\tau$ and relabeling $\tau\lambda$ as the auxiliary scalar
$\lambda$,
\begin{equation}\label{eq:relaxmdr-strong}
\frac{1}{\alpha}P_{\Gamma}\left(\vec u-\vec\id|_{\Gamma}\right)
+\lambda\,\vec\nu
=\Delta_s\left(\vec u-\vec\id|_{\Gamma}\right)
\quad\mbox{on}\quad\Gamma,
\end{equation}
where $P_\Gamma=\Id-\vec\nu\otimes\vec\nu$ is the orthogonal projection onto the tangent
space of $\Gamma$, and $\alpha>0$. The constraint~\eqref{eq:relaxmdr-strong} is
admissible, since
$\vec u=\vec\id|_\Gamma$ satisfies it with $\lambda=0$.  Letting
$\alpha\to+\infty$ formally recovers the MDR constraint~\eqref{eq:mdr-strong}.

The weak form of~\eqref{eq:relaxmdr-strong} is
\begin{align}
0=\ipd{\nabs(\vec u-\vec\id|_\Gamma),\nabs\vec y}_{\Gamma}
+\frac1\alpha
\ipd{P_\Gamma(\vec u-\vec\id|_\Gamma),\vec y}_{\Gamma}
+\ipd{\lambda\vec\nu,\vec y}_{\Gamma}
\quad\forall\,\vec y\in[H^1(\Gamma)]^d.
\label{eq:rmdr-weak-constraint}
\end{align}
Recalling~\eqref{eq:weak-constrained-el}, the relaxed MDR Euler--Lagrange system for
surface diffusion flow is as follows: find
$(\vec u^*, z^*, \lambda^*,\vec p)\in[H^1(\Gamma)]^d\times H^1(\Gamma)
\times L^2(\Gamma)\times[H^1(\Gamma)]^d$ such that
\begin{subequations}\label{eqn:rmdr-weakel}
\begin{align}
0={}&
\frac{1}{\tau}
\ipd{[\vec u^*-\vec\id|_{\Gamma}]\cdot\vec\nu,~\varphi}_{\Gamma}
-\ipd{\nabs z^*,~\nabs\varphi}_{\Gamma}
\quad\forall\,\varphi\in H^1(\Gamma),
\label{eq:rmdr-weakel-z}\\
0={}&
\ipd{z^*,~\vec\eta\cdot\vec\nu}_{\Gamma}
+\ipd{\nabs\vec u^*,~\nabs\vec\eta}_{\Gamma}
+\ipd{\nabs\vec p,~\nabs\vec\eta}_{\Gamma}\nn\\
{}&
+\frac1\alpha\ipd{P_\Gamma\vec\eta,~\vec p}_{\Gamma}
\quad\forall\,\vec\eta\in[H^1(\Gamma)]^d,
\label{eq:rmdr-weakel-u}\\
0={}&
\ipd{\rho\vec\nu,~\vec p}_{\Gamma}
\quad\forall\,\rho\in L^2(\Gamma),
\label{eq:rmdr-weakel-lambda}\\
0={}&
\ipd{\nabs(\vec u^*-\vec\id|_\Gamma),~\nabs\vec y}_{\Gamma}
+\ipd{\lambda^*\vec\nu,~\vec y}_{\Gamma}\nn\\
{}&+\frac1\alpha\ipd{P_\Gamma(\vec u^*-\vec\id|_\Gamma),~\vec y}_{\Gamma}
\quad\forall\,\vec y\in[H^1(\Gamma)]^d.
\label{eq:rmdr-weakel-g}
\end{align}
\end{subequations}

For the spatial discretization, we use the vertex normal
$\vec\nu_{h,p}^m$ defined in~\eqref{eq:nuhomegah}.  At every vertex
$\vec q\in\mQ^m$ with $\vec\nu_{h,p}^m(\vec q)\neq\vec0$, set
\[
\widehat{\vec\nu}_{h,p}^m(\vec q)
=\frac{\vec\nu_{h,p}^m(\vec q)} {|\vec\nu_{h,p}^m(\vec q)|},
\qquad
P_{h,p}^m(\vec q) =I-\widehat{\vec\nu}_{h,p}^m(\vec q)
\otimes\widehat{\vec\nu}_{h,p}^m(\vec q),
\]
where we assume that this normalization is well-defined at all vertices. Given the
initial nondegenerate polyhedral surface $\Gamma^0$, for each $m\geq0$ the relaxed MDR
scheme seeks
\[
(\vec u^{m+1},z^{m+1},\vec p^{m+1},\lambda^{m+1})\in
[\bV^h(\Gm)]^d\times\bV^h(\Gm)\times[\bV^h(\Gm)]^d\times\bV^h(\Gm)
\]
such that, with $\vec w^{m+1}:=\vec u^{m+1}-\vec\id|_{\Gm}$,
\begin{subequations}\label{eq:adm-mdr-system}
\begin{align}
0={}&
\frac1{\tau}
\ipd{\vec w^{m+1}\cdot\vec\nu_{h,p}^m,\zeta^h}_{\Gm}^h
-\ipd{\nabs z^{m+1},\nabs\zeta^h}_{\Gm}
\quad\forall\,\zeta^h\in\bV^h(\Gm),
\label{eq:adm-mdr-z}\\[0.4em]
0={}&
\ipd{z^{m+1},\vec\eta^h\cdot\vec\nu_{h,p}^m}_{\Gm}^h
+\ipd{\nabs\vec u^{m+1},\nabs\vec\eta^h}_{\Gm}
+\ipd{\nabs\vec p^{m+1},\nabs\vec\eta^h}_{\Gm}\nn\\
{}&
+\frac1\alpha\ipd{P_{h,p}^m\vec\eta^h,\vec p^{m+1}}_{\Gm}^h
\quad\forall\,\vec\eta^h\in[\bV^h(\Gm)]^d,
\label{eq:adm-mdr-u}\\[0.4em]
0={}&
\ipd{\mu^h\vec\nu_{h,p}^m,\vec p^{m+1}}_{\Gm}^h
\quad\forall\,\mu^h\in\bV^h(\Gm),
\label{eq:adm-mdr-lambda}\\[0.4em]
0={}&
\ipd{\nabs\vec w^{m+1},\nabs\vec y^h}_{\Gm}
+\frac1\alpha
\ipd{P_{h,p}^m\vec w^{m+1},\vec y^h}_{\Gm}^h
+\ipd{\lambda^{m+1}\vec\nu_{h,p}^m,\vec y^h}_{\Gm}^h
\nn\\
{}&\quad\forall\,\vec y^h\in[\bV^h(\Gm)]^d.
\label{eq:adm-mdr-p}
\end{align}
\end{subequations}
and then set $\Gamma^{m+1}=\vec u^{m+1}(\Gm)$. In the limit $\alpha\to+\infty$, the
scheme~\eqref{eq:adm-mdr-system} formally reduces to the dual-MDR
scheme~\eqref{eq:dual-mdr-system}.

The following theorem establishes unique solvability of the resulting linear system
in~\eqref{eq:adm-mdr-system}.
\begin{thm}[existence and uniqueness]
Assume that $\alpha>0$ and that
\begin{enumerate}[label=$(\mathbf{B \arabic*})$, ref=$\mathbf{B \arabic*}$]
\item \label{assumpRI} $\mH^{d-1}(\sigma) > 0$ for all $\sigma \in \mT^m$;
\item \label{assumpRII} $\dim\operatorname{span}\left(
\bigl\{\vec\nu_{h,p}^{m}(\vec q)\bigr\}_{\vec q\in \mQ^m}\right)=d$;
\item \label{assumpRIII} $\vec\nu_{h,p}^{m}(\vec q)\neq \vec 0$ for all
$\vec q\in\mQ^m$.
\end{enumerate}
Then the scheme~\eqref{eq:adm-mdr-system} admits a unique solution
\[
(\vec u^{m+1},z^{m+1},\vec p^{m+1},\lambda^{m+1})\in
[\bV^h(\Gm)]^d\times\bV^h(\Gm)\times[\bV^h(\Gm)]^d\times\bV^h(\Gm).
\]
\end{thm}

\begin{proof}
As in the proof of the corresponding result for~\eqref{eq:adm-bgn-system}, the
finite-dimensional linear system is square, so it is enough to consider the homogeneous
system. Let
$(\vec u^h,z^h,\vec p^h,\lambda^h)\in
[\bV^h(\Gm)]^d\times\bV^h(\Gm)\times[\bV^h(\Gm)]^d\times\bV^h(\Gm)$ solve the
homogeneous version of~\eqref{eq:adm-mdr-system}, i.e., for all
$(\zeta^h,\vec\eta^h,\mu^h,\vec y^h)\in
\bV^h(\Gm)\times[\bV^h(\Gm)]^d\times\bV^h(\Gm)
\times[\bV^h(\Gm)]^d$,
\begin{subequations}\label{eq:homdr-system}
\begin{align}
0={}&
\frac1{\tau}\ipd{\vec u^h\cdot\vec\nu_{h,p}^m,\zeta^h}_{\Gm}^h
-\ipd{\nabs z^h,\nabs\zeta^h}_{\Gm},
\label{eq:homdr-z}\\
0={}&
\ipd{z^h,\vec\eta^h\cdot\vec\nu_{h,p}^m}_{\Gm}^h
+\ipd{\nabs\vec u^h,\nabs\vec\eta^h}_{\Gm}\nn\\
{}&
+\ipd{\nabs\vec p^h,\nabs\vec\eta^h}_{\Gm}
+\frac1\alpha\ipd{P_{h,p}^m\vec\eta^h,\vec p^h}_{\Gm}^h,
\label{eq:homdr-u}\\
0={}&
\ipd{\mu^h\vec\nu_{h,p}^m,\vec p^h}_{\Gm}^h,
\label{eq:homdr-lambda}\\
0={}&
\ipd{\nabs\vec u^h,\nabs\vec y^h}_{\Gm}
+\frac1\alpha\ipd{P_{h,p}^m\vec u^h,\vec y^h}_{\Gm}^h
+\ipd{\lambda^h\vec\nu_{h,p}^m,\vec y^h}_{\Gm}^h.
\label{eq:homdr-p}
\end{align}
\end{subequations}
We test~\eqref{eq:homdr-z} with $\zeta^h=z^h$,~\eqref{eq:homdr-u} with
$\vec\eta^h=\vec u^h$,~\eqref{eq:homdr-lambda} with
$\mu^h=\lambda^h$, and~\eqref{eq:homdr-p} with
$\vec y^h=\vec p^h$.  Multiplying the tested~\eqref{eq:homdr-z} by
$\tau$ and subtracting the tested~\eqref{eq:homdr-p} from the
tested~\eqref{eq:homdr-u}, while using the tested~\eqref{eq:homdr-lambda},
gives
\[
\tau\ipd{\nabs z^h,\nabs z^h}_{\Gm}
+\ipd{\nabs\vec u^h,\nabs\vec u^h}_{\Gm}=0.
\]
Hence $z^h=z^c$ and $\vec u^h=\vec u^c$ are constants. Inserting this
into~\eqref{eq:homdr-z} yields
\[
\ipd{\vec u^c\cdot\vec\nu_{h,p}^m,\zeta^h}_{\Gm}^h=0
\quad\forall\,\zeta^h\in\bV^h(\Gm),
\]
and~\ref{assumpRII} implies $\vec u^c=\vec0$. Equation~\eqref{eq:homdr-p} then gives
$\lambda^h=0$ by~\ref{assumpRIII}. Next, testing~\eqref{eq:homdr-u} with
$\vec\eta^h=\vec p^h$ and~\eqref{eq:homdr-lambda} with $\mu^h=z^c$ gives
\[
\ipd{\nabs\vec p^h,\nabs\vec p^h}_{\Gm}
+\frac1\alpha\ipd{P_{h,p}^m\vec p^h,\vec p^h}_{\Gm}^h=0.
\]
Both terms are nonnegative, since $P_{h,p}^m$ is the nodal orthogonal projection onto
the discrete tangent plane, and hence $\vec p^h=\vec p^c$ is constant.
Using~\eqref{eq:homdr-lambda} and the spanning condition~\ref{assumpRII} gives
$\vec p^c=\vec0$. Finally,~\eqref{eq:homdr-u} with
$\vec p^h=\vec0$ gives
\[
\ipd{z^c,\vec\eta^h\cdot\vec\nu_{h,p}^m}_{\Gm}^h=0
\quad\forall\,\vec\eta^h\in[\bV^h(\Gm)]^d,
\]
and~\ref{assumpRIII} implies $z^c=0$. Thus the homogeneous system has only the zero
solution.
\end{proof}

The identity map is admissible for~\eqref{eq:adm-mdr-p}: if
$\vec u^{m+1}=\vec\id|_{\Gm}$, then $\vec w^{m+1}=\vec0$
and~\eqref{eq:adm-mdr-p} holds with $\lambda^{m+1}=0$.  This gives the following
stability result.

\begin{thm}[unconditional stability]
Let
\[
(\vec u^{m+1},z^{m+1},\vec p^{m+1},\lambda^{m+1})\in
[\bV^h(\Gm)]^d\times\bV^h(\Gm)\times[\bV^h(\Gm)]^d\times\bV^h(\Gm)
\]
be a solution to~\eqref{eq:adm-mdr-system}. Then
\begin{equation}\label{eq:adm-mdr-energy}
|\Gamma^{m+1}|+\tau
\ipd{\nabs z^{m+1},\nabs z^{m+1}}_{\Gm}
\leq |\Gm|\qquad 0\leq m\leq M-1.
\end{equation}
\end{thm}

\begin{proof}
We test~\eqref{eq:adm-mdr-z} with $\zeta^h=z^{m+1}$,~\eqref{eq:adm-mdr-u} with
$\vec\eta^h=\vec u^{m+1}-\vec\id|_{\Gm}$, and use~\eqref{eq:adm-mdr-p} with
$\vec y^h=\vec p^{m+1}$ together with~\eqref{eq:adm-mdr-lambda} with
$\mu^h=\lambda^{m+1}$.  The relaxed
tangential terms cancel, and we obtain
\[
\tau\ipd{\nabs z^{m+1},\nabs z^{m+1}}_{\Gm}
+\ipd{\nabs\vec u^{m+1},
\nabs[\vec u^{m+1}-\vec\id]}_{\Gm}=0.
\]
Using the standard surface-area inequality in~\eqref{eq:surfaceinq} then
gives~\eqref{eq:adm-mdr-energy}.
\end{proof}

\section{Further insights and extensions}
\label{sec:further-insights}

We briefly discuss two possible extensions of the minimizing-movement framework.

First, for surface diffusion flow, it is often desirable to enforce exact volume
preservation at the fully discrete level. This can be achieved by modifying the fully
discrete parametric finite element approximation. For example, in the admissible BGN
scheme~\eqref{eq:adm-bgn-system}, one may replace
$\vec\nu_{h,{\rm rd}}^m$ in~\eqref{eq:adm-bgn-z} and~\eqref{eq:adm-bgn-u}
by the time-weighted discrete normal
$\vec\nu^{m+\frac12}$ introduced in~\cite{BZ21SPFEM}.  This normal satisfies
\begin{equation}\label{eq:timeweight}
\vol(\Gamma^{m+1})-\vol(\Gamma^m)
=\ipd{(\vec u^{m+1}-\vec\id)\cdot
\vec\nu^{m+\frac12},1}_{\Gm}^h,
\quad
\Gamma^{m+1}=\vec u^{m+1}(\Gm).
\end{equation}
With this replacement, choosing $\zeta^h=1$ in the modified version
of~\eqref{eq:adm-bgn-z} and using~\eqref{eq:timeweight} gives
$\vol(\Gamma^{m+1})=\vol(\Gamma^m)$.  The same algebraic testing argument
still yields the energy estimate, because the same weighted normal is used in the two
equations and the cancellation between the $z^{m+1}$-terms is preserved. The price is
that $\vec\nu^{m+\frac12}$ depends on the new surface. The resulting linear system is
therefore replaced by a nonlinear polynomial system, whose solvability would require a
separate analysis.

Second, the framework can also be adapted to anisotropic geometric flows. Let
$\gamma:\bR^d\to\bR_{\geq0}$ be a convex, positively one-homogeneous
anisotropy and consider the anisotropic surface energy
\[
E_\gamma(\Gamma)=\int_{\Gamma}\gamma(\vec\nu)\,\dH^{d-1}.
\]
The corresponding $L^2$- and $H^{-1}$-gradient flows are anisotropic mean curvature flow
and anisotropic surface diffusion flow, respectively; see~\cite{BGN08Ani,BGN08ani3d,
BLani23}. In the present framework, one would replace the isotropic Dirichlet energy in
$J$ by a suitable discrete anisotropic energy. For example, in view of the work
in~\cite{BLani23}, the anisotropic Dirichlet energy could take the form
\begin{equation}
\int_{\Gamma}\nabs\vec u\cdot(\mathbf{Z}(\vec\nu)\nabs\vec u)\,\dH^{d-1},
\end{equation}
where $\mathbf{Z}(\vec\nu)$ is the symmetric positive definite matrix defined
in~\cite[(2.1)]{BLani23}, which satisfies an anisotropic energy inequality of the form
\[
E_\gamma^h(\Gamma^{m+1})-E_\gamma^h(\Gm)
\leq \ipd{\mathbf{Z}(\vec\nu^m)\nabs\vec u^{m+1},~\nabs(\vec u^{m+1}-\vec\id)}_{\Gm}.
\]
Further investigation of anisotropic geometric flows within the minimizing-movement
framework is left for future work.

\section{Numerical results}
\label{sec:numerics}

In this section, we compare four fully discrete schemes:
\begin{itemize}
\item the BGN scheme in~\eqref{eq:bgn-system};
\item the dual-MDR scheme in~\eqref{eq:dual-mdr-system}, abbreviated below
as MDR;
\item the admissible BGN scheme in~\eqref{eq:adm-bgn-system}, abbreviated below as ad-BGN;
\item the relaxed MDR scheme in~\eqref{eq:adm-mdr-system}, abbreviated below as r-MDR.
\end{itemize}
The schemes are implemented using the open-source finite element package
NGSolve~\cite{schoberl2014c++}. The resulting sparse linear systems are solved with the
UMFPACK direct solver~\cite{Davis04}. Unless stated otherwise, the relaxation parameter
in the r-MDR scheme is fixed at
$\alpha=10$.  To quantify the quality of the polyhedral mesh, we use the
elementwise mesh-quality indicator
\[
  \mathsf{r}(\sigma) = \frac{R_*(\sigma)}{2r_*(\sigma)},\qquad \sigma\in\mT^m,
\]
where $R_*(\sigma)$ denotes the circumradius of the triangle $\sigma$, and
$r_*(\sigma)$ denotes its inradius, so that $2r_*(\sigma)$ is the diameter
of the inscribed ball. We report two mesh-quality statistics: $\mathsf{r}_{\rm p95}^m$,
the largest value remaining after
discarding the worst $5\%$ of the elements, and $\mathsf{r}_{\max}^m$, the
maximum value over all elements. Geometric errors between two closed surfaces
$\Gamma_1$ and $\Gamma_2$ are measured by the symmetric-difference volume
(see~\cite{Zhao2021energy})
\begin{equation}
  \mathrm{M}(\Gamma_1,\Gamma_2)
  := |\Omega_1\triangle\Omega_2|
  = |\Omega_1| + |\Omega_2| - 2|\Omega_1\cap\Omega_2|,
  \label{eq:symdiff-volume}
\end{equation}
where $\Omega_1$ and $\Omega_2$ are the regions enclosed by $\Gamma_1$ and $\Gamma_2$,
respectively, and $|\Omega|$ denotes the volume of the region $\Omega$.

\subsection{Spherical convergence test}
We first consider mean curvature flow for a shrinking sphere, for which the exact
solution is given by
\begin{equation}
\Gamma(t)=r(t)\mathbb{S}^2,\quad r(t)=\sqrt{1-4t},\quad \varkappa=-\frac{2}{r(t)},\qquad 0\leq t\leq 0.25.
\end{equation}
The computation is carried out up to $T=0.24$ on a sequence of successively refined
space-time discretizations,
\begin{align*}
(J,\tau)\in\{&(324,10^{-2}),\ (1296,2.5\times10^{-3}),\\
&(5184,6.25\times10^{-4}),\ (20736,1.5625\times10^{-4})\}.
\end{align*}
We measure the numerical error using the discrete nodal $L^2$ error
\[
\norm{\vec u^h-\vec u}_{L^2}(T)
=\sqrt{1/K\sum_{\vec q\in\mQ^M}
\left|\vec u^M(\vec q)-r(T)\vec u^0(\vec q)\right|^2},
\]
and the symmetric-difference volume $\mathrm{M}(\Gamma^M,\Gamma(T))$. We also monitor
the maximum radial position error
\begin{align*}
\norm{\vec u^h-\vec u}_{L^\infty} &=
\max_{0\leq m\leq M}\max_{\vec q\in \mQ^m}
\bigl||\vec u^m(\vec q)| - r(t_m)\bigr|.
%\\
%\norm{\varkappa^h-\varkappa}_{L^\infty} &=
%\max_{0\leq m\leq M}\max_{\vec q\in \mQ^m}
%\left|\varkappa^m(\vec q) + \frac{2}{r(t_m)}\right|.
\end{align*}

\begin{table}[tbp]
\centering
\small
\caption{Errors and convergence orders for the shrinking sphere under mean curvature flow.}
\label{tab:mcf-exp51-position-summary}
\begin{tabular}{l r c c c c c c}
\toprule
\multirow{2}{*}{Scheme} & \multirow{2}{*}{$J$}
& \multicolumn{2}{c}{$\norm{\vec u^h - \vec u}_{L^2}(T)$}
& \multicolumn{2}{c}{$\norm{\vec u^h - \vec u}_{L^\infty}$}
& \multicolumn{2}{c}{$\mathrm{M}(\Gamma^M,\Gamma(T))$} \\
\cmidrule(lr){3-4} \cmidrule(lr){5-6} \cmidrule(lr){7-8}
& & Error & Order & Error & Order & Error & Order \\
\midrule
\multirow{4}{*}{BGN}
&   324 & $7.46\mathrm{E}{-2}$ & --   & $7.75\mathrm{E}{-2}$ & --   & $4.74\mathrm{E}{-2}$ & --\\
&  1296 & $3.17\mathrm{E}{-2}$ & 1.24 & $2.79\mathrm{E}{-2}$ & 1.48 & $1.43\mathrm{E}{-2}$ & 1.72\\
&  5184 & $1.89\mathrm{E}{-2}$ & 0.75 & $8.14\mathrm{E}{-3}$ & 1.78 & $3.83\mathrm{E}{-3}$ & 1.91\\
& 20736 & $1.72\mathrm{E}{-2}$ & 0.13 & $2.16\mathrm{E}{-3}$ & 1.92 & $9.70\mathrm{E}{-4}$ & 1.98\\
\midrule
\multirow{4}{*}{MDR}
&   324 & $6.97\mathrm{E}{-2}$ & --   & $7.48\mathrm{E}{-2}$ & --   & $4.48\mathrm{E}{-2}$ & --\\
&  1296 & $2.50\mathrm{E}{-2}$ & 1.48 & $2.66\mathrm{E}{-2}$ & 1.49 & $1.36\mathrm{E}{-2}$ & 1.72\\
&  5184 & $7.21\mathrm{E}{-3}$ & 1.79 & $7.67\mathrm{E}{-3}$ & 1.79 & $3.63\mathrm{E}{-3}$ & 1.91\\
& 20736 & $1.87\mathrm{E}{-3}$ & 1.95 & $2.01\mathrm{E}{-3}$ & 1.93 & $9.19\mathrm{E}{-4}$ & 1.98\\
\midrule
\multirow{4}{*}{\shortstack[l]{r-MDR}}
&   324 & $6.97\mathrm{E}{-2}$ & --   & $7.48\mathrm{E}{-2}$ & --   & $4.48\mathrm{E}{-2}$ & --\\
&  1296 & $2.50\mathrm{E}{-2}$ & 1.48 & $2.66\mathrm{E}{-2}$ & 1.49 & $1.36\mathrm{E}{-2}$ & 1.72\\
&  5184 & $7.21\mathrm{E}{-3}$ & 1.79 & $7.67\mathrm{E}{-3}$ & 1.79 & $3.63\mathrm{E}{-3}$ & 1.91\\
& 20736 & $1.87\mathrm{E}{-3}$ & 1.95 & $2.01\mathrm{E}{-3}$ & 1.93 & $9.19\mathrm{E}{-4}$ & 1.98\\
\midrule
\multirow{4}{*}{ad-BGN}
&   324 & $6.96\mathrm{E}{-2}$ & --   & $7.48\mathrm{E}{-2}$ & --   & $4.49\mathrm{E}{-2}$ & --\\
&  1296 & $2.50\mathrm{E}{-2}$ & 1.48 & $2.66\mathrm{E}{-2}$ & 1.49 & $1.36\mathrm{E}{-2}$ & 1.72\\
&  5184 & $7.21\mathrm{E}{-3}$ & 1.79 & $7.67\mathrm{E}{-3}$ & 1.79 & $3.63\mathrm{E}{-3}$ & 1.91\\
& 20736 & $1.87\mathrm{E}{-3}$ & 1.95 & $2.01\mathrm{E}{-3}$ & 1.93 & $9.19\mathrm{E}{-4}$ & 1.98\\
\bottomrule
\end{tabular}
\end{table}

As shown in~\Cref{tab:mcf-exp51-position-summary}, the radial position errors and the
symmetric-difference errors converge at nearly second order for all four schemes. The
nodal
$L^2$ error behaves differently because it is sensitive to the parametrization
of the surface. This effect is most apparent for BGN: its nodal error stagnates under
refinement, whereas the geometric error continues to converge at the expected rate.

\subsection{MCF of a dumbbell surface}

\begin{figure}[t]
  \centering
    \includegraphics[width=0.45\textwidth]{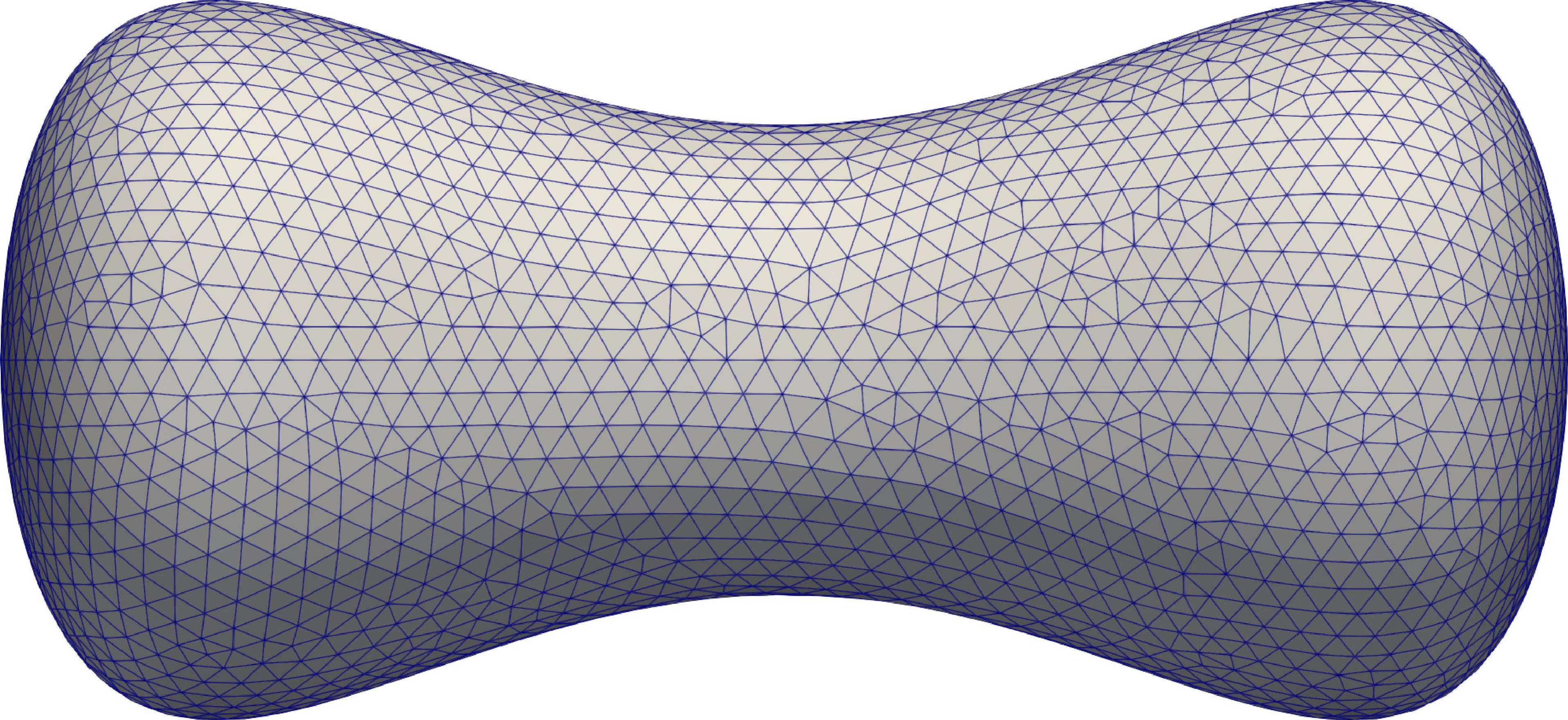}
  \\[2ex]
    \includegraphics[width=0.45\textwidth]{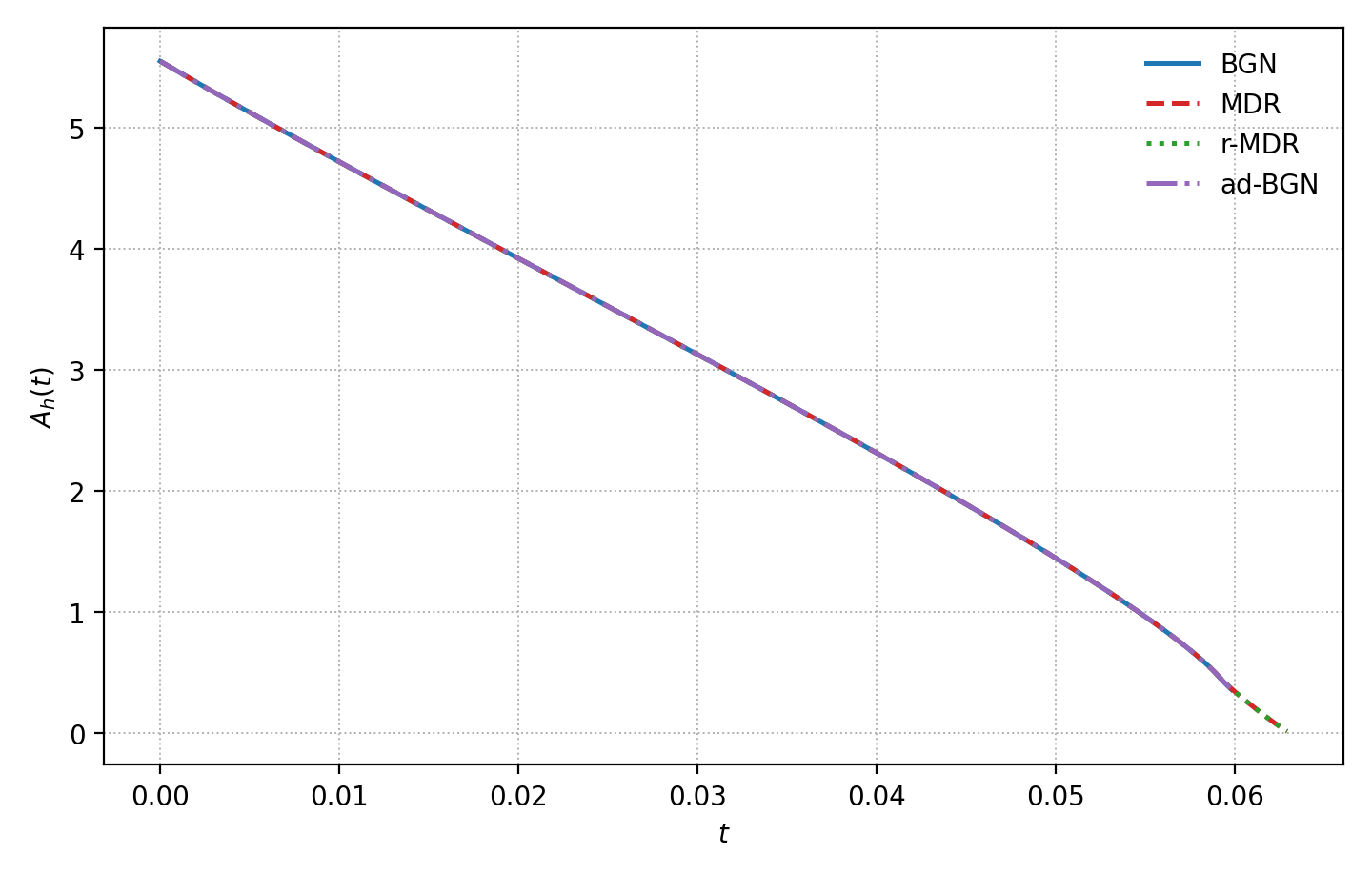}
  \hspace{1cm}
    \includegraphics[width=0.45\textwidth]{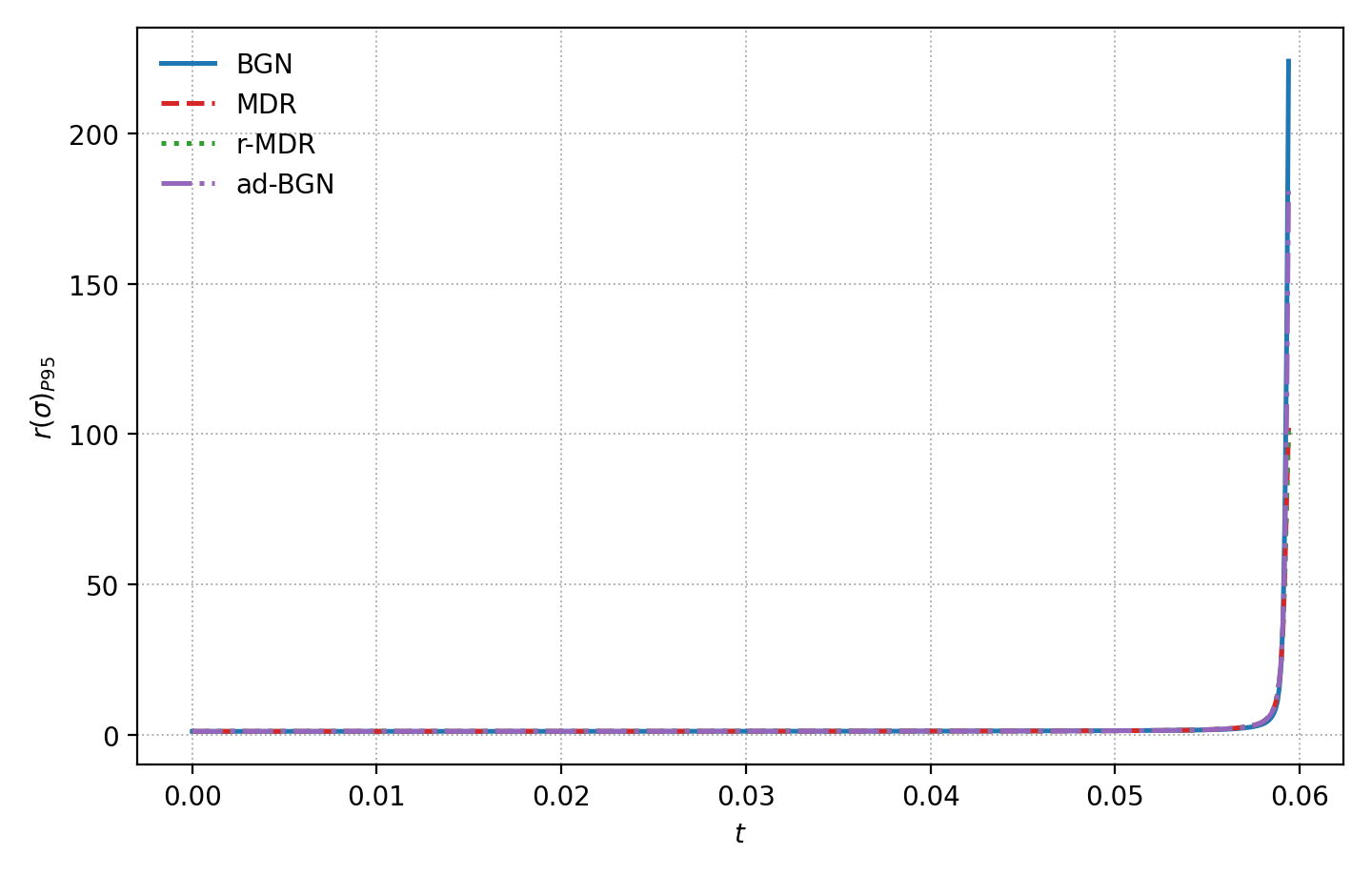}
  \caption{The initial polyhedral surface, surface-area decay, and the
  mesh-quality statistic $\mathsf{r}_{\rm p95}^m$ in the pinching dumbbell
  test.}
  \label{fig:exp52-area-p95}
\end{figure}

We next consider mean curvature flow of dumbbell-shaped surfaces. The initial surface is
parametrized by
\begin{equation}\label{eq:dumb}
\vec u_0(\theta,\varphi) = \begin{pmatrix}
    \cos\varphi\\[0.4em]
    (a\cos^2\varphi+b)\cos\theta\sin\varphi\\[0.4em]
    (a\cos^2\varphi+b)\sin\theta\sin\varphi
  \end{pmatrix},\quad
  \theta\in[0,2\pi),\quad \varphi\in[0,\pi],
\end{equation}
where $a\geq0$ and $b\geq0$ are parameters to control the neck thickness.

In the first experiment, we take $a=0.7$ and $b=0.3$. The initial surface is discretized
with parameters $(J,K)=(5632,2818)$, and the time step size is
$\tau=10^{-4}$.  For this choice of parameters, the smooth mean-curvature
evolution is expected to develop a topological pinch-off.

As shown in~\Cref{fig:exp52-area-p95}, the surface-area curves are nearly
indistinguishable before the pinching stage. The plot of
$\mathsf{r}_{\rm p95}^m$ is truncated at $t=0.0595$, before severe neck
degeneration occurs, so that mesh quality is compared only on nondegenerate meshes. At
$t=0.0597$, all four schemes still resolve the thin neck; see~\Cref{fig:exp52-t00597}.

The schemes differ mainly in how the fixed-connectivity polyhedral meshes behave near
the pinching regime. For the classical BGN scheme, the computation stops because of mesh
degeneration while the surface area is still approximately
$5\times10^{-1}$.  By contrast, MDR and r-MDR remain computable up to
$t=0.063$, by which time the area has fallen below $5\times10^{-3}$.  This
extended computability should be viewed as fixed-connectivity behavior near an
underresolved neck, rather than as a resolution of the topological change. The ad-BGN
scheme also extends the computable time beyond BGN, although it stops at an intermediate
stage. Representative late-stage profiles are shown in~\Cref{fig:exp52-late}.

\begin{figure}[!tbp]
  \centering
  \subfloat[BGN]{
    \includegraphics[width=0.45\textwidth]{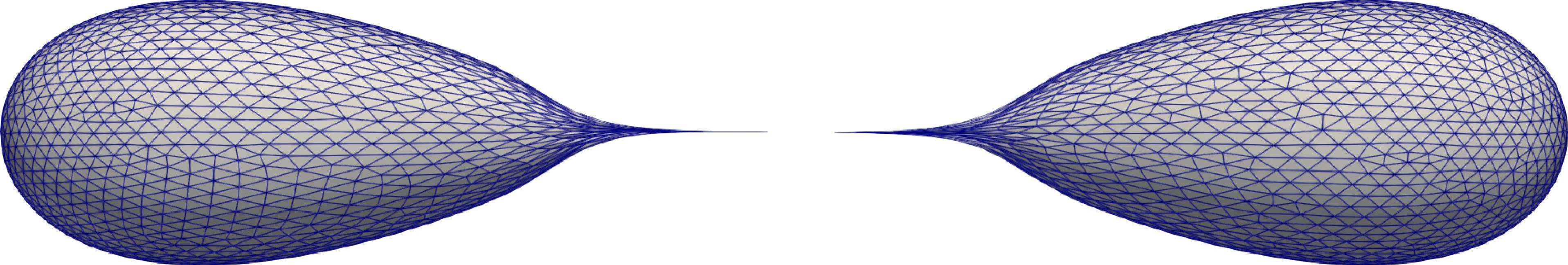}
  }
  \hfill
  \subfloat[ad-BGN]{
    \includegraphics[width=0.45\textwidth]{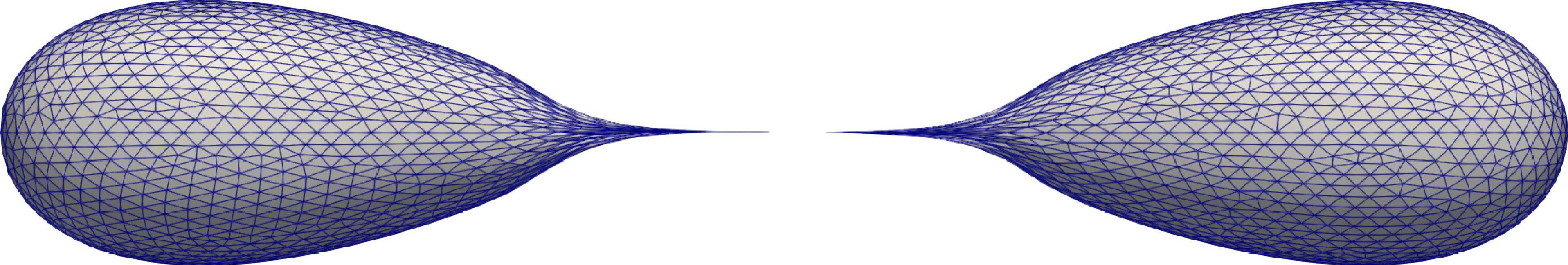}
  }
\\[2ex]
  \subfloat[MDR]{
    \includegraphics[width=0.45\textwidth]{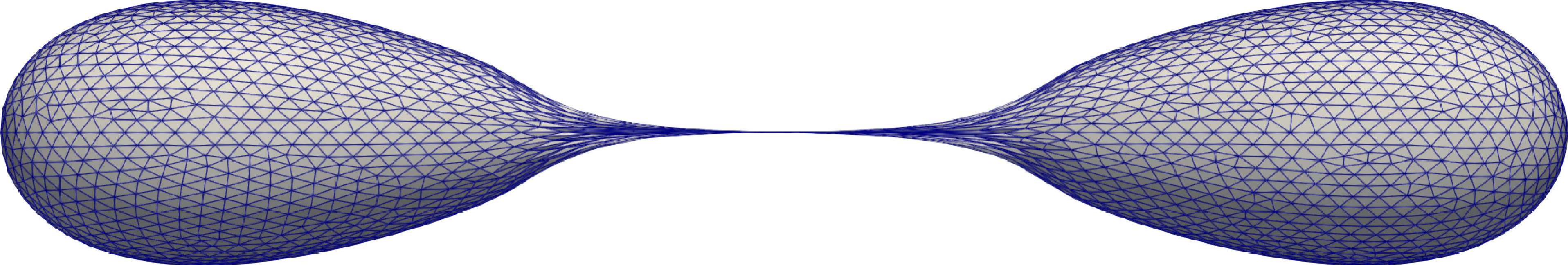}
  }
  \hfill
  \subfloat[r-MDR]{
    \includegraphics[width=0.45\textwidth]{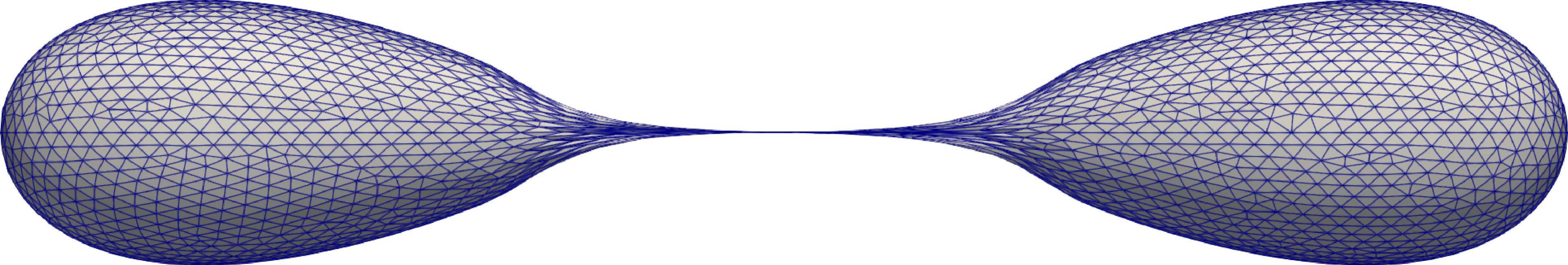}
  }
  \caption{Polyhedral surfaces at $t=0.0597$ in the pinching dumbbell test.}
  \label{fig:exp52-t00597}
\end{figure}

\begin{figure}[!tbp]
  \centering
  \subfloat[ad-BGN, $t=0.0599$]{
    \includegraphics[width=0.31\textwidth]{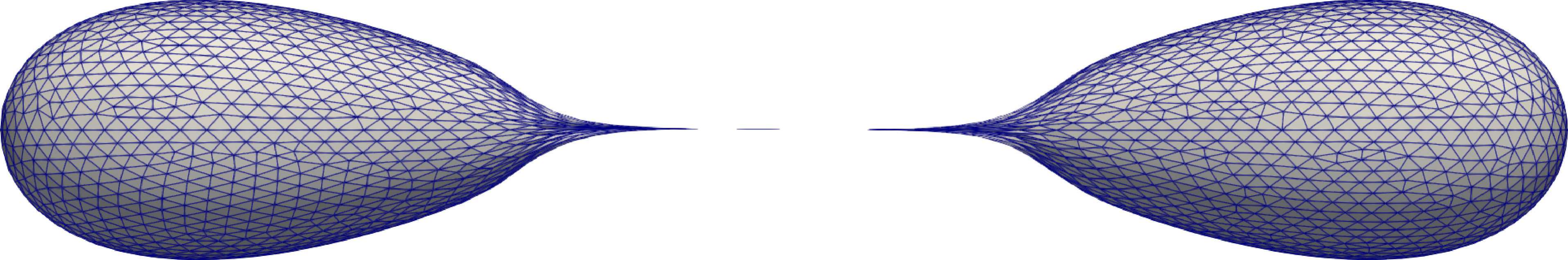}
  }
  \hfill
  \subfloat[MDR, $t=0.0611$]{
    \includegraphics[width=0.31\textwidth]{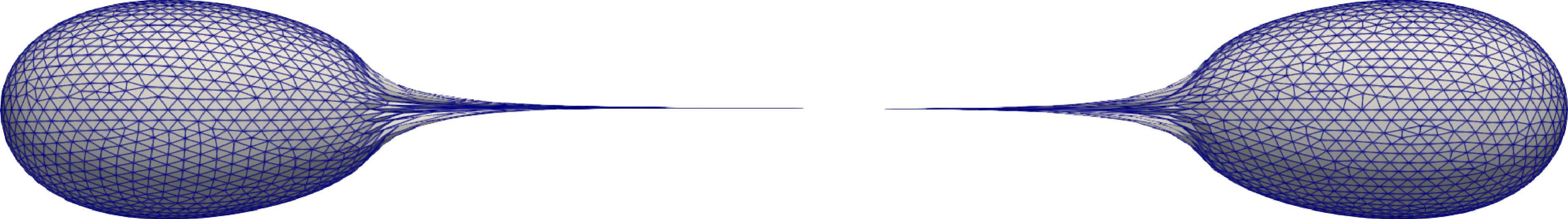}
  }
  \hfill
  \subfloat[r-MDR, $t=0.0611$]{
    \includegraphics[width=0.31\textwidth]{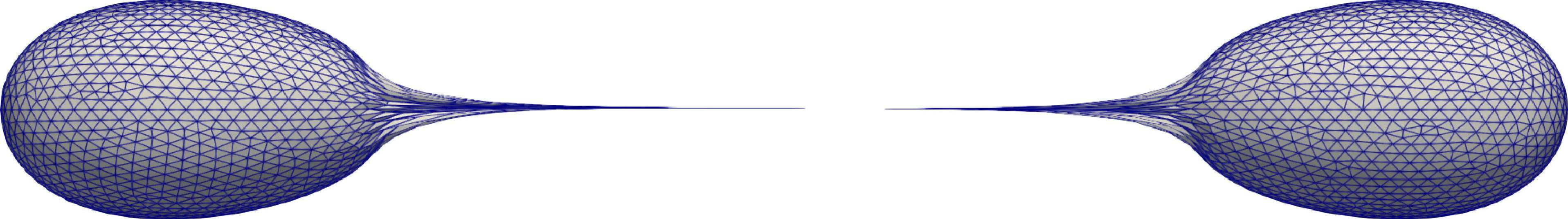}
  }
  \caption{Late-stage polyhedral surfaces for ad-BGN, MDR, and r-MDR.}
  \label{fig:exp52-late}
\end{figure}

\begin{figure}[htbp]
  \centering
  \subfloat[Initial mesh]{
    \includegraphics[width=0.48\textwidth]{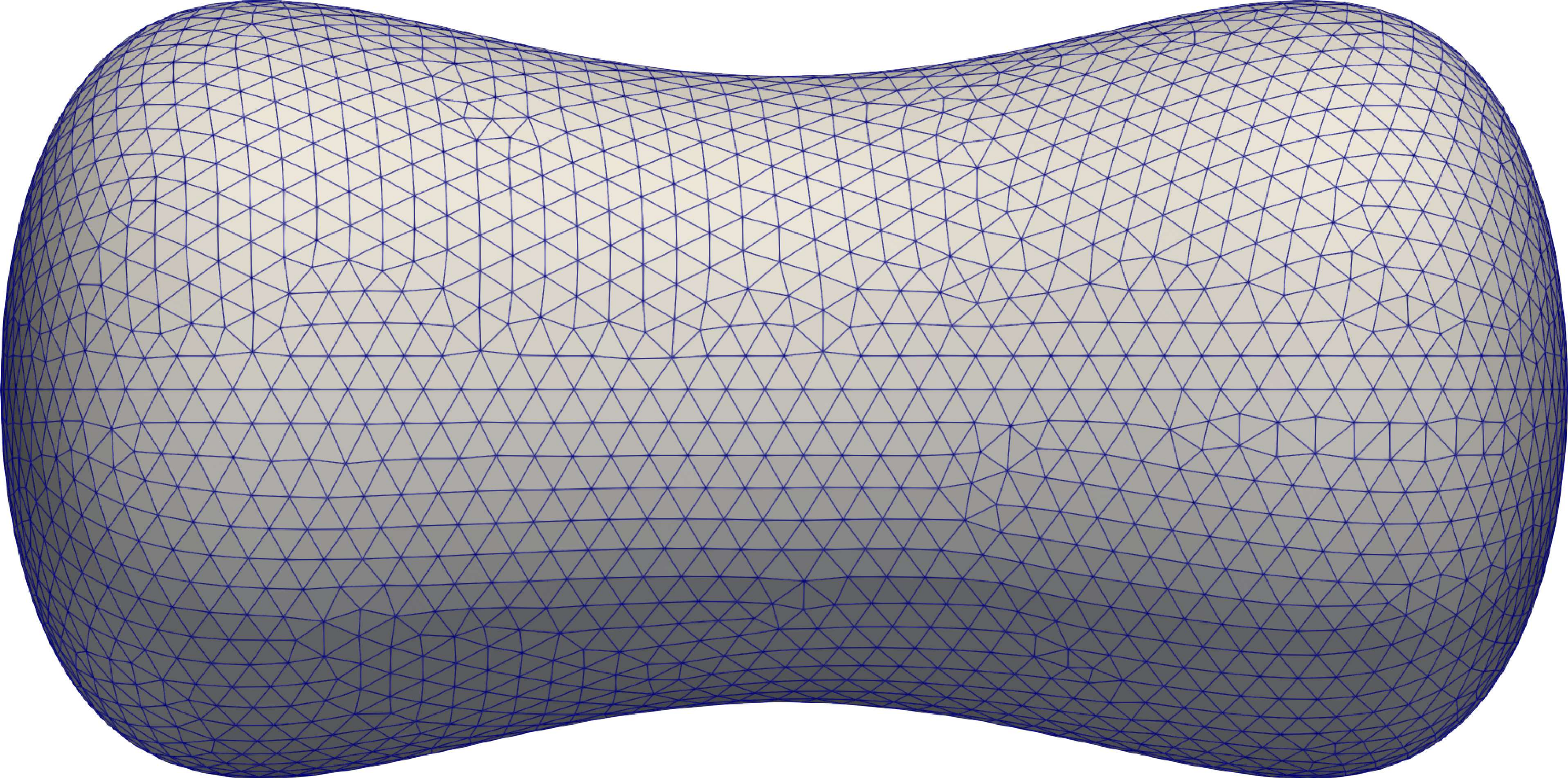}
  }%
  \subfloat[Surface area]{
    \includegraphics[width=0.48\textwidth]{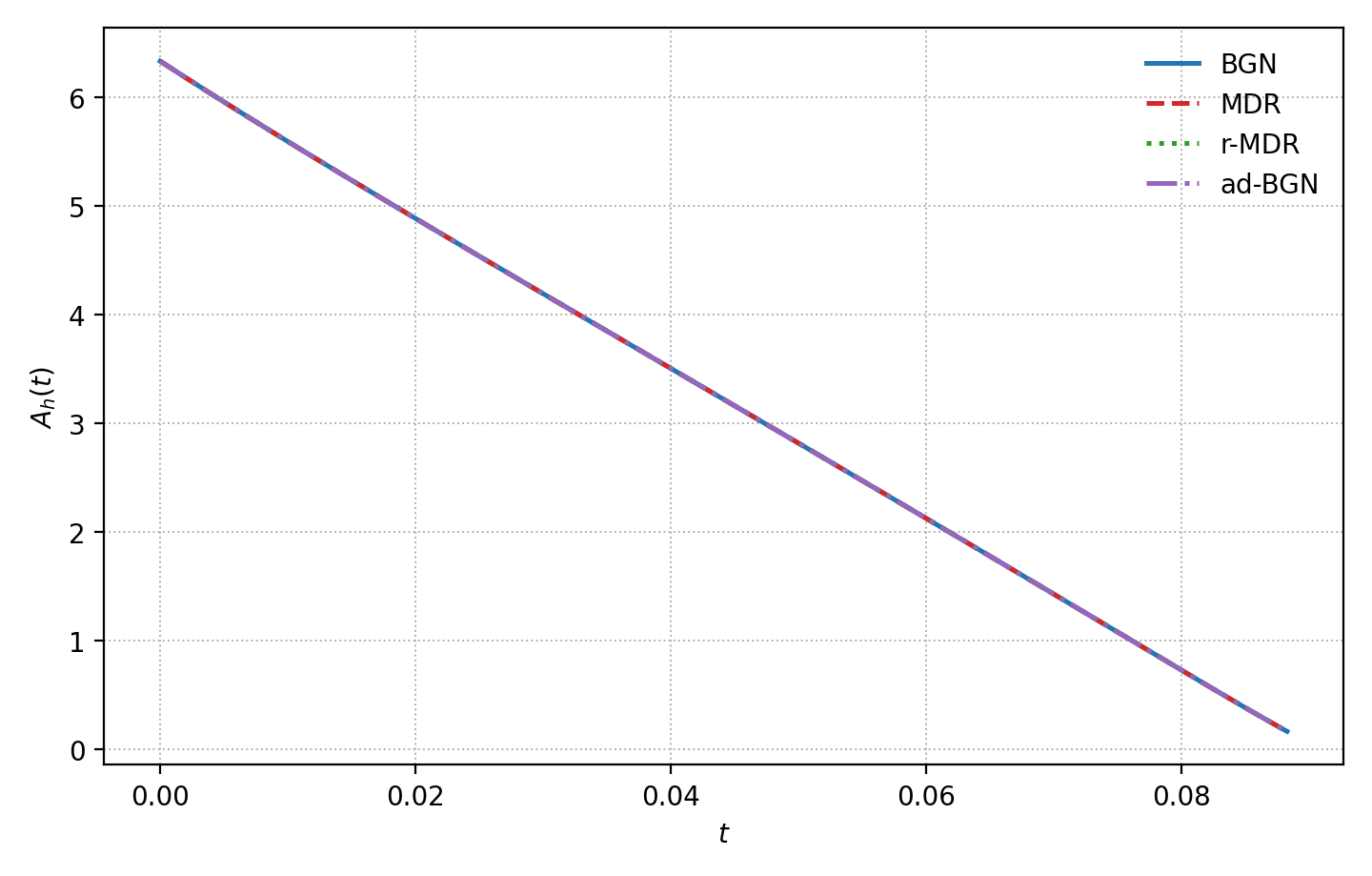}
  }\\[0pt]%
  \subfloat[$\mathsf{r}_{\rm p95}^m$]{
    \includegraphics[width=0.48\textwidth]{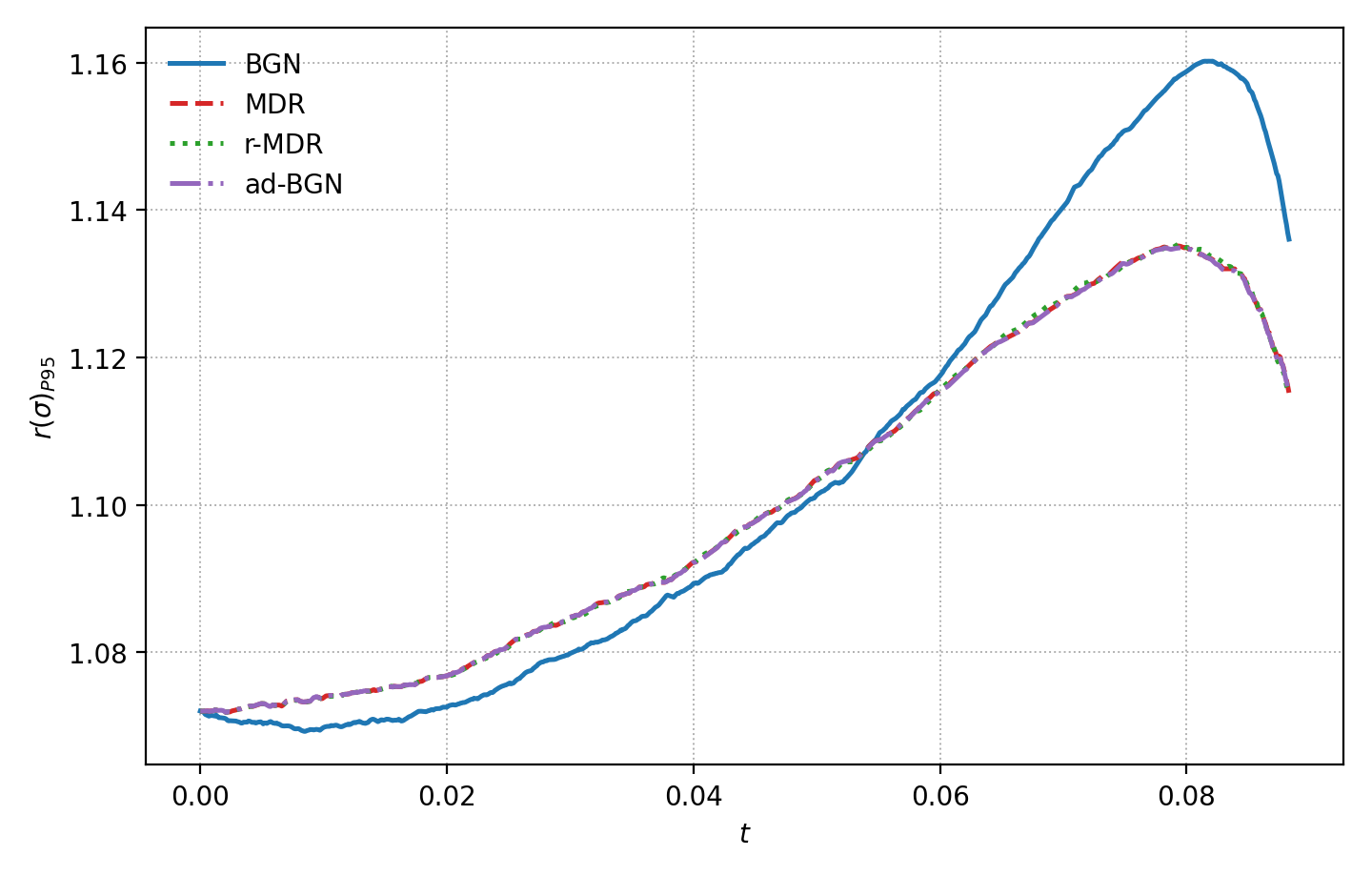}
  }%
  \subfloat[$\mathsf{r}_{\max}^m$]{
    \includegraphics[width=0.48\textwidth]{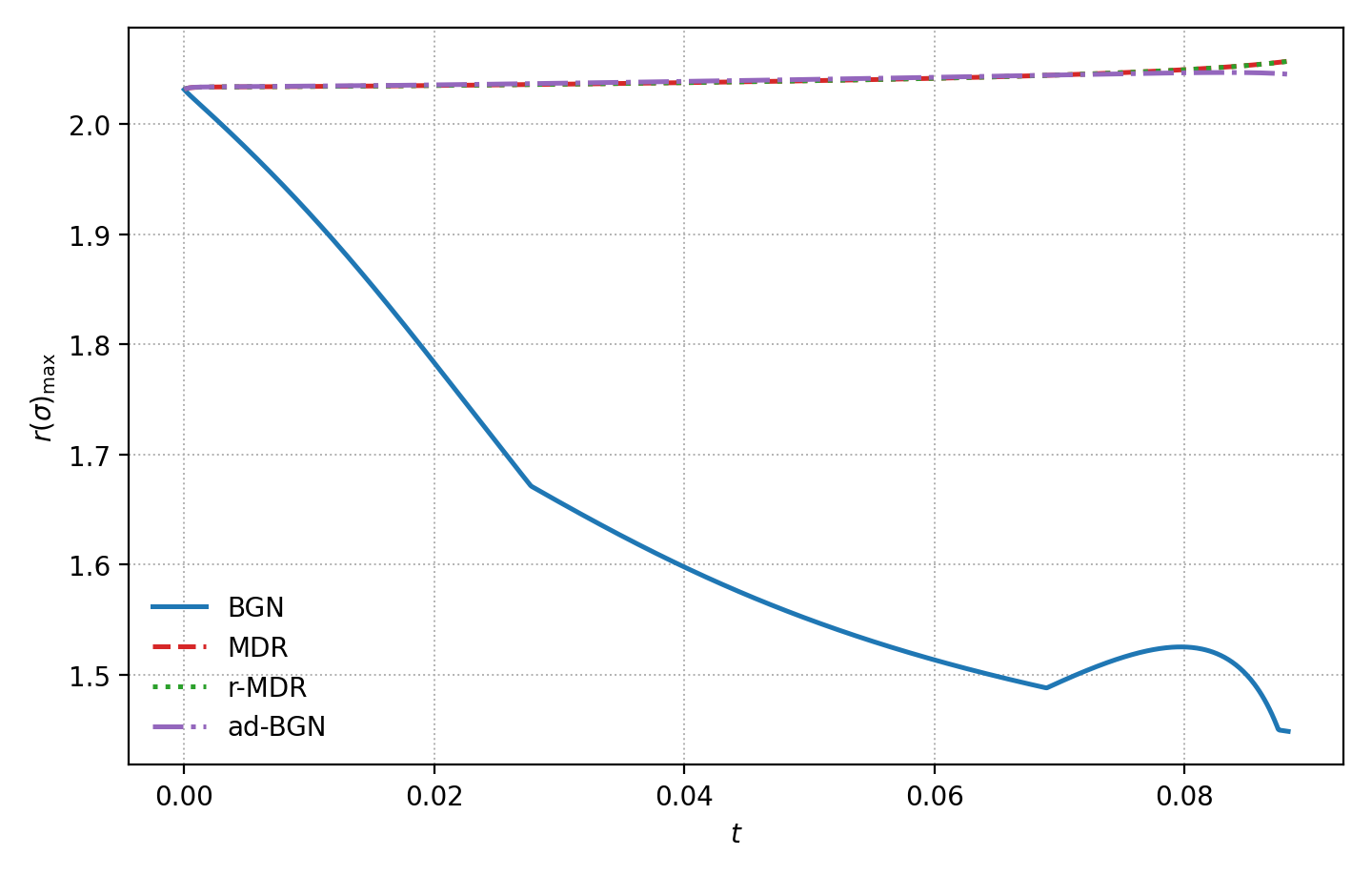}
  }
  \caption{Initial mesh, surface-area decay, $\mathsf{r}_{\rm p95}^m$, and
  $\mathsf{r}_{\max}^m$ in the thicker-neck dumbbell test.}
  \label{fig:exp53-area-quality}
\end{figure}

We then repeat the above experiment for a thicker-neck dumbbell by considering an
initial surface in~\eqref{eq:dumb} with $a=0.6$ and $b=0.4$. The discretization
parameters are $(J,K)=(6296,3150)$, and the time step size is again $\tau=10^{-4}$. For
this choice of parameters, the surface is expected to first become convex and then
shrink smoothly, without developing a topological pinch-off.

The numerical results are presented in~\Cref{fig:exp53-area-quality}, which shows
monotone area decay and acceptable mesh quality for all four schemes. The maximum
mesh-quality statistic is smallest for BGN, whereas the 95th-percentile statistic is
better controlled by MDR, r-MDR, and ad-BGN. Thus in this example BGN suppresses the
worst elements more effectively, while the other three schemes tend to improve the
quality of most elements.

\subsection{MCF of a perturbed torus}

\begin{figure}[!htbp]
  \centering
  \subfloat[Initial mesh]{
    \includegraphics[width=0.48\textwidth]{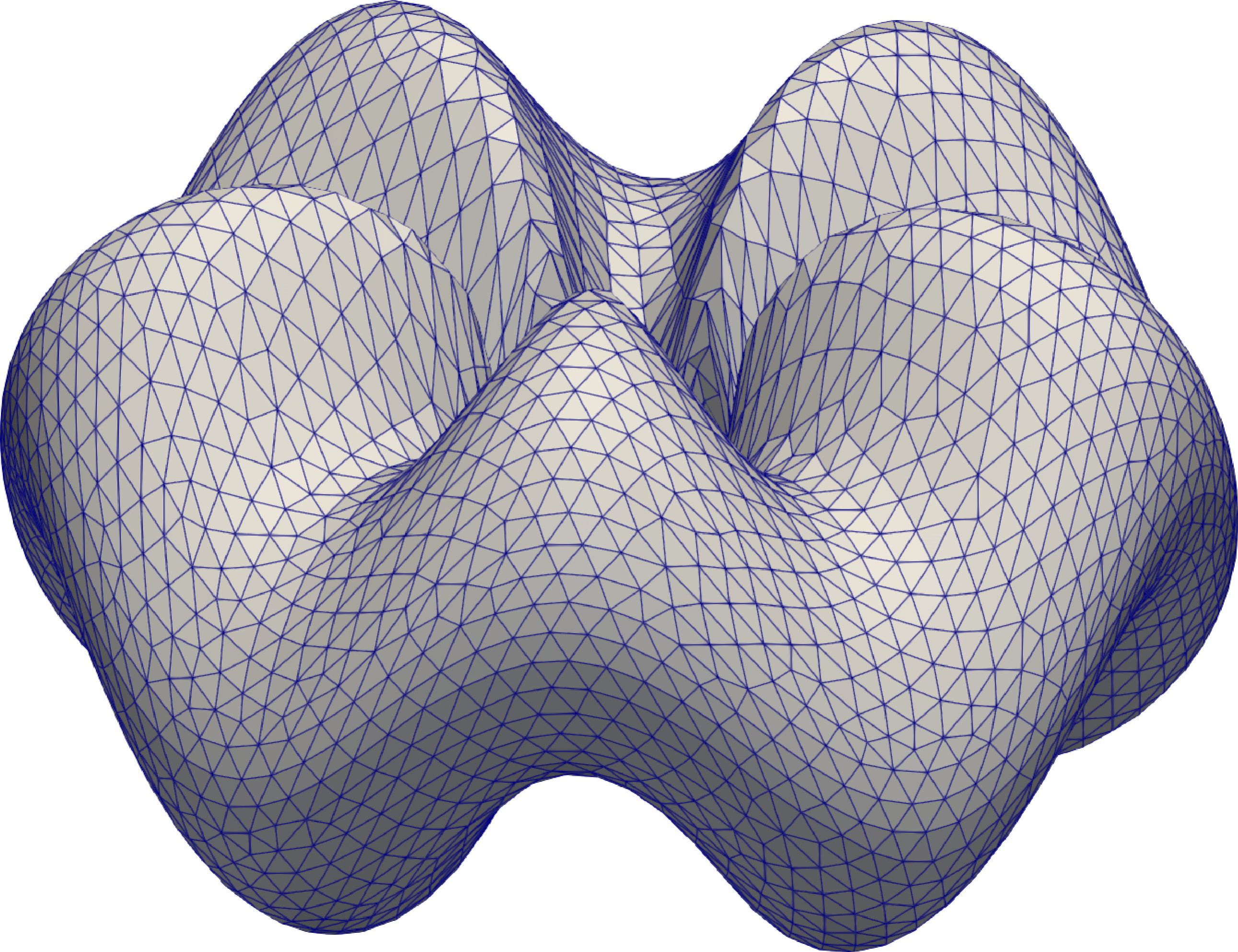}
  }
  \hfill
  \subfloat[Surface area]{
    \includegraphics[width=0.48\textwidth]{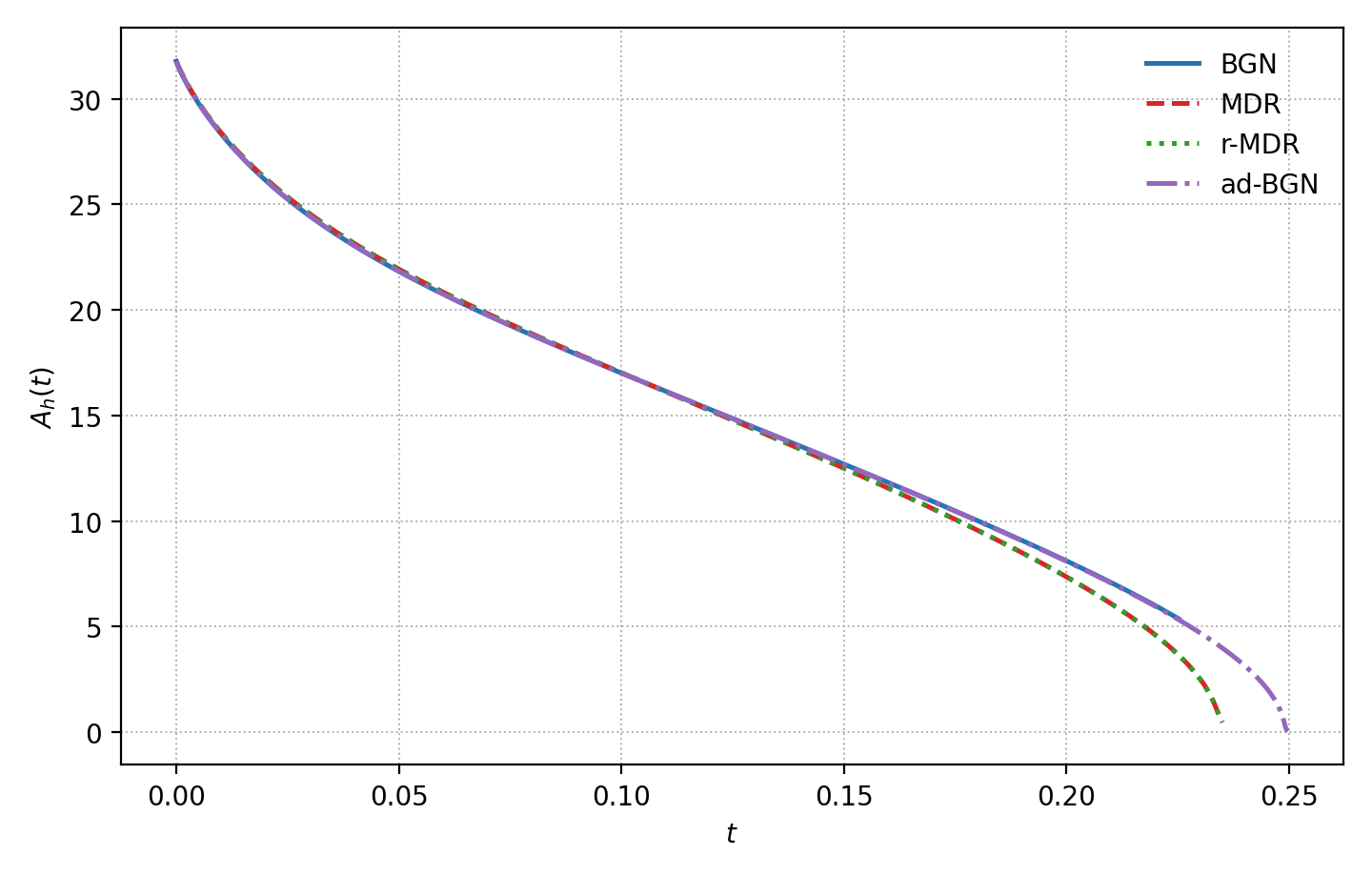}
  }
  \caption{Perturbed torus test: (a) the low-quality initial mesh and (b) the monotone decay of the surface area.}
  \label{fig:exp54-initial-area}
\end{figure}

\begin{figure}[!htbp]
  \centering
  \subfloat[BGN]{
    \includegraphics[width=0.4\textwidth]{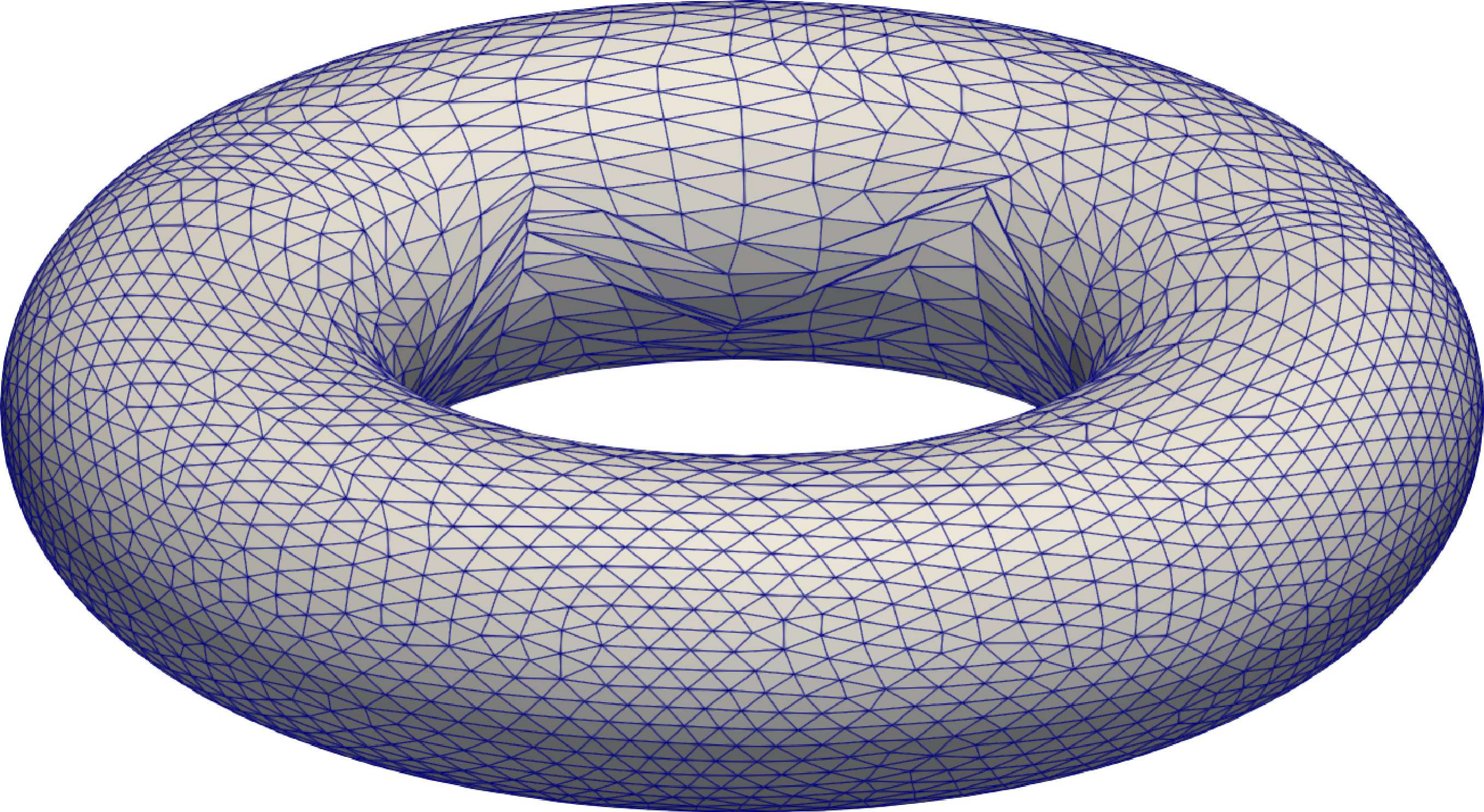}
  }
  \hspace{0.2em}
  \subfloat[ad-BGN]{
    \includegraphics[width=0.4\textwidth]{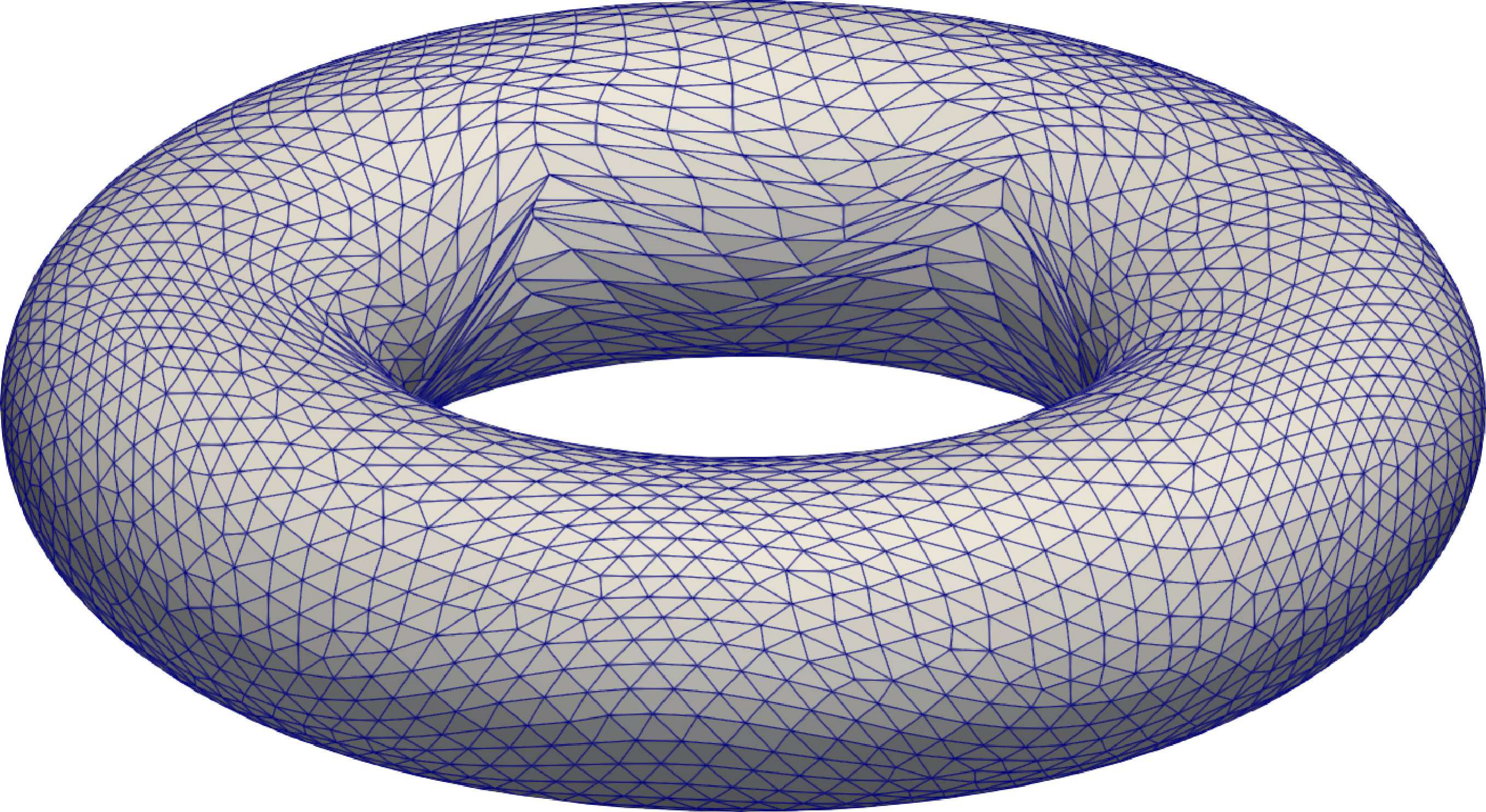}
  }
  \\[2ex]
  \subfloat[MDR]{
    \includegraphics[width=0.4\textwidth]{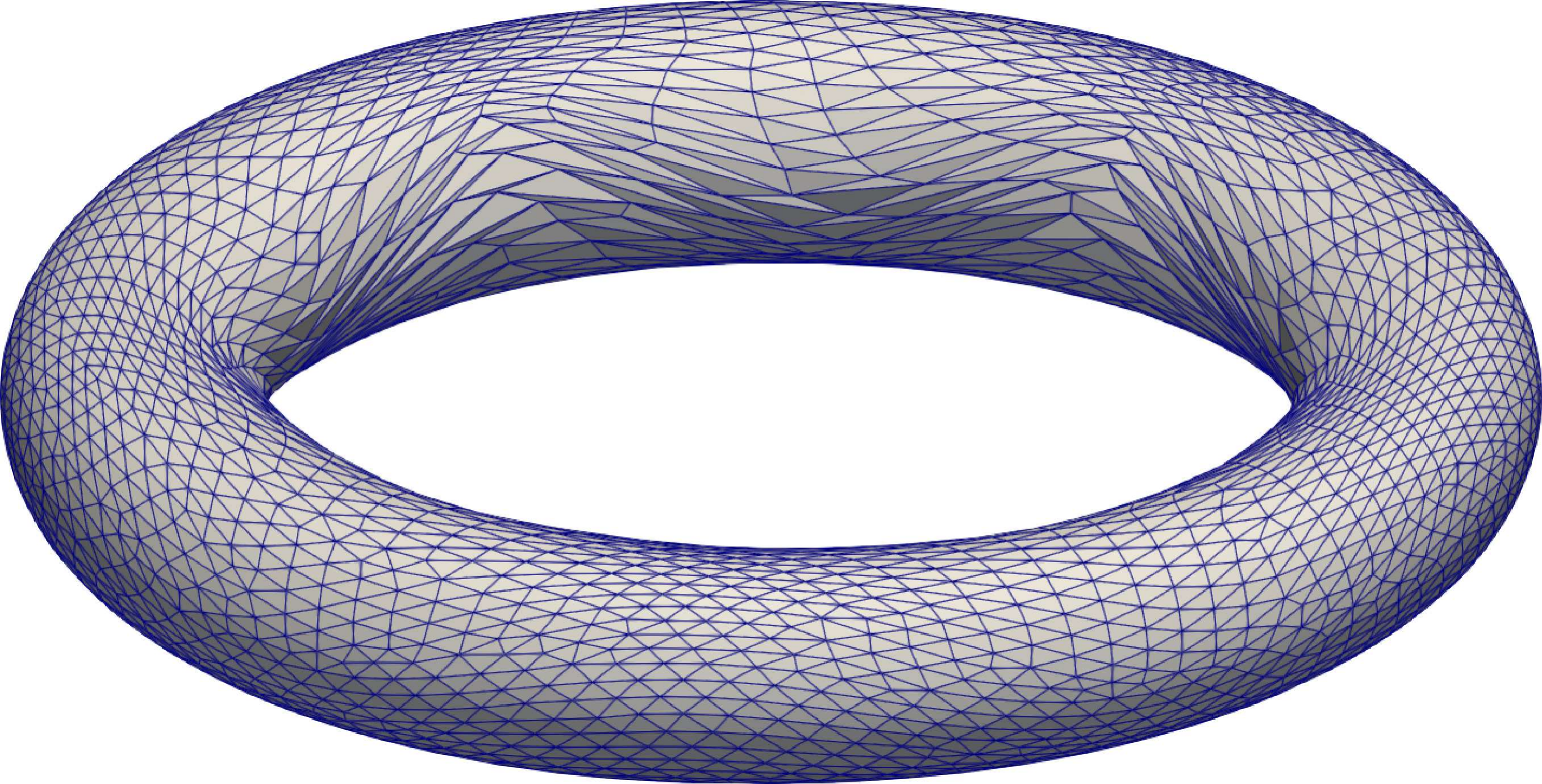}
  }
 \hspace{0.2em}
  \subfloat[r-MDR]{
    \includegraphics[width=0.4\textwidth]{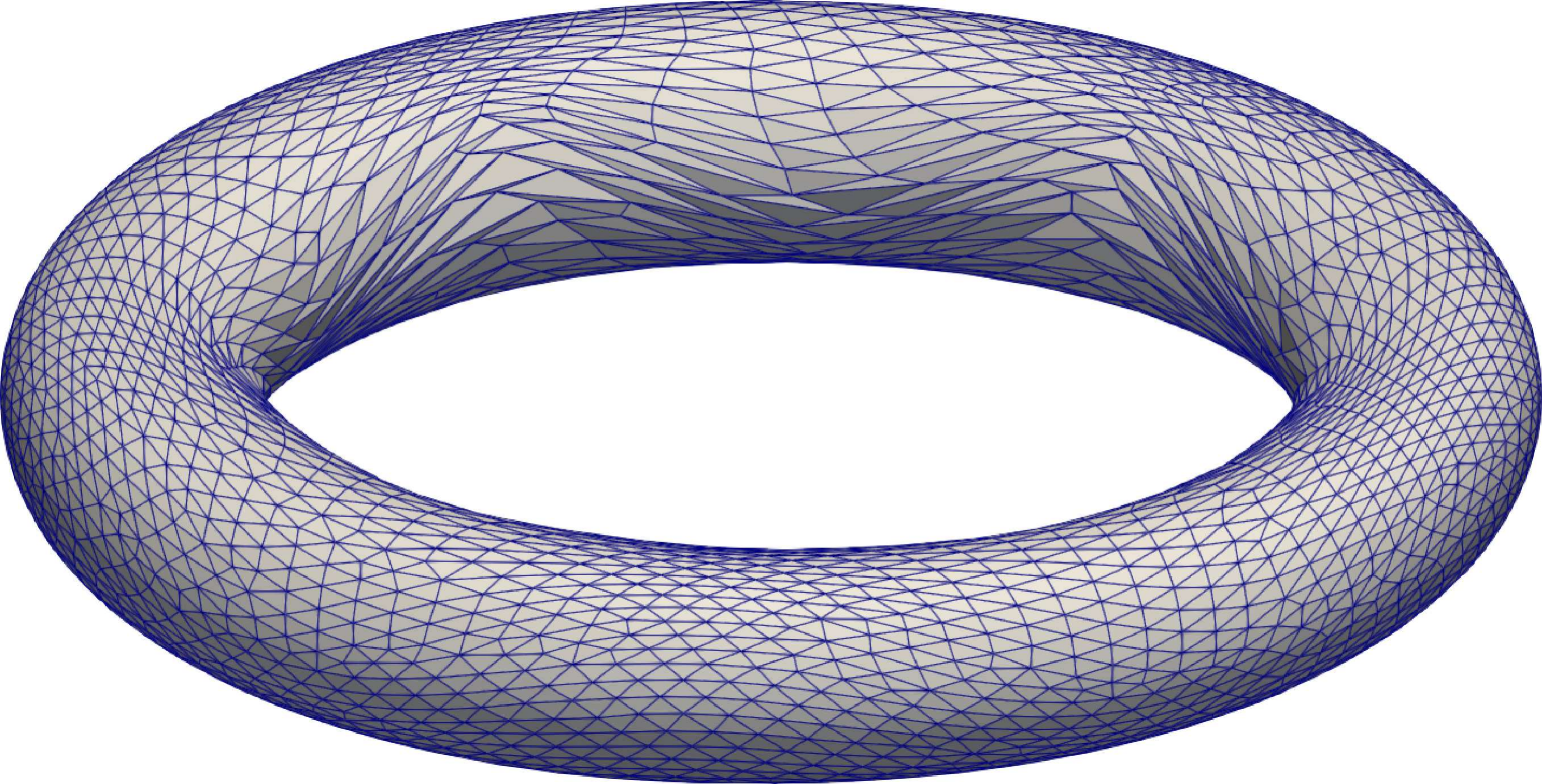}
  }
  \caption{The polyhedral surfaces at $t=0.2252$ in the perturbed torus test.}
  \label{fig:exp54-t02252}
\end{figure}

We next test the schemes on a strongly perturbed torus with a deliberately low-quality
initial mesh. The initial surface is parametrized by
\[
  \vec u_0(\theta,\varphi) = \begin{pmatrix}
    (1 + 0.65\cos\varphi)\cos\theta\\[0.2em]
    (1 + 0.65\cos\varphi)\sin\theta\\[0.2em]
    0.65\sin\varphi + 0.3\sin(5\theta)
  \end{pmatrix},
  \qquad \theta,\varphi\in[0,2\pi).
\]
The discretization parameters are $(J,K)=(6220,3110)$, and the time step size is fixed
at $\tau=10^{-4}$.

As illustrated in~\Cref{fig:exp54-initial-area}(b), all four schemes exhibit the
expected monotone decay of surface area, while significant differences appear after
approximately $t=0.15$. These differences are much clearer in~\Cref{fig:exp54-t02252},
where we visualize the polyhedral surfaces at
$t=0.2252$.  The comparison indicates that tangential redistribution affects
not only mesh quality but also the long-time discrete evolution. In particular, near the
late stage of extinction of the thin torus, the MDR scheme develops a visibly nonuniform
contraction, whereas the torus computed by the ad-BGN scheme shrinks more uniformly; see
\Cref{fig:exp54-near-zero}.

\begin{figure}[!htbp]
  \centering
  \subfloat[MDR, $t=0.2330$]{
    \includegraphics[width=0.30\textwidth]{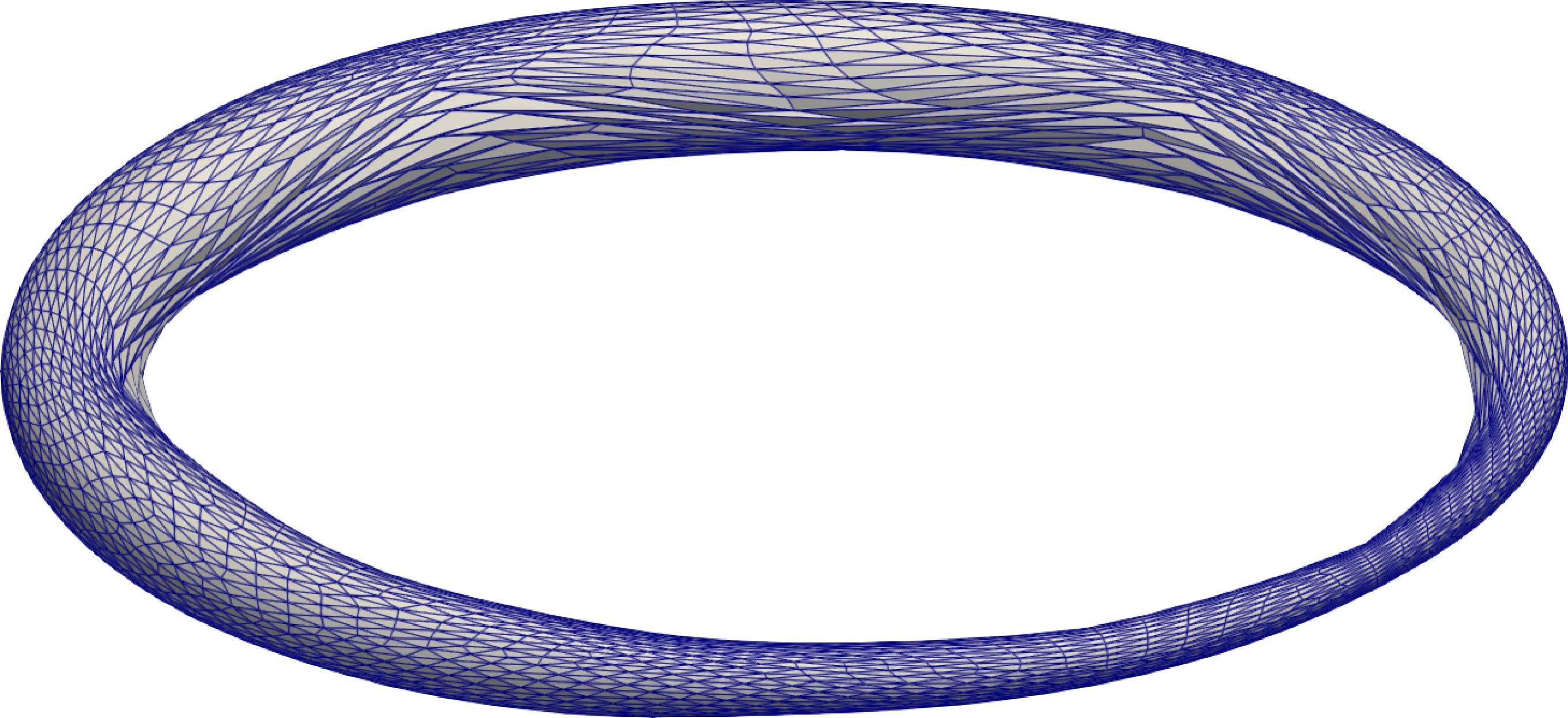}
  }
  \hfill
  \subfloat[MDR, $t=0.2340$]{
    \includegraphics[width=0.30\textwidth]{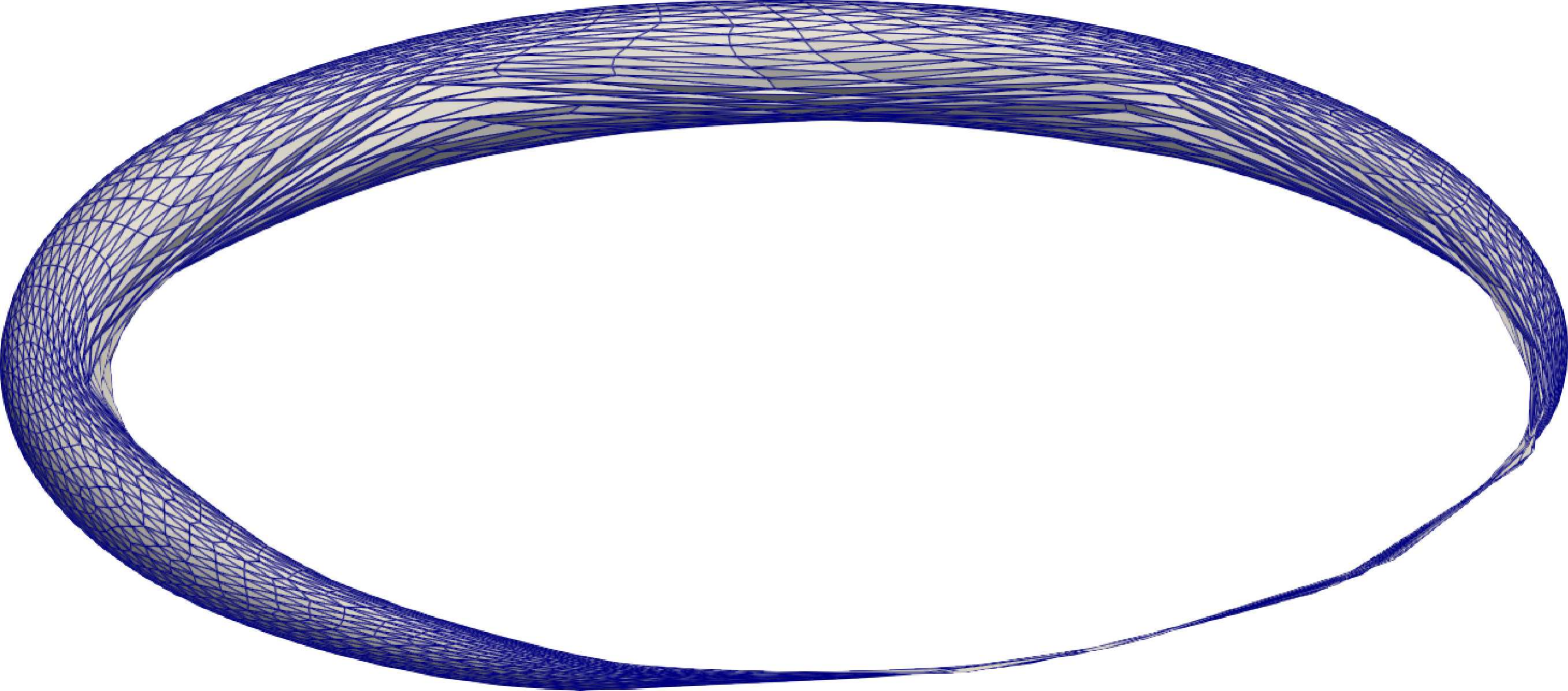}
  }
  \hfill
  \subfloat[MDR, $t=0.2350$]{
    \includegraphics[width=0.30\textwidth]{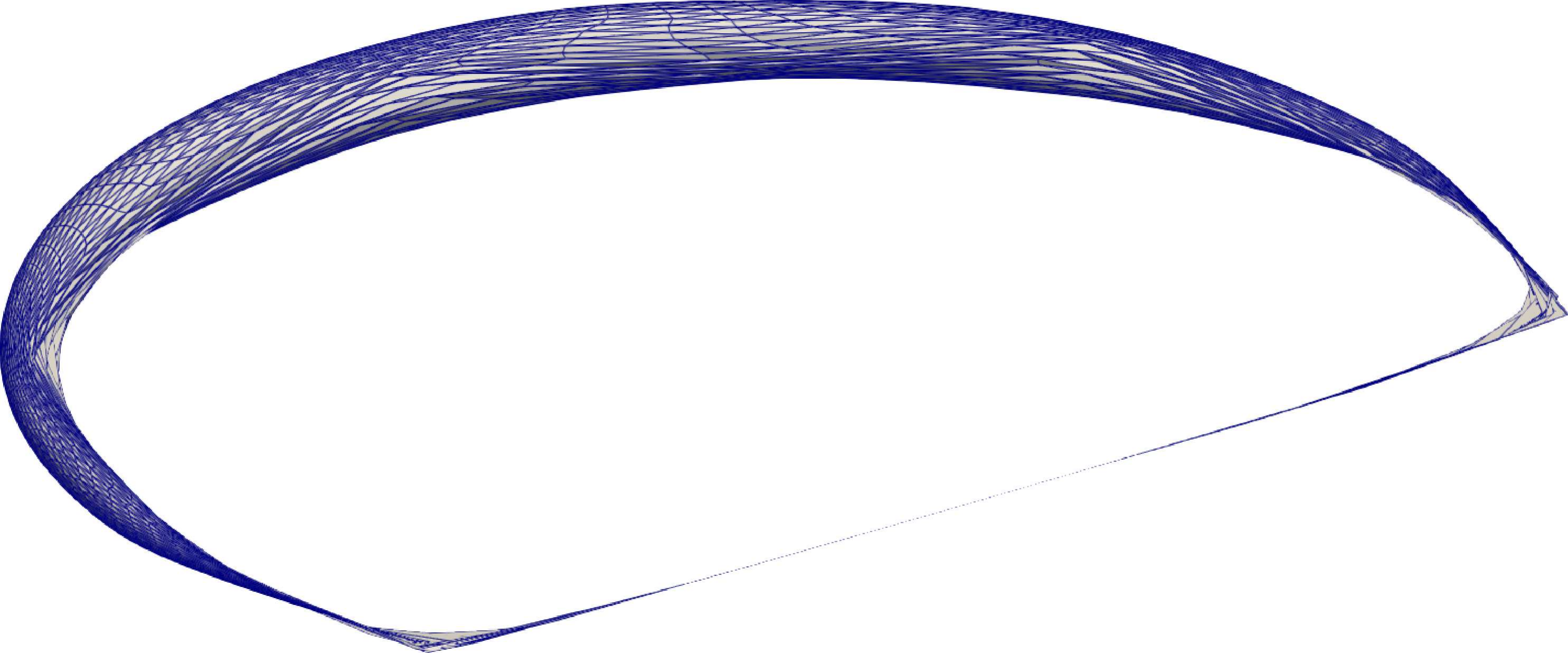}
  }\\[2ex]
  \subfloat[ad-BGN, $t=0.2476$]{
    \includegraphics[width=0.30\textwidth]{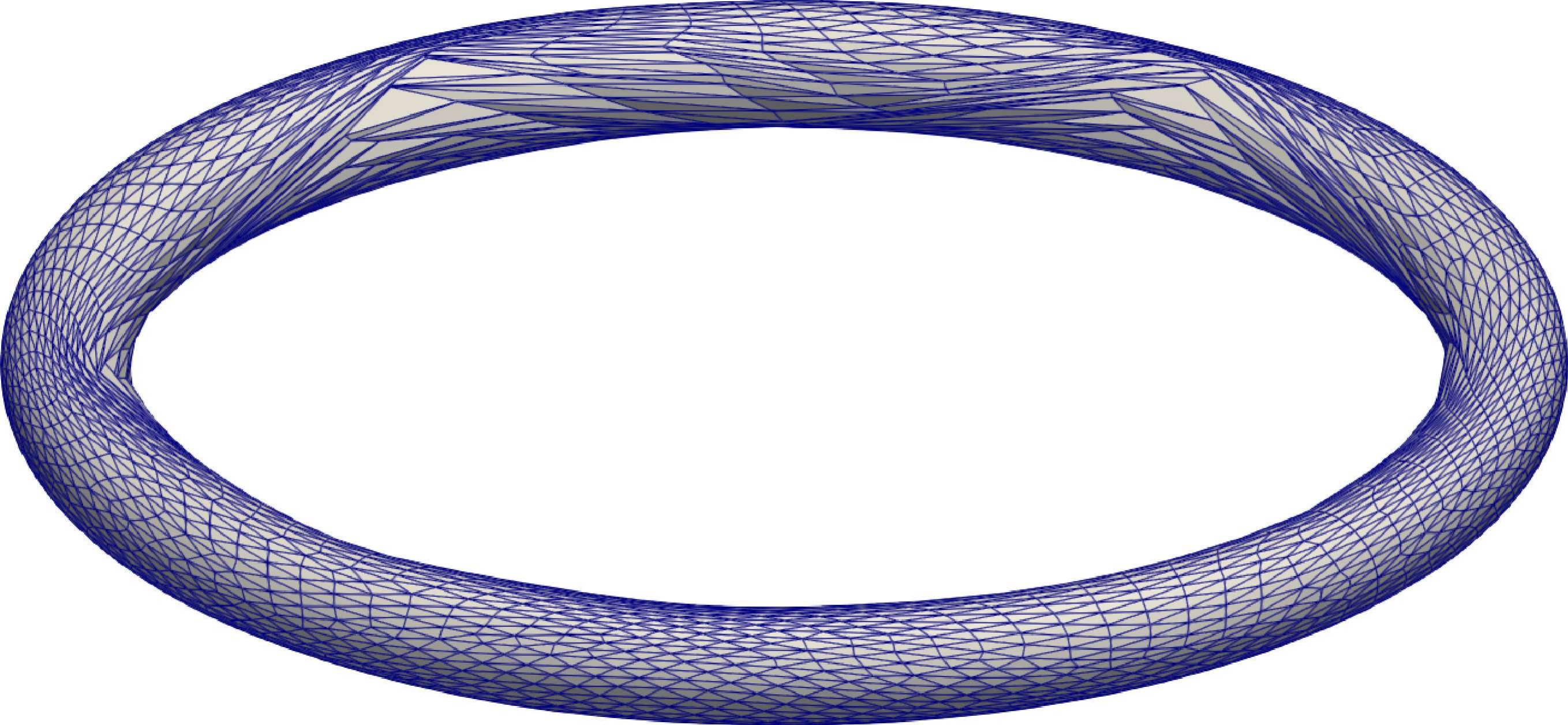}
  }
  \hfill
  \subfloat[ad-BGN, $t=0.2486$]{
    \includegraphics[width=0.30\textwidth]{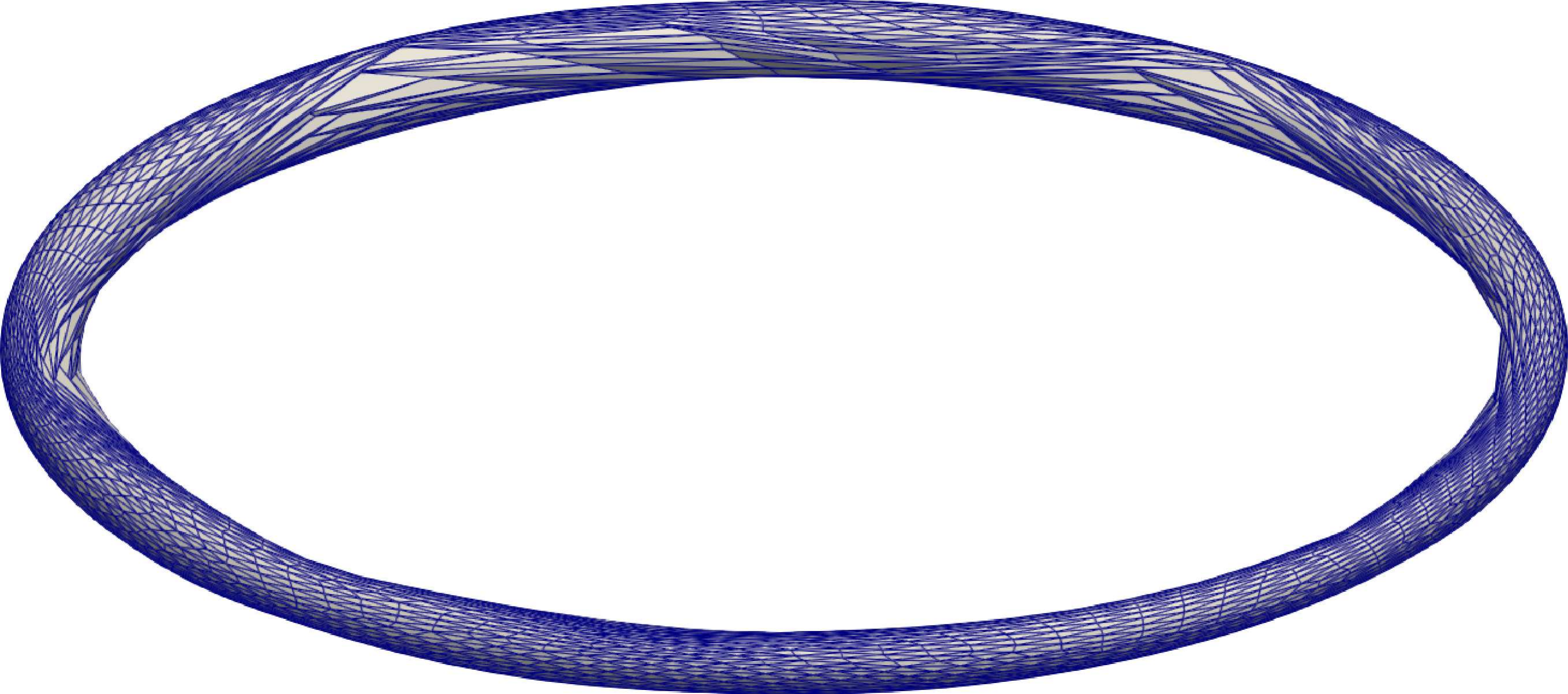}
  }
  \hfill
  \subfloat[ad-BGN, $t=0.2496$]{
    \includegraphics[width=0.30\textwidth]{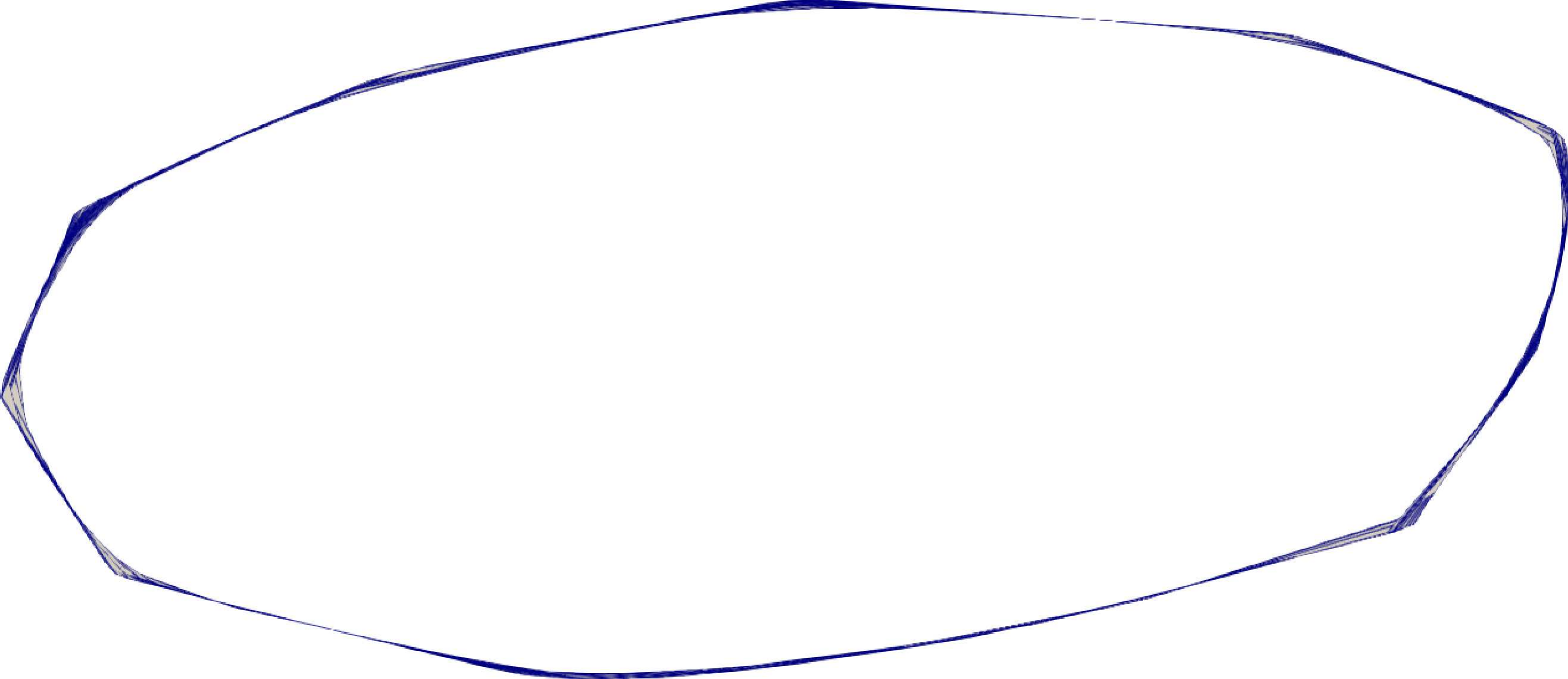}
  }
  \caption{Late-stage evolution in the perturbed torus test.  Top row: MDR.  Bottom row: ad-BGN.}
  \label{fig:exp54-near-zero}
\end{figure}

\subsection{Surface diffusion of an elongated cuboid}

In this example, we consider the evolution of an $8\times1\times1$ cuboid under surface
diffusion. This is a standard benchmark for neck formation and pinch-off. We compare how
the four schemes maintain mesh nondegeneracy using a coarse time step $\tau=10^{-3}$ and
a fine time step $\tau=10^{-4}$, with the spatial discretization fixed at
$(J,K)=(1886,945)$.

\begin{figure}[htbp]
  \centering
  \subfloat[Initial mesh]{
    \includegraphics[width=0.48\textwidth]{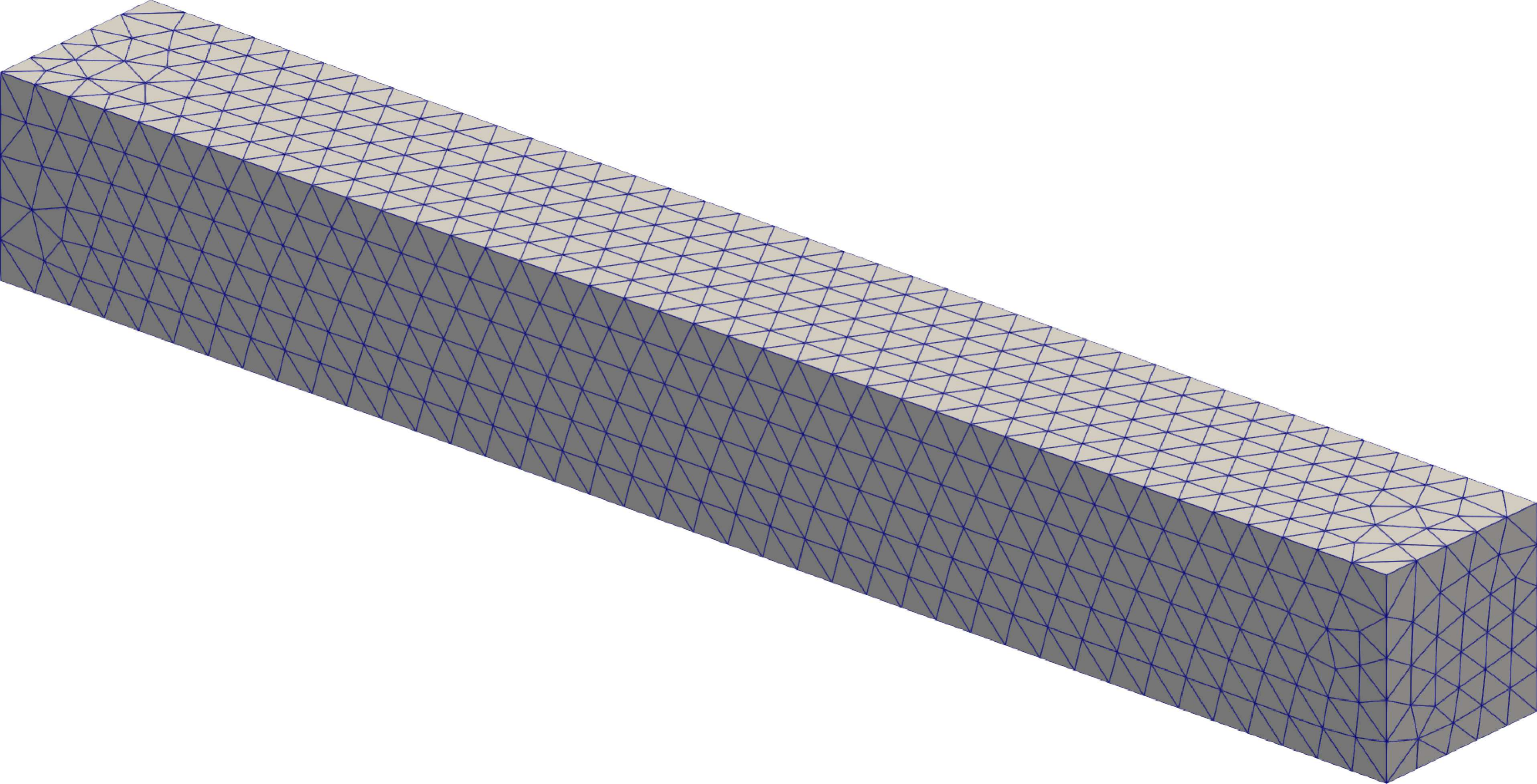}
  }
  \hfill
  \subfloat[Surface area]{
    \includegraphics[width=0.48\textwidth]{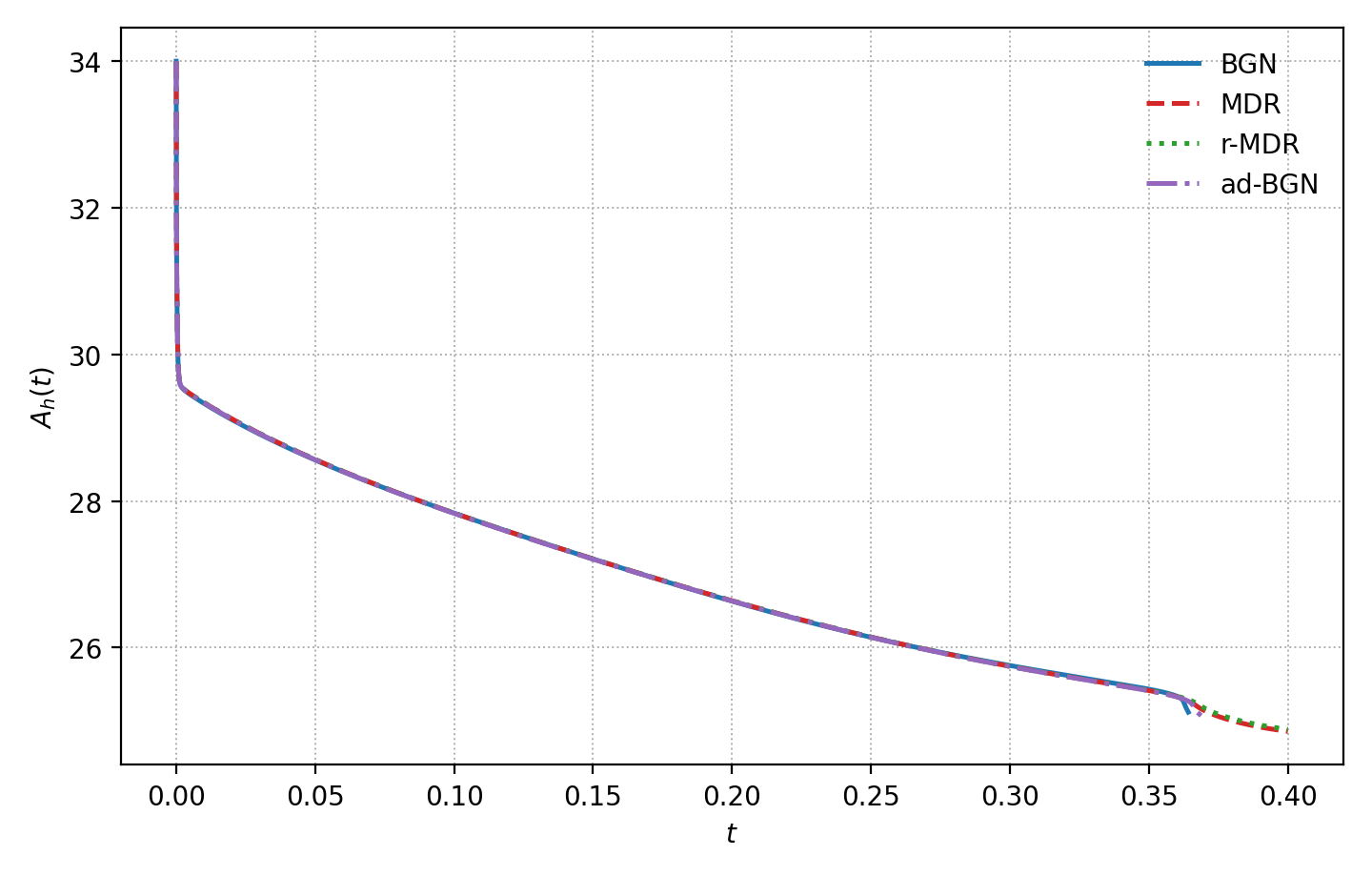}
  }
  \caption{Initial mesh and surface-area decay in the elongated-cuboid test
  under surface diffusion flow.}
  \label{fig:exp55-initial-area}
\end{figure}

The numerical results for the coarse time step are presented in
\Cref{fig:exp55-dt1e3-t0366}. As the cuboid evolves toward pinch-off, the standard BGN
scheme eventually fails because of mesh deterioration at
$t=0.366$ (see also~\cite[Figure 24]{BGN08parametric}).  The polyhedral
surfaces at $t=0.369$ for the other three schemes (ad-BGN, MDR, and r-MDR) are shown for
comparison.

\begin{figure}[htbp]
  \centering
  \subfloat[BGN: $t=0.366$]{
    \includegraphics[width=0.42\textwidth]{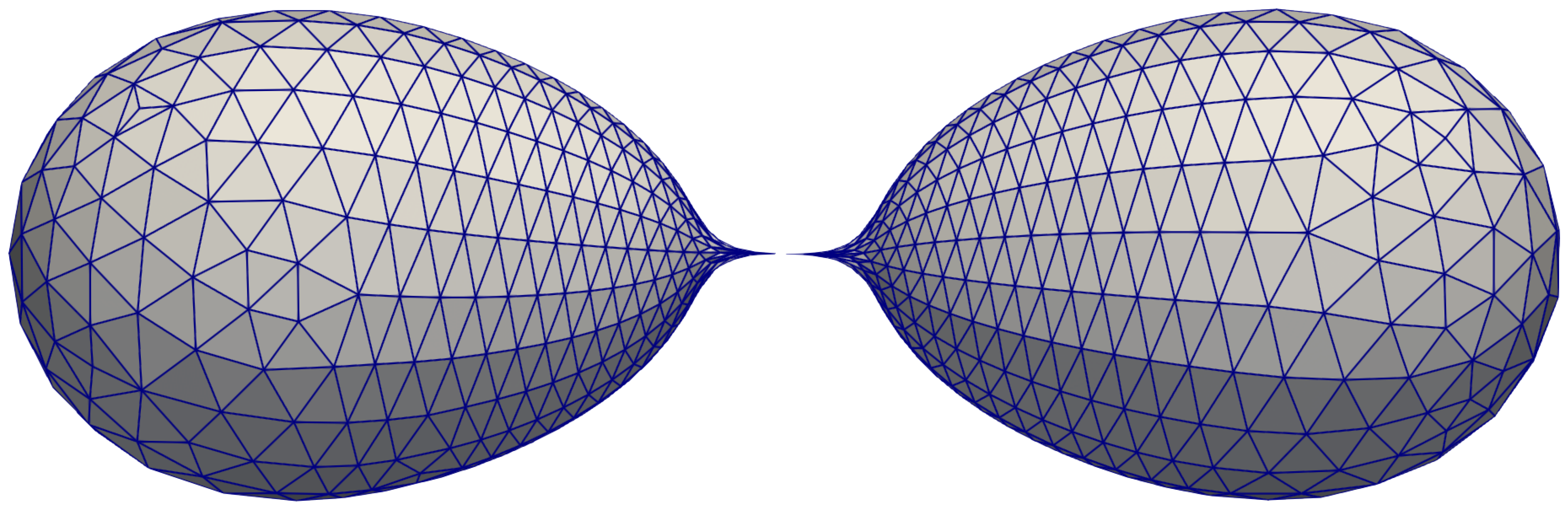}
  }
  \hspace{1cm}
  \subfloat[ad-BGN: $t=0.369$]{
    \includegraphics[width=0.42\textwidth]{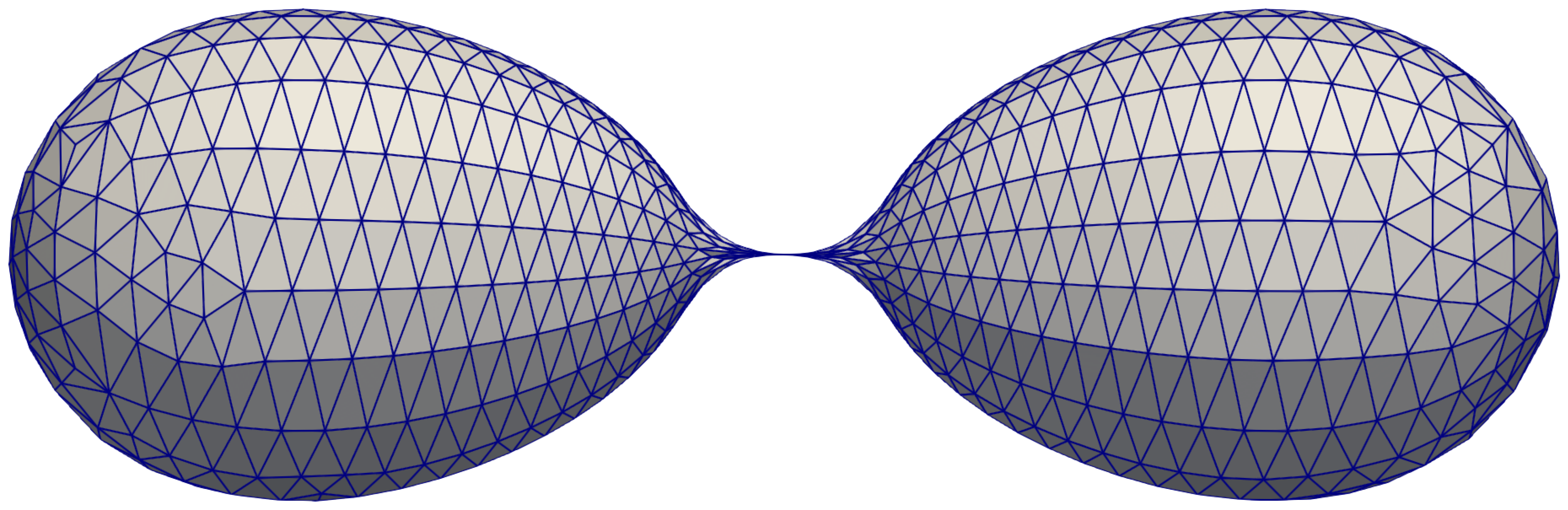}
  }
  \\[2ex]
  \subfloat[MDR: $t=0.369$]{
    \includegraphics[width=0.42\textwidth]{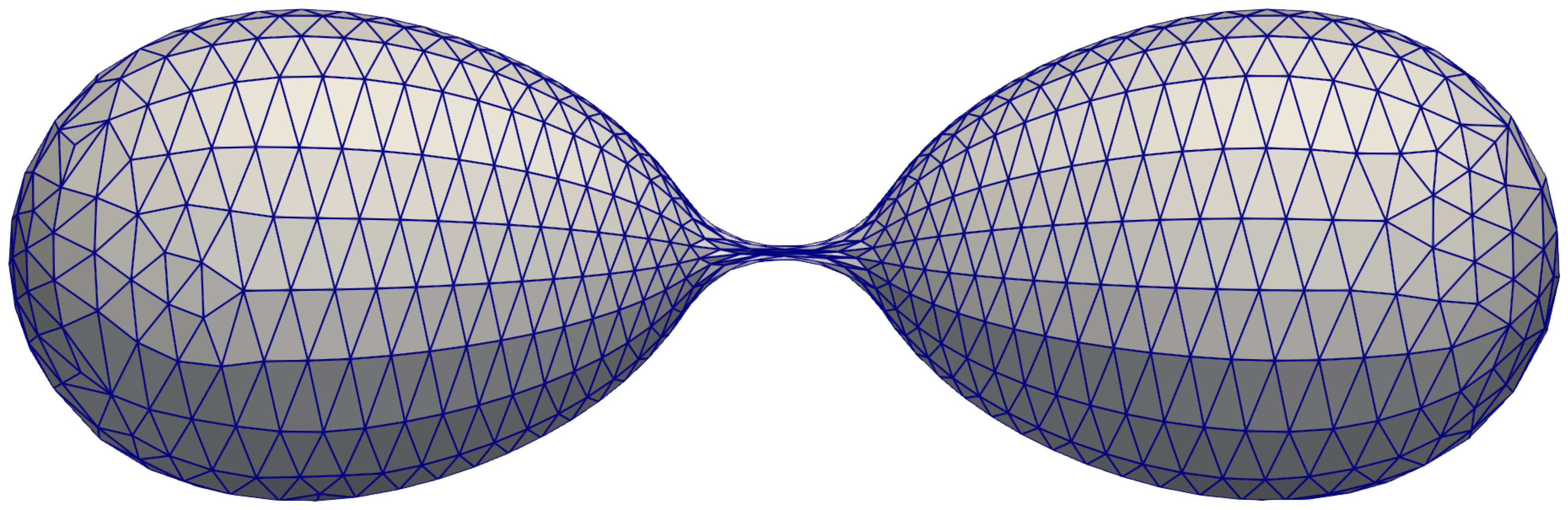}
  }
  \hspace{1cm}
  \subfloat[r-MDR: $t=0.369$]{
    \includegraphics[width=0.42\textwidth]{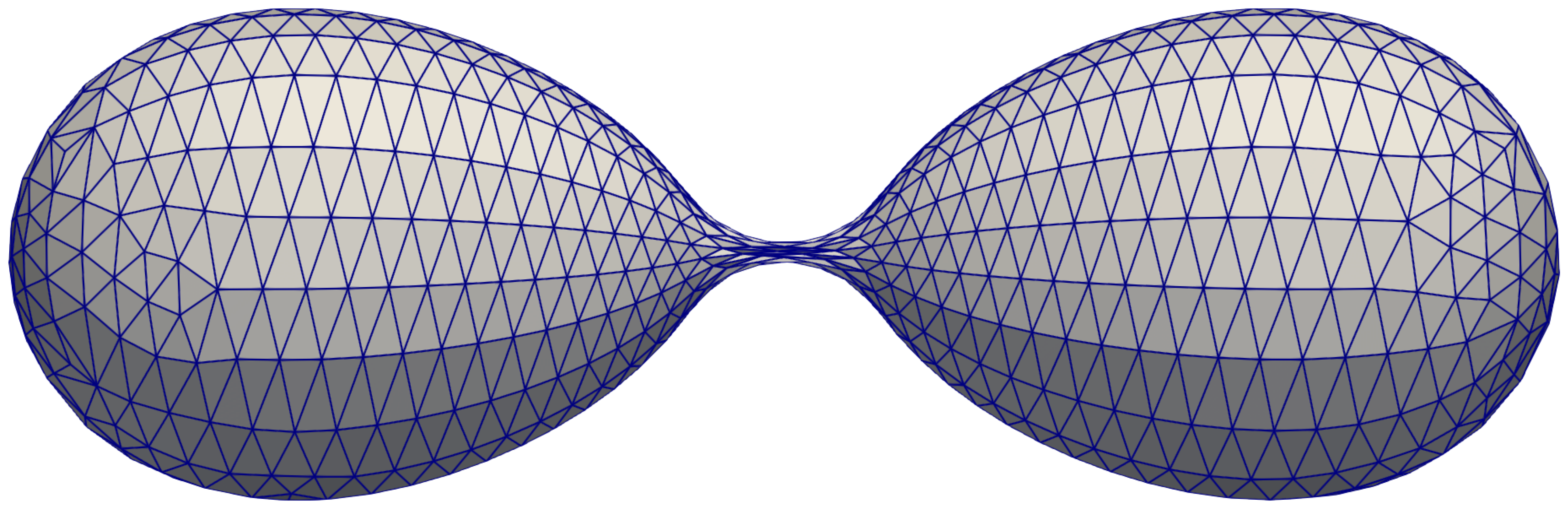}
  }
  \caption{Polyhedral surfaces near the necking regime computed with the
  coarse time step $\tau=10^{-3}$.}
  \label{fig:exp55-dt1e3-t0366}
\end{figure}

Numerical results for the fine time step $\tau=10^{-4}$ are presented in
\Cref{fig:exp55-dt1e4}. Similar to the coarse-time-step computation, BGN stops at
$t=0.3640$ with a severely deteriorated surface mesh, whereas the other three
schemes remain computable with nondegenerate meshes into later stages. The schemes also
exhibit small differences in the apparent necking speed: r-MDR with $\alpha=10$ evolves
slightly more slowly than MDR, although it approaches MDR as $\alpha$ is increased. The
ad-BGN evolution is closer to that of the original BGN scheme. For the fine time-step
computation,~\Cref{tab:exp55-volume-loss} reports the relative enclosed-volume loss for
each scheme.

\begin{figure}[htbp]
  \centering
  \subfloat[BGN, $t=0.3640$]{
    \includegraphics[width=0.42\textwidth]{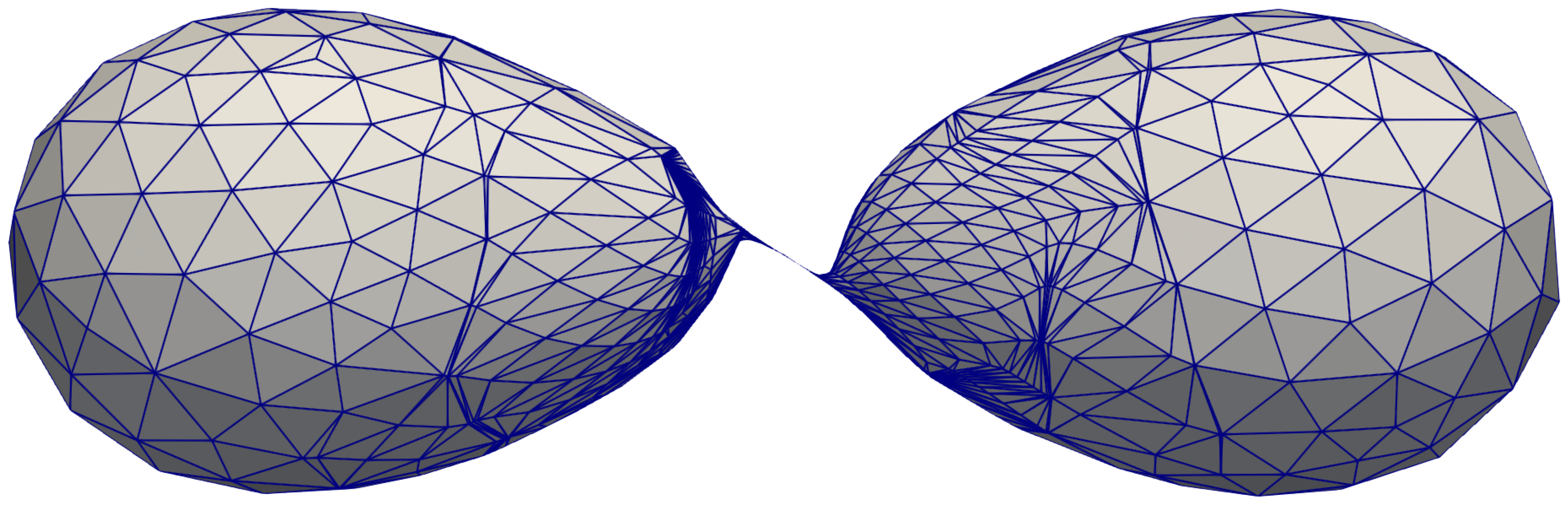}
  }
 \hspace{1cm}
  \subfloat[ad-BGN, $t=0.3690$]{
    \includegraphics[width=0.42\textwidth]{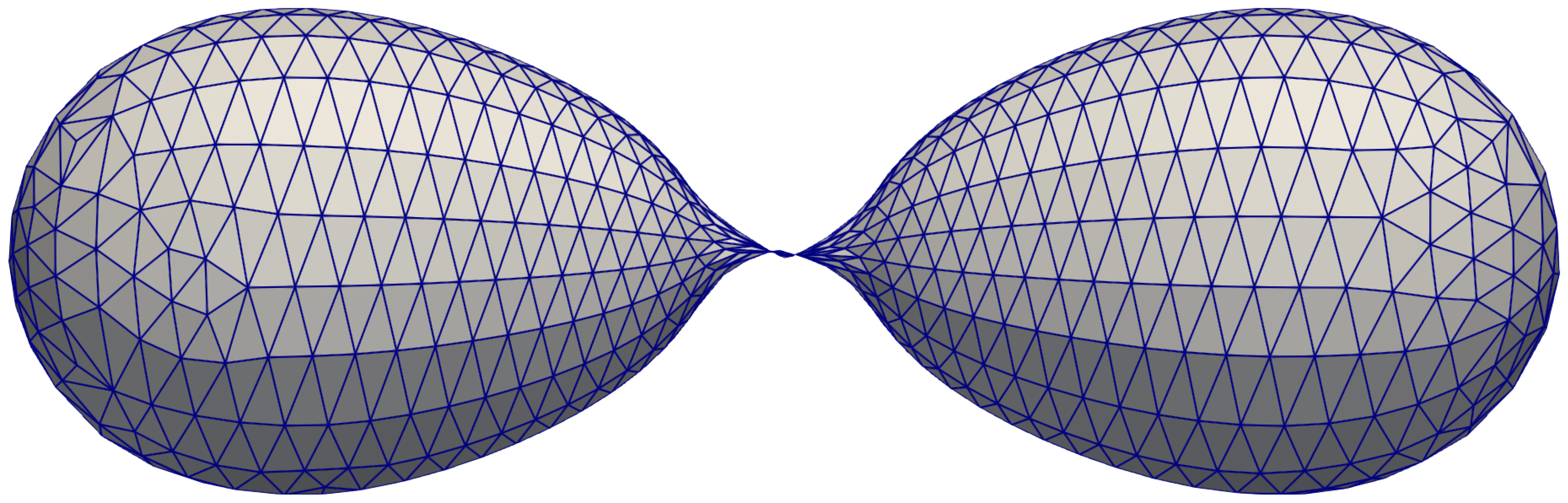}
  }
   \\[2ex]
  \subfloat[MDR, $t=0.3690$]{
    \includegraphics[width=0.42\textwidth]{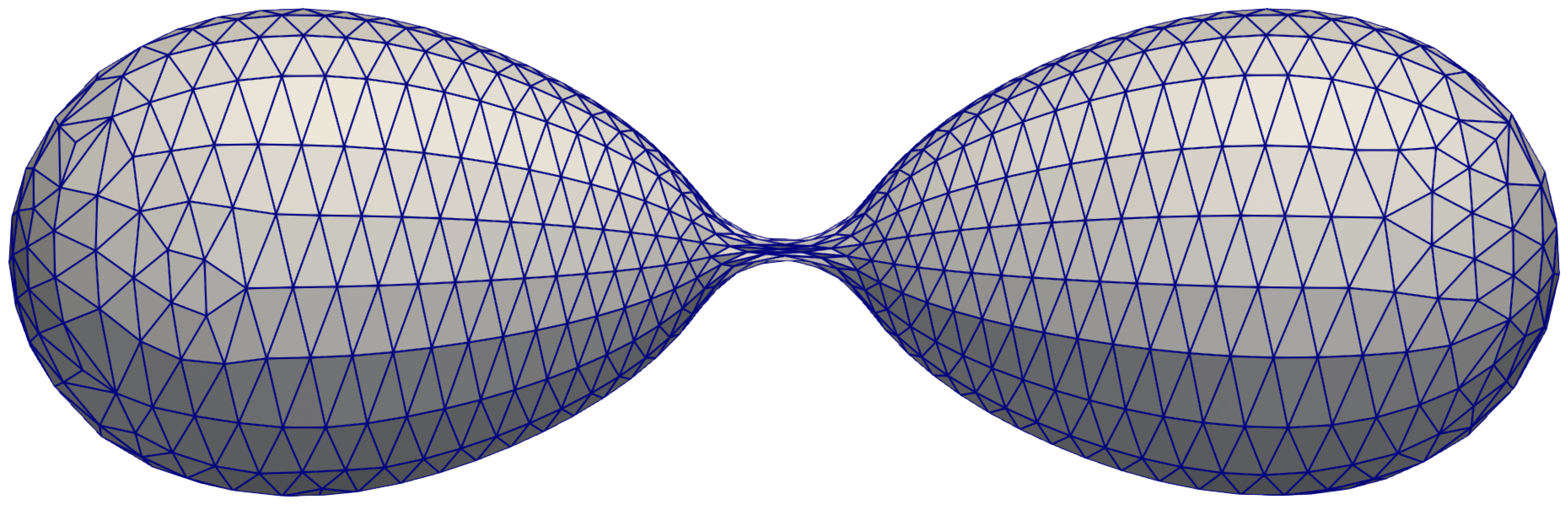}
  }
  \hspace{1cm}
  \subfloat[r-MDR, $t=0.3690$]{
    \includegraphics[width=0.42\textwidth]{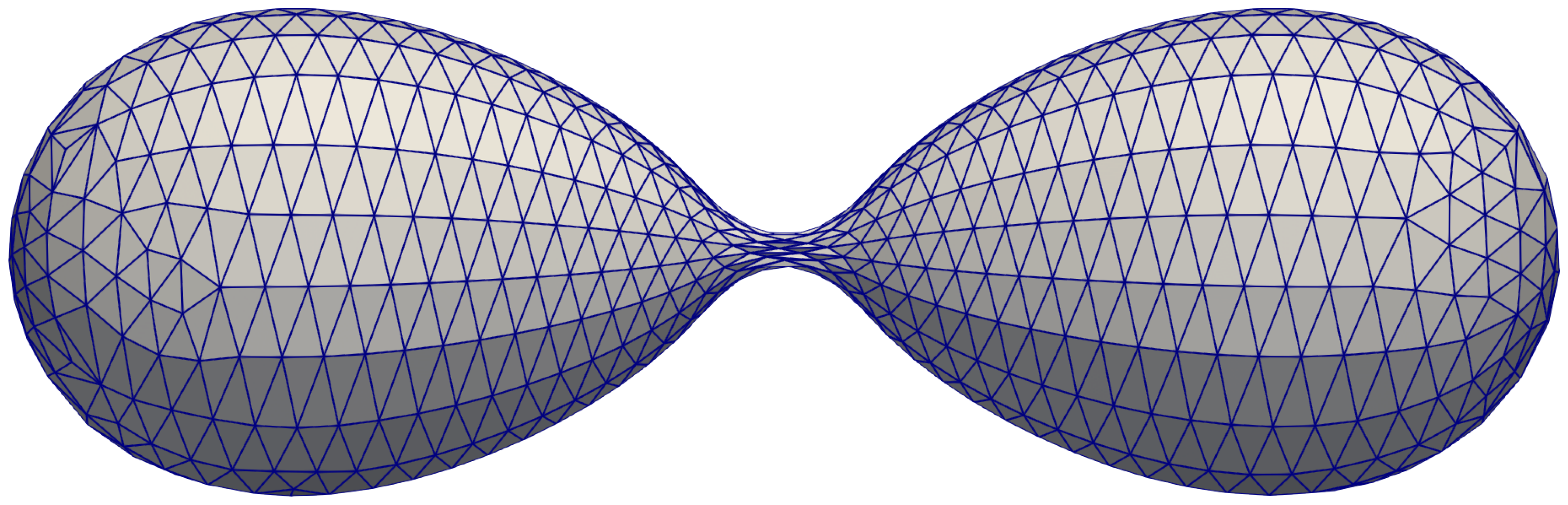}
  }
  \caption{Polyhedral surfaces near the necking regime computed with the
  fine time step $\tau=10^{-4}$.}
  \label{fig:exp55-dt1e4}
\end{figure}

The pinch-off benchmarks show that admissible tangential motions can lead to
substantially different mesh redistributions. In a fixed-connectivity computation, these
redistributions affect how long the computation can be continued near a topological
singularity. Resolving the pinch-off itself is beyond the scope of the present direct
parametric approximations; an actual topological change would require an additional
surgery or remeshing step.

\begin{table}[htbp]
\centering
\small
\caption{The maximum relative volume loss in the elongated-cuboid test
with $\tau=10^{-4}$.}
\label{tab:exp55-volume-loss}
\begin{tabular}{cccc}
\toprule
BGN & ad-BGN & MDR & r-MDR \\
\midrule
$7.8332\mathrm{E}{-3}$ & $7.9832\mathrm{E}{-3}$
& $7.8342\mathrm{E}{-3}$ & $7.8346\mathrm{E}{-3}$ \\
\bottomrule
\end{tabular}
\end{table}

\subsection{Surface diffusion of a cross-shaped surface}

We finally consider the evolution of an initial closed cross-shaped surface under
surface diffusion. Here the surface is given by a unit cuboid with four equal attached
limbs, each a $1\times 3\times 1$ cuboid, see~\Cref{fig:exp56-area-quality}. This
example tests the schemes on a more complex geometry with several concave regions and
thin connecting parts. The discretization parameters are
$(J,K)=(2240,1122)$, and the time step size is fixed at $\tau=10^{-3}$.  The
computation is carried out up to $T=1.5$.

\begin{figure}[htbp]
  \centering
  \subfloat[Initial mesh]{
    \includegraphics[width=0.48\textwidth]{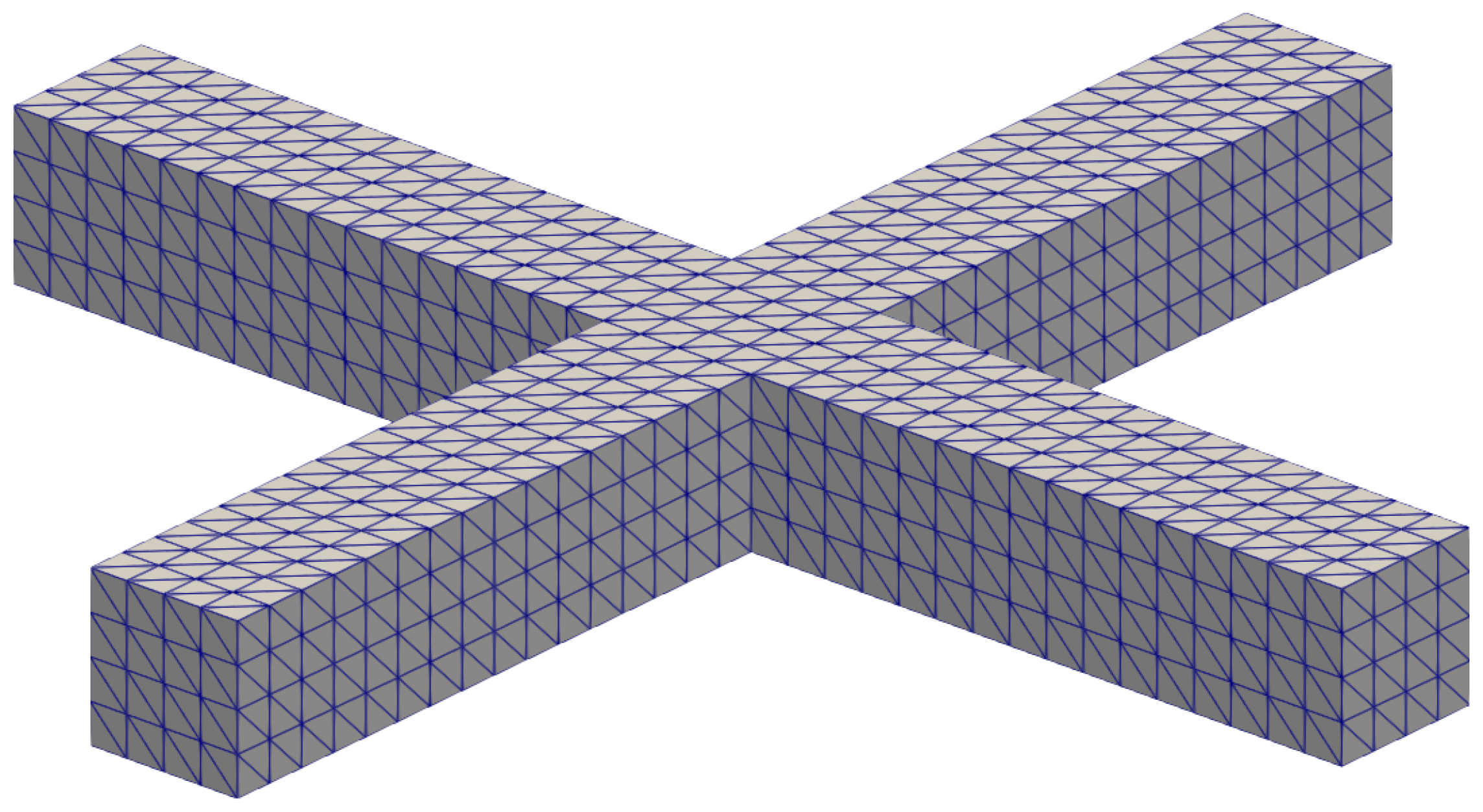}
  }%
  \subfloat[Surface area]{
    \includegraphics[width=0.48\textwidth]{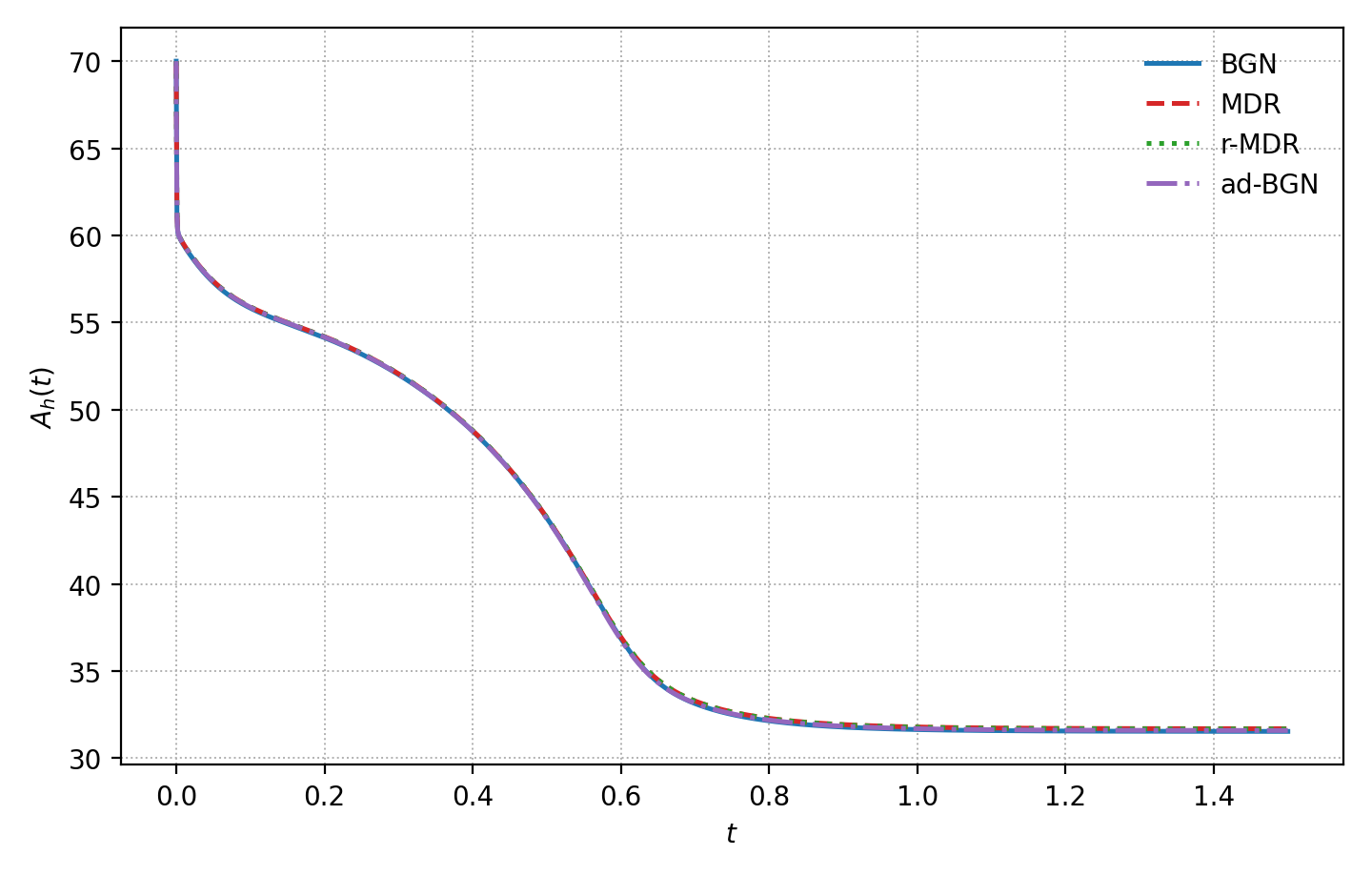}
  }\\[0pt]%
  \subfloat[$\mathsf{r}_{\max}^m$]{
    \includegraphics[width=0.48\textwidth]{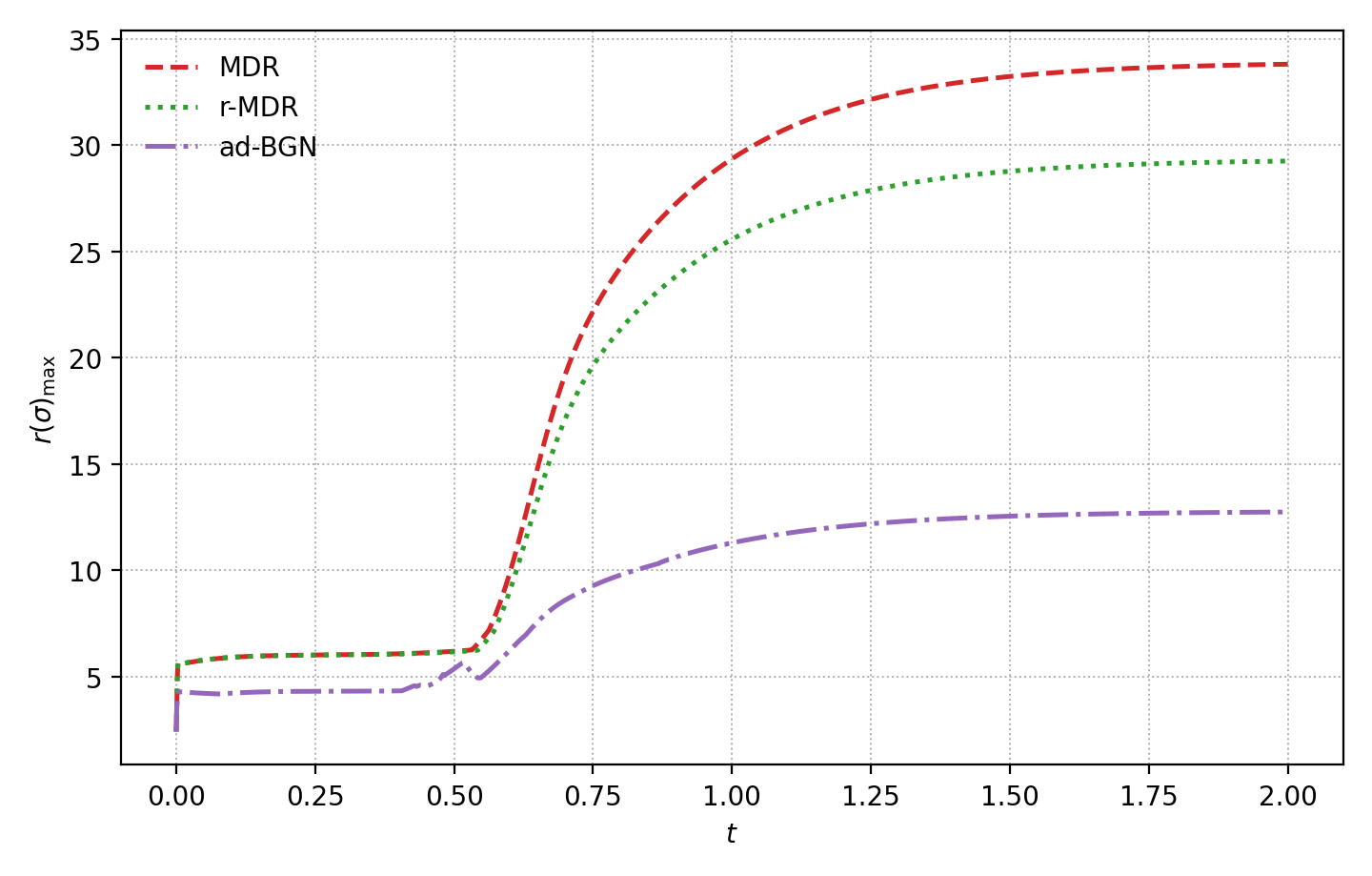}
  }%
  \subfloat[$\mathsf{r}_{\rm p95}^m$]{
    \includegraphics[width=0.48\textwidth]{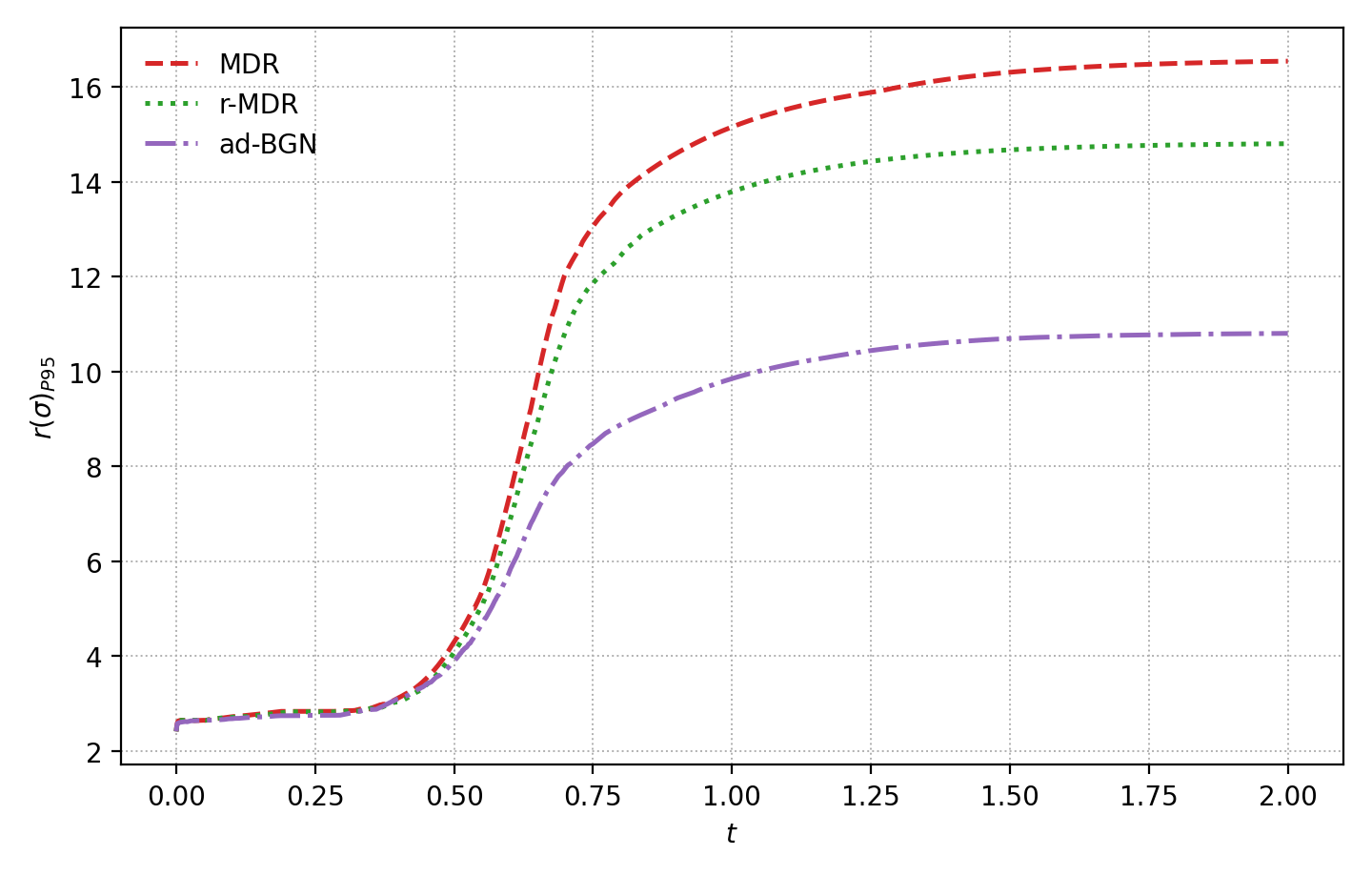}
  }
  \caption{Initial mesh, surface-area decay, $\mathsf{r}_{\max}^m$, and
  $\mathsf{r}_{\rm p95}^m$ in the cross-shaped-surface test.  The mesh-quality
  plots omit BGN, for which $\mathsf{r}_{\max}^m$ reaches 28530 at $t=1.5$.}
  \label{fig:exp56-area-quality}
\end{figure}

\begin{figure}[htbp]
  \centering
  \subfloat[$t=0$]{
    \includegraphics[width=0.25\textwidth]{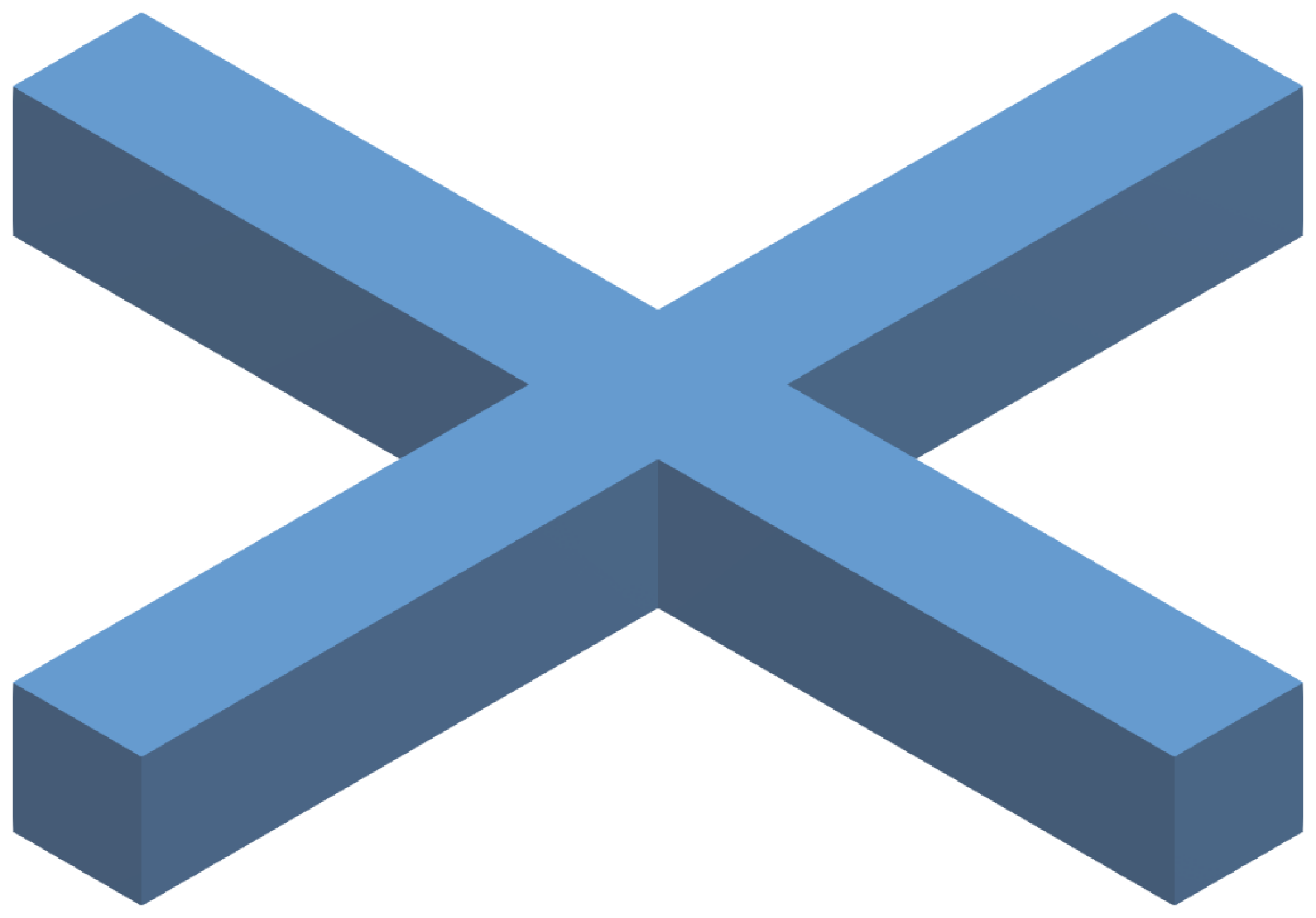}
  }
  \hspace{1.5cm}
  \subfloat[$t=0.2$]{
    \includegraphics[width=0.25\textwidth]{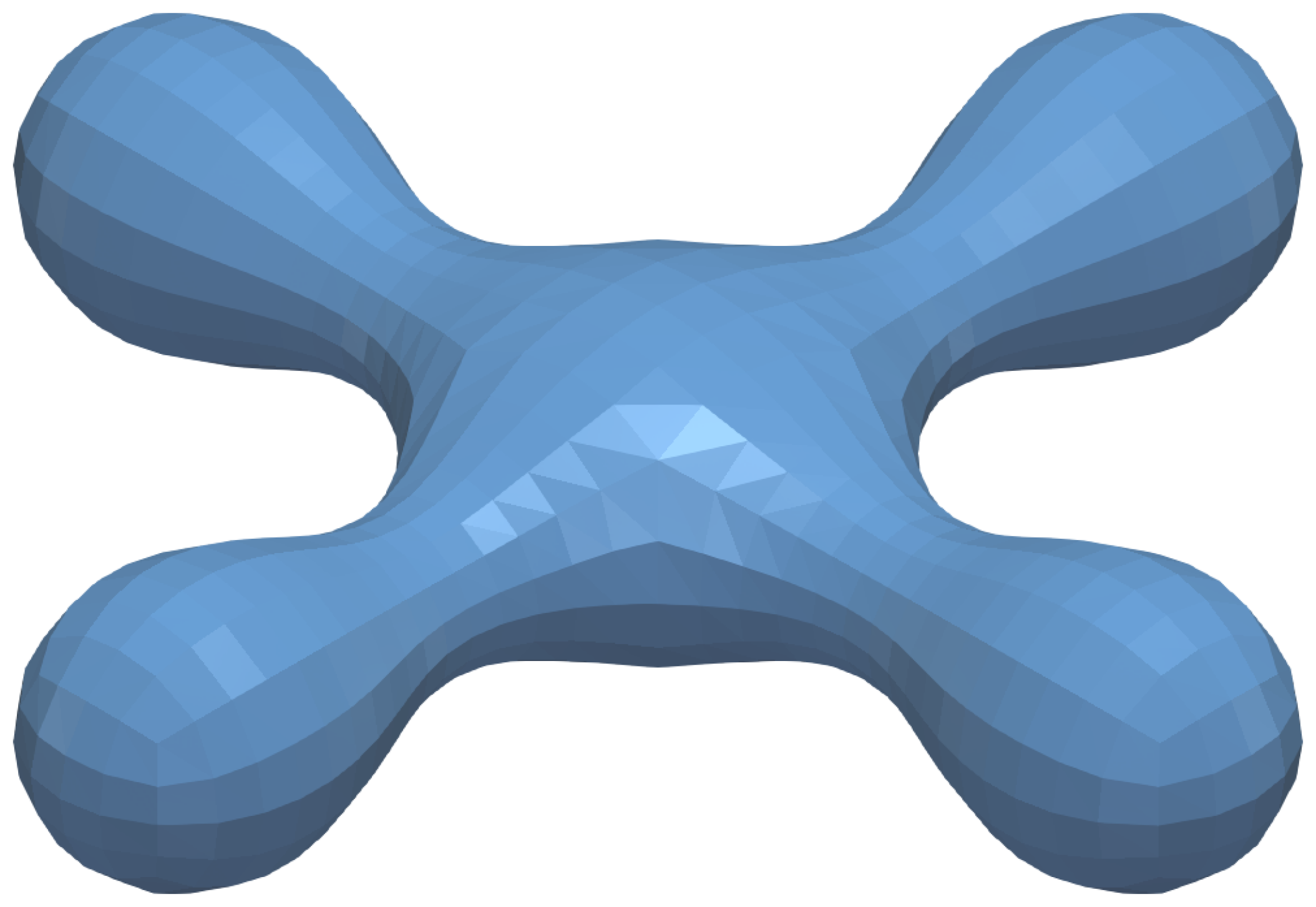}
  }\\[2ex]
  \subfloat[$t=0.6$]{
    \includegraphics[width=0.25\textwidth]{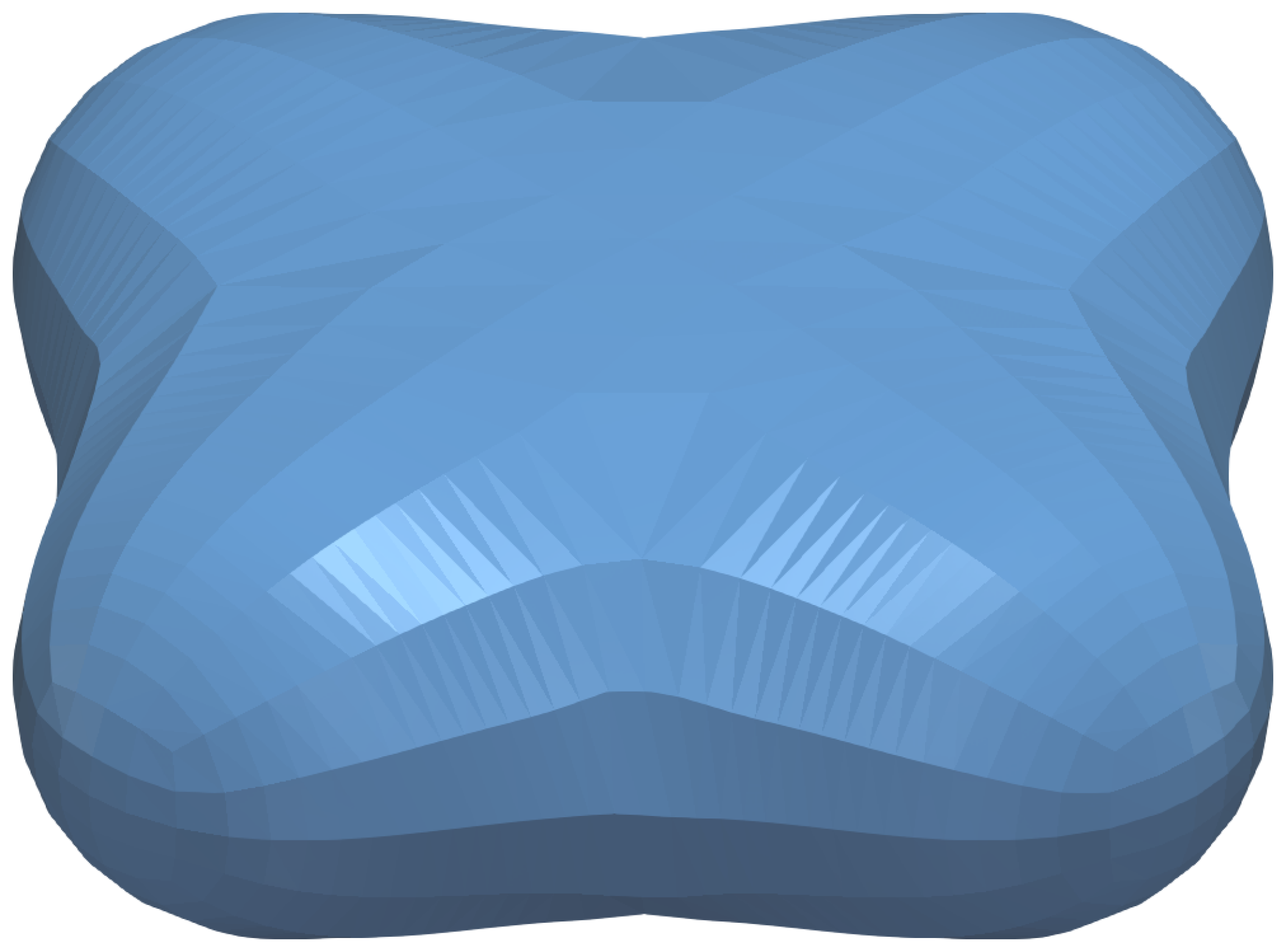}
  }
  \hspace{1.5cm}
  \subfloat[$t=1.5$]{
    \includegraphics[width=0.25\textwidth]{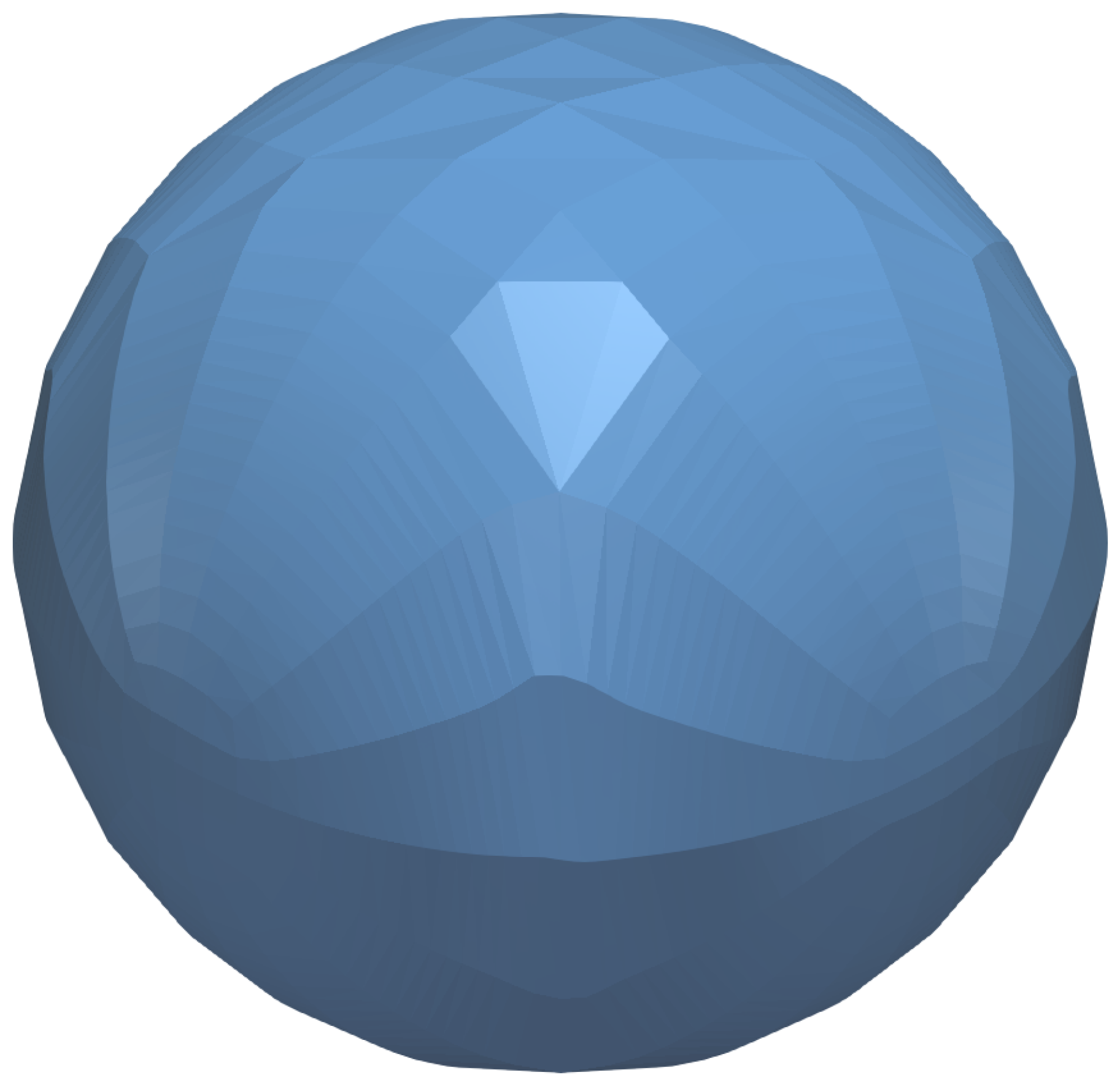}
  }
  \caption{Snapshots in the evolution of the cross-shaped surface.}
  \label{fig:exp56-evolution}
\end{figure}

\begin{table}[!htbp]
\centering
\small
\caption{The maximum relative volume loss in the cross-shaped-surface
test with $\tau=10^{-3}$.}
\label{tab:exp56-volume-loss}
\begin{tabular}{cccc}
\toprule
BGN & ad-BGN & MDR & r-MDR \\
\midrule
$2.3930\mathrm{E}{-2}$ & $2.4138\mathrm{E}{-2}$
& $2.4176\mathrm{E}{-2}$ & $2.4080\mathrm{E}{-2}$ \\
\bottomrule
\end{tabular}
\end{table}

\begin{figure}[!htbp]
  \centering
  \subfloat[BGN]{
    \includegraphics[width=0.29\textwidth]{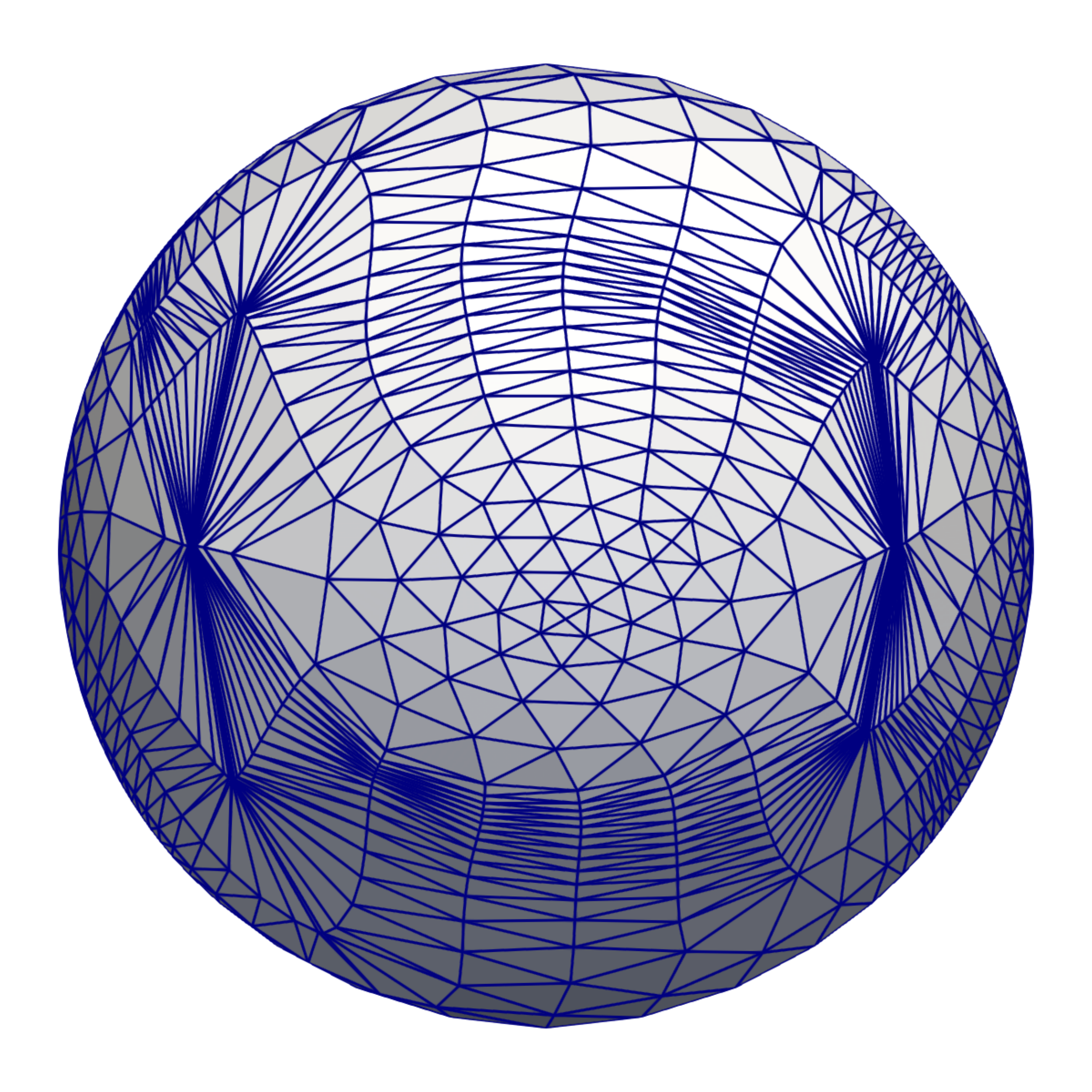}
  }
  \hspace{0.2em}
  \subfloat[ad-BGN]{
    \includegraphics[width=0.29\textwidth]{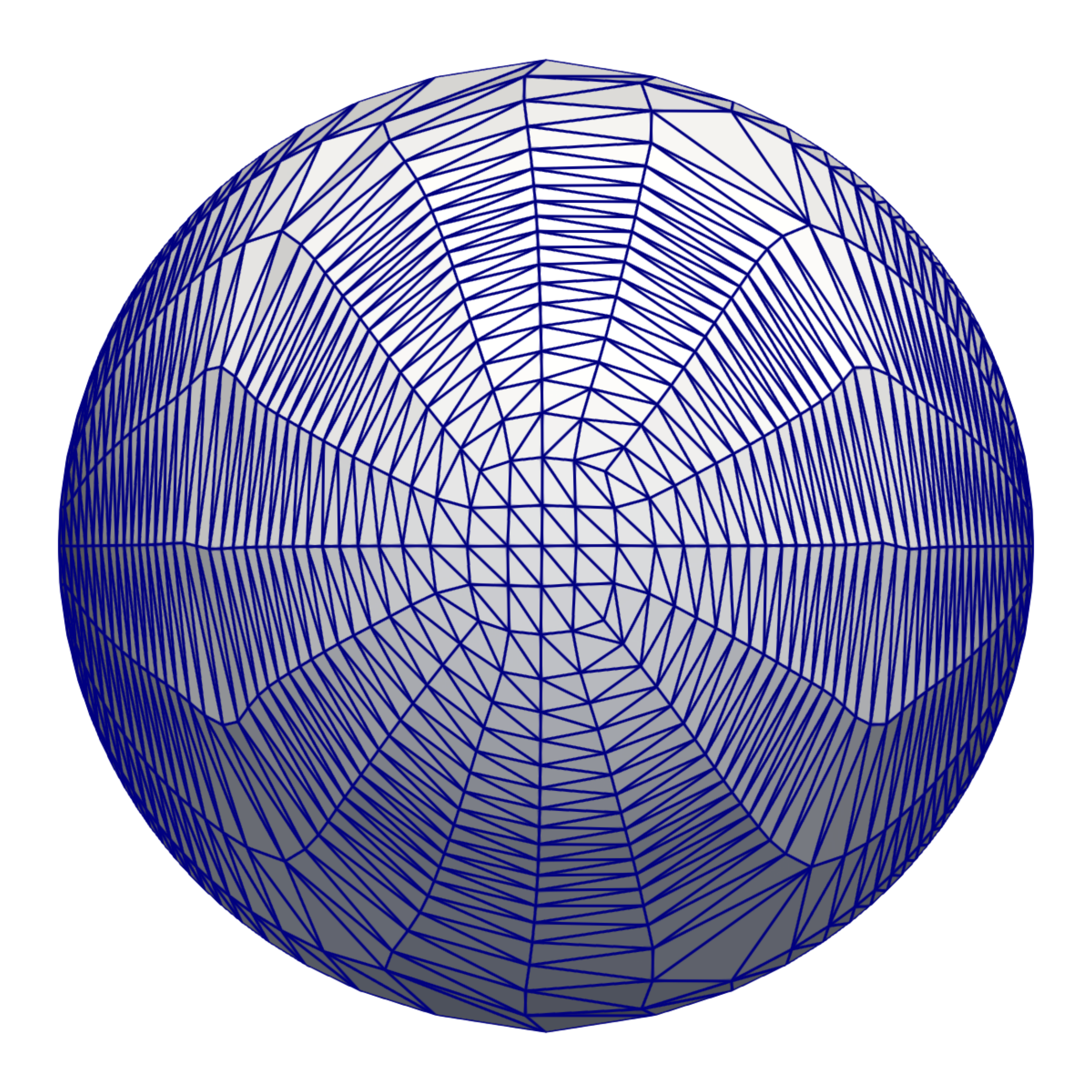}
  }\\[2ex]
  \subfloat[MDR]{
    \includegraphics[width=0.29\textwidth]{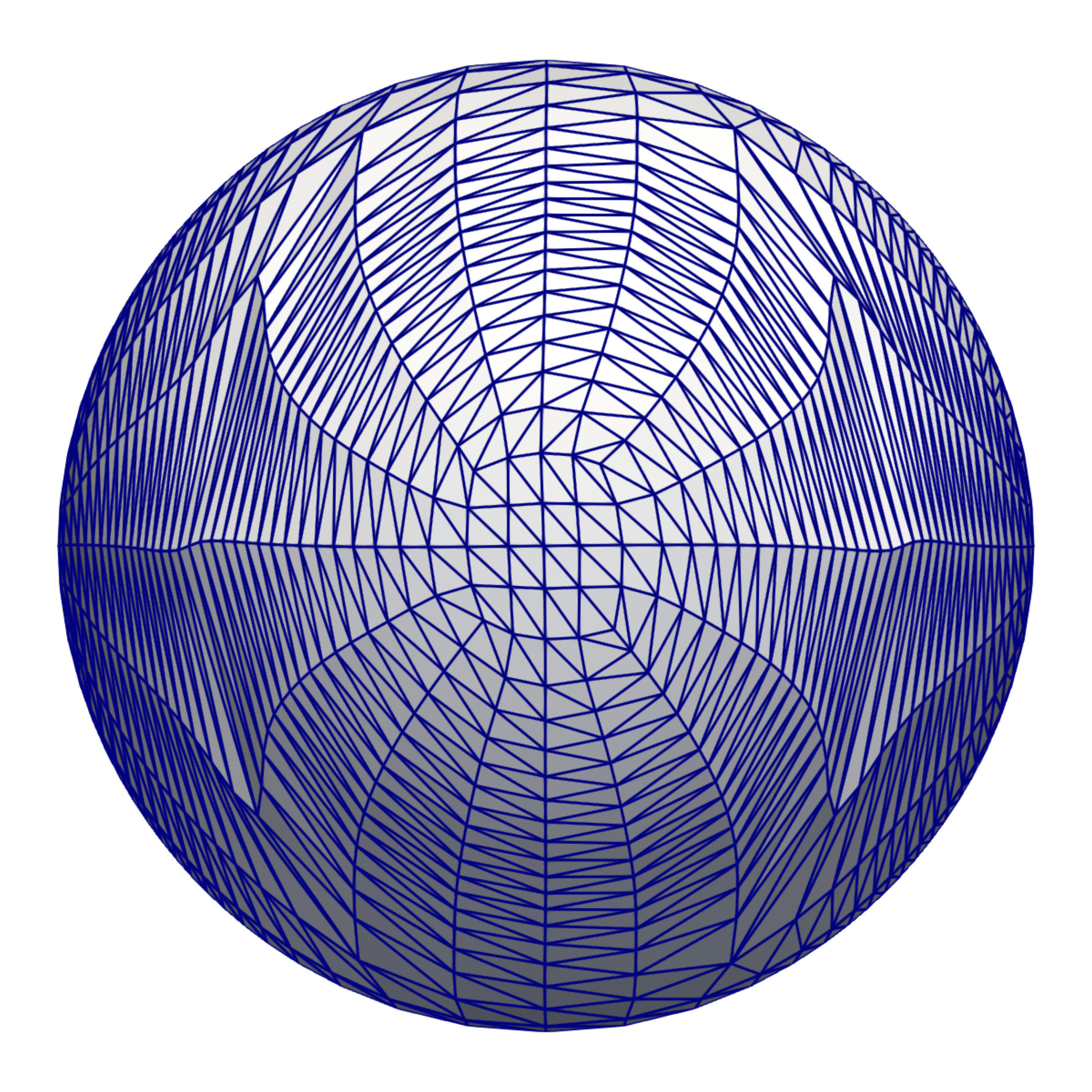}
  }
  \hspace{0.2em}
  \subfloat[r-MDR]{
    \includegraphics[width=0.29\textwidth]{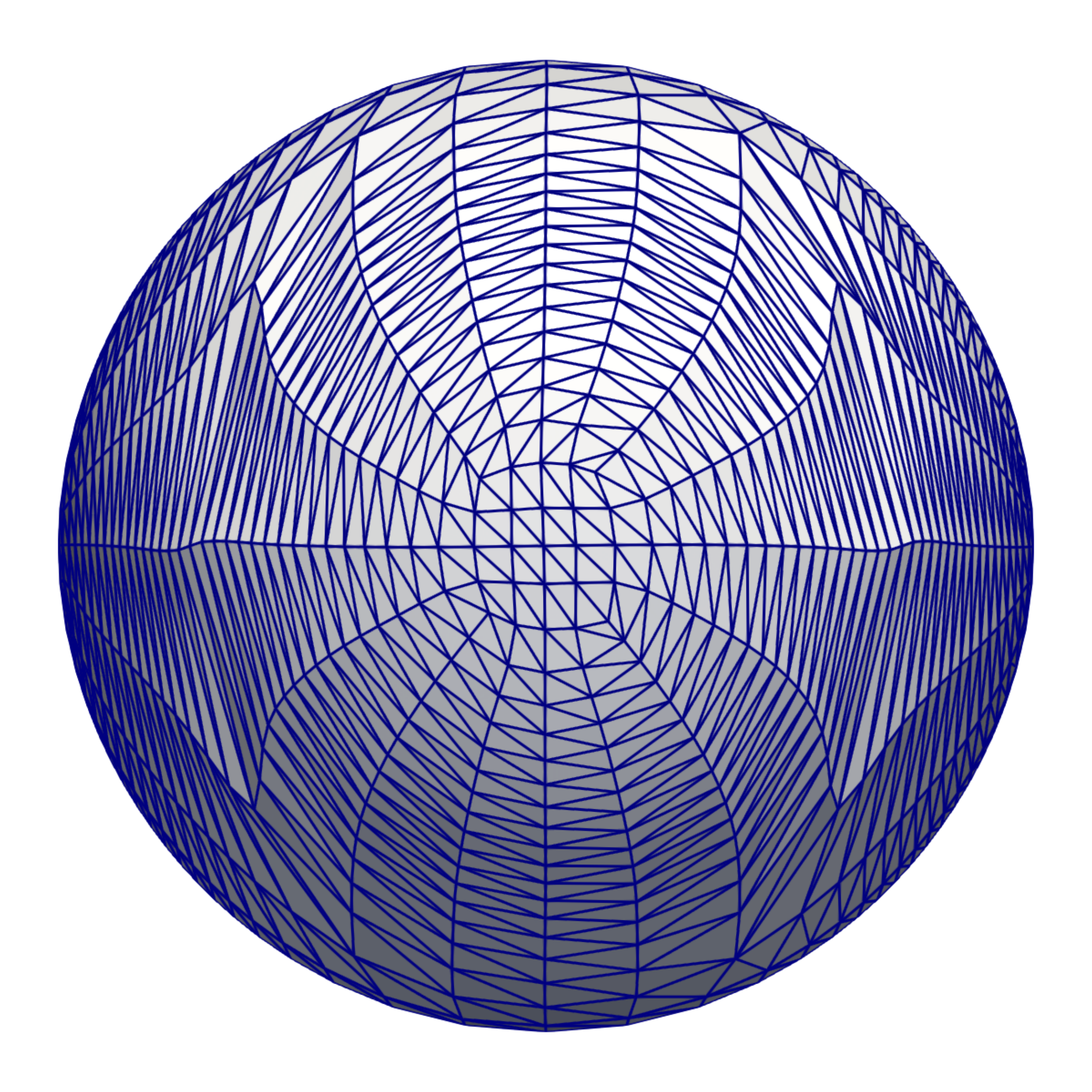}
  }
  \caption{The polyhedral surfaces at $T=1.5$ in the cross-shaped surface test with $\tau=10^{-3}$.}
  \label{fig:exp56-final-mesh}
\end{figure}

As shown in~\Cref{fig:exp56-area-quality}, all four schemes give nearly the same
monotone decay of surface area. The mesh-quality statistics distinguish the tangential
redistributions more clearly: ad-BGN keeps both $\mathsf{r}_{\max}^m$ and
$\mathsf{r}_{\rm p95}^m$ smallest throughout the later stage, while r-MDR remains
better controlled than MDR.~\Cref{tab:exp56-volume-loss} further shows that the
relative enclosed-volume loss remains small for all four schemes.

\Cref{fig:exp56-evolution} shows the evolution of the concave cross-shaped surface
toward the rounded sphere as the steady state. We also visualize the polyhedral surfaces
at the final time $T=1.5$ in~\Cref{fig:exp56-final-mesh}, which is consistent with the
quality statistics: ad-BGN gives the most regular mesh distribution at $T=1.5$, whereas
the BGN mesh contains severe localized deterioration.

\section{Conclusions}
\label{sec:conclusions}

We proposed a minimizing-movement framework for constructing energy-stable parametric
approximations of mean curvature flow and surface diffusion flow. Within this framework,
the metric dissipation determined the geometric normal velocity, while tangential motion
was incorporated through weak constraints. At the variational level, the key requirement
was admissibility: the identity map had to remain feasible so that it could be used as
the comparison map in the energy estimate.

Within this framework, we recovered the classical BGN scheme from an unconstrained
minimization problem and the dual-MDR scheme through the MDR constraint. We further
introduced an admissible variant of the BGN scheme based on vertex normals derived from
the mass-lumped vector curvature, thereby rendering the discrete BGN constraint
admissible. In addition, we introduced a relaxed MDR formulation, which provided additional flexibility in the selection of tangential motion.

The numerical experiments compared the mesh-redistribution properties of the admissible
and classical constraints while confirming the expected energy decay behavior. More
broadly, the framework provided a systematic route for constructing energy-stable
approximations based on other admissible weak constraints. Natural directions for future
work include extensions to anisotropic surface energies, higher-order time
approximations, and further exploration of admissible tangential motions and their
mesh-quality properties.

\section*{Acknowledgements}

This work was partially supported by the National Natural Science Foundation of China
(No. 12401572, Q.Z.) and the Key Project of the National Natural Science Foundation of
China (No. 12494555, Q.Z.).

\bibliographystyle{siam}
\bibliography{bib}

@article{BMN05,
title = {A finite element method for surface diffusion: the parametric case},
author = {Eberhard Bänsch and Pedro Morin and Ricardo H. Nochetto},
journal = {J. Comput. Phys.},
volume = {203},
number = {1},
pages = {321-343},
year = {2005},
issn = {0021-9991},
doi = {https://doi.org/10.1016/j.jcp.2004.08.022},
}

@article{BalzaniR12,
AUTHOR = {Balzani, Nadine and Rumpf, Martin},
TITLE = {A nested variational time discretization for parametric
{W}illmore flow},
JOURNAL = {Interfaces Free Bound.},
FJOURNAL = {Interfaces and Free Boundaries. Mathematical Analysis,
Computation and Applications},
VOLUME = {14},
YEAR = {2012},
NUMBER = {4},
PAGES = {431--454},
ISSN = {1463-9963,1463-9971},
MRCLASS = {53C44 (65N30)},
MRNUMBER = {3016467},
DOI = {10.4171/IFB/287},
URL = {https://doi.org/10.4171/IFB/287},
}

@article{BLani23,
author = {Bao, Weizhu and Li, Yifei},
title = {A Symmetrized Parametric Finite Element Method for Anisotropic Surface Diffusion in Three Dimensions},
journal = {SIAM J. Sci. Comput.},
volume = {45},
number = {4},
pages = {A1438-A1461},
year = {2023},
doi = {10.1137/22M1500575}
}

@article{BGN08Ani,
  title={Numerical approximation of anisotropic geometric evolution equations in the plane},
  author={Barrett, John W and Garcke, Harald and N{\"u}rnberg, Robert},
  journal={IMA J. Numer. Anal.},
  volume={28},
  number={2},
  pages={292--330},
  year={2008},
  publisher={Oxford University Press}
}

@article{BZ21SPFEM,
  title={A structure-preserving parametric finite element method for surface diffusion},
  author={Bao, Weizhu and Zhao, Quan},
  journal={SIAM J. Numer. Anal.},
  volume={59},
  number={5},
  pages={2775--2799},
  year={2021},
  publisher={SIAM}
}

@article{BGN08ani3d,
  title={A variational formulation of anisotropic geometric evolution equations in higher dimensions},
  author={Barrett, John W and Garcke, Harald and N{\"u}rnberg, Robert},
  journal={Numer. Math.},
  volume={109},
  number={1},
  pages={1--44},
  year={2008},
  publisher={Springer}
}

@article{Barrett20,
  title={Parametric finite element approximations of curvature driven interface evolutions},
  author={Barrett, John W and Garcke, Harald and N{\"u}rnberg, Robert},
  volume={21},
  pages={275--423},
  journal={Handb. Numer. Anal. (Andrea Bonito and Ricardo H. Nochetto, eds.)},
  year={2020},
  publisher={Elsevier, Amsterdam}
}

@article{BGN07,
  title={A parametric finite element method for fourth order geometric evolution equations},
  author={Barrett, John W and Garcke, Harald and N{\"u}rnberg, Robert},
  journal={J. Comput. Phys.},
  volume={222},
  number={1},
  pages={441--467},
  year={2007},
  publisher={Elsevier}
}

@article{BGN08parametric,
  title={On the parametric finite element approximation of evolving hypersurfaces in $\mathbb{R}^3$},
  author={Barrett, John W and Garcke, Harald and N{\"u}rnberg, Robert},
  journal={J. Comput. Phys.},
  volume={227},
  number={9},
  pages={4281--4307},
  year={2008},
  publisher={Elsevier}
}

@article{BGN08willmore,
  title={Parametric approximation of {W}illmore flow and related geometric evolution equations},
  author={Barrett, John W and Garcke, Harald and N{\"u}rnberg, Robert},
  journal={SIAM J. Sci. Comput.},
  volume={31},
  number={1},
  pages={225--253},
  year={2008},
  publisher={SIAM}
}

@article{GNZ25willmore,
author = {Garcke, Harald and N{\"u}rnberg, Robert and Zhao, Quan},
title = {Stable fully discrete finite element methods with {BGN} tangential motion for {Willmore} flow of planar curves},
JOURNAL = {J. Sci. Comput.},
volume = 105,
number = 2,
pages = {45},
numpages = {23},
year = 2025,
doi = {10.1007/s10915-025-03068-9},
url = {https://doi.org/10.1007/s10915-025-03068-9}
}

@article{cahn1994,
  title={Surface motion by surface diffusion},
  author={Cahn, John W and Taylor, Jean E},
  journal={Acta Metall. Mater.},
  volume={42},
  number={4},
  pages={1045--1063},
  year={1994},
  publisher={Elsevier}
}

@article{Davis04,
AUTHOR = {Davis, Timothy A.},
TITLE = {Algorithm 832: {UMFPACK} {V}4.3---an unsymmetric-pattern
multifrontal method},
JOURNAL = {ACM Trans. Math. Software},
FJOURNAL = {Association for Computing Machinery. Transactions on
Mathematical Software},
VOLUME = {30},
YEAR = {2004},
NUMBER = {2},
PAGES = {196--199},
ISSN = {0098-3500},
CODEN = {ACMSCU},
MRCLASS = {65F50 (65-04 65F05)},
MRNUMBER = {2075981},
DOI = {10.1145/992200.992206},
URL = {https://doi.org/10.1145/992200.992206},
}

@article{Deckelnick2005,
  title={Computation of geometric partial differential equations and mean curvature flow},
  author={Deckelnick, Klaus and Dziuk, Gerhard and Elliott, Charles M},
  journal={Acta Numer.},
  volume={14},
  pages={139--232},
  year={2005},
  publisher={Cambridge University Press}
}

@Article{Dziuk91,
  author = {Dziuk, Gerhard},
  title = {An algorithm for evolutionary surfaces},
  journal = {Numer. Math.},
  fjournal = {Numerische Mathematik},
  volume = 58,
  year = 1990,
  number = 6,
  pages = {603--611},
  issn = {0029-599X},
  coden = {NUMMA7},
  mrclass = {65D17 (53A10 65N30)},
  mrnumber = {1083523},
  doi = {10.1007/BF01385643},
  URL = {https://doi.org/10.1007/BF01385643}
}

@article{Duan24new,
  title={New artificial tangential motions for parametric finite element approximation of surface evolution},
  author={Duan, Beiping and Li, Buyang},
  journal={SIAM J. Sci. Comput.},
  volume={46},
  number={1},
  pages={A587--A608},
  year={2024},
  publisher={SIAM}
}

@article{Duan25,
author = {Duan, Beiping},
title = {Mesh-Preserving and Energy-Stable Parametric {FEM} for Geometric Flows of Surfaces},
journal = {SIAM J. Numer. Anal.},
volume = {63},
number = {2},
pages = {619-640},
year = {2025},
doi = {10.1137/24M1671542}
}

@article{Gao26,
author = {Gao, Guangwei and Garcke, Harald and Li, Buyang and Tang, Rong},
title = {An Energy-Stable Minimal Deformation Rate Scheme for Mean Curvature Flow and Surface Diffusion},
journal = {SIAM J. Sci. Comput. },
volume = {48},
number = {1},
pages = {A103-A131},
year = {2026},
doi = {10.1137/25M1753838},
}

@misc{Gao26dualMDR,
      title={Dual formulations of geometric curvature flows and their discretizations}, 
      author={Guangwei Gao and Buyang Li and Rong Tang},
      year={2026},
      howpublished ={arXiv: 2604.18288},
      archivePrefix={arXiv},
      primaryClass={math.NA},
      url={https://arxiv.org/abs/2604.18288}, 
}

@article{GJSZ25,
author = {Garcke, Harald and Jiang, Wei and Su, Chunmei and Zhang, Ganghui},
title = {Structure-Preserving Parametric Finite Element Method for Surface Diffusion Based on {Lagrange} Multiplier Approaches},
journal = {SIAM J. Sci. Comput.},
volume = {47},
number = {3},
pages = {A1983-A2011},
year = {2025},
doi = {10.1137/24M1687546},    
}

@article{GNZ26,
author = {Garcke, Harald and N\"{u}rnberg, Robert and Zhao, Quan},
title = {An Energy-Stable Parametric Finite Element Method for {Willmore} Flow with Normal-Tangential Velocity Splitting},
journal = {SIAM J. Sci. Comput.},
volume = {48},
number = {3},
pages = {A1235-A1259},
year = {2026},
doi = {10.1137/25M1773878}
}

@article{Hu22evolving,
  title={Evolving finite element methods with an artificial tangential velocity for mean curvature flow and {W}illmore flow},
  author={Hu, Jiashun and Li, Buyang},
  journal={Numer. Math.},
  volume={152},
  number={1},
  pages={127--181},
  year={2022},
  publisher={Springer}
}

@article{Helfrich73elastic,
  title={Elastic properties of lipid bilayers: theory and possible experiments},
  author={Helfrich, Wolfgang},
  journal={Z. Naturforsch C},
  volume={28},
  number={11-12},
  pages={693--703},
  year={1973},
  publisher={Verlag der Zeitschrift f\"ur Naturforschung}
}

@article{Kemmochi25structure,
  title={Structure-preserving numerical methods for constrained gradient flows of planar closed curves with explicit tangential velocities},
  author={Kemmochi, Tomoya and Miyatake, Yuto and Sakakibara, Koya},
  journal={Japan J. Ind. Appl. Math.},
  volume={42},
  number={2},
  pages={575--603},
  year={2025},
  publisher={Springer}
}

@article{LX25,
author = {Yihe Liu and Xianmin Xu},
title = {A variational discretization method for mean curvature flows by the {Onsager} principle},
volume={6}, 
url={https://www.global-sci.com/csiam-am/article/view/7878}, 
DOI={10.4208/csiam-am.SO-2024-0005}, 
number={1}, 
journal={CSIAM Trans. Appl. Math. }, 
year={2025}, 
month={Feb.}, 
pages={63--95} 
}

@article{DeTurck17,
  title={On approximations of the curve shortening flow and of the mean curvature flow based on the {DeTurck} trick},
  author={M. Elliott, Charles and Fritz, Hans},
  journal={IMA J. Numer. Anal. },
  volume={37},
  number={2},
  pages={543--603},
  year={2017},
  publisher={Oxford University Press}
}

@article{Mullins57,
  author={W. W. Mullins},
  title={Theory of thermal grooving},
  journal={J. Appl. Phys.},
  volume={28},
  number={3},
  pages={333--339},
  year={1957},
  publisher={American Institute of Physics}
}

@article{Mikula14,
author = {Mikula, Karol and Reme\v{s}\'{\i}kov\'{a}, Mariana and Sarkoci, Peter and \v{S}ev\v{c}ovi\v{c}, Daniel},
title = {Manifold Evolution with Tangential Redistribution of Points},
journal = {SIAM J. Sci. Comput. },
volume = {36},
number = {4},
pages = {A1384-A1414},
year = {2014},
doi = {10.1137/130927668},
}

@InProceedings{OlischlagerR09,
author={Olischl{\"a}ger, N. and Rumpf, M.},
editor={Hancock, Edwin R. and Martin, Ralph R. and Sabin, Malcolm A.},
title={Two Step Time Discretization of {W}illmore Flow},
booktitle={Mathematics of Surfaces XIII},
year=2009,
publisher={Springer},
address={Berlin},
pages={278--292},
}

@article{PAN26,
title = {When surface evolution meets {Fokker-Planck} equation: A novel tangential velocity model for uniform parametrization},
author = {Jiangong Pan and Guozhi Dong and Hailong Guo and Zuoqiang Shi},
journal = {J. Comput. Phys.},
volume = {549},
pages = {114604},
year = {2026},
issn = {0021-9991},
doi = {https://doi.org/10.1016/j.jcp.2025.114604}
}

@article{Remacle10,
author = {Remacle, J.-F. and Geuzaine, C. and Compère, G. and Marchandise, E.},
title = {High-quality surface remeshing using harmonic maps},
journal = {Int. J. Numer. Methods Eng. },
volume = {83},
number = {4},
pages = {403-425},
doi = {https://doi.org/10.1002/nme.2824},
year = {2010}
}

@Misc{RumpfSS25preprint,
author = {Martin Rumpf and Josua Sassen and Christoph Smoch},
title = {A hybrid minimizing movement and neural network approach to {W}illmore 
flow},
year = 2025,
howpublished = {arXiv:2502.14656},
url = {https://arxiv.org/abs/2502.14656}
}

@article{schoberl2014c++,
  title={C++ 11 implementation of finite elements in {NGSolve}},
  author={Sch{\"o}berl, Joachim},
  journal={Institute for analysis and scientific computing, Vienna University of Technology},
  volume={30},
  year={2014}
}

@article{Thompson12solid,
  title={Solid-state dewetting of thin films},
  author={Thompson, Carl V},
  journal={Annu. Rev. Mater. Res.},
  volume={42},
  pages={399--434},
  year={2012},
  publisher={Annual Reviews}
}

@article{Zhao2021energy,
  title={An energy-stable parametric finite element method for simulating solid-state dewetting},
  author={Zhao, Quan and Jiang, Wei and Bao, Weizhu},
  journal={IMA J. Numer. Anal.},
  volume={41},
  number={3},
  pages={2026--2055},
  year={2021},
  publisher={Oxford University Press}
}

\end{document}